\documentclass[letterpaper,11pt]{article}
\usepackage{color}

\usepackage{theorem,amsmath,amssymb,amscd,mathrsfs,verbatim}
\usepackage[alphabetic]{amsrefs}
\usepackage[hyperfootnotes=false, colorlinks, linkcolor={blue},
citecolor={magenta}, filecolor={blue}, urlcolor={blue}]{hyperref}
\theoremstyle{plain}
\allowdisplaybreaks
\newcommand{\A}{{\mathbb A}}
\newcommand{\Q}{{\mathbb Q}}
\newcommand{\Z}{{\mathbb Z}}
\newcommand{\R}{{\mathbb R}}
\newcommand{\C}{{\mathbb C}}
\newcommand{\bs}{\backslash}

\newcommand{\p}{\mathfrak p}
\newcommand{\OF}{{\mathfrak o}}
\newcommand{\GL}{{\rm GL}}

\newcommand{\SL}{{\rm SL}}

\newcommand{\sgn}{{\rm sgn}}
\newcommand{\St}{{\rm St}}

\newcommand{\Mat}{{\rm M}}

\renewcommand{\P}{\mathfrak{P}}

\newcommand{\Hom}{{\rm Hom}}

\newcommand{\mat}[4]{{\setlength{\arraycolsep}{0.5mm}\left[
\begin{array}{cc}#1&#2\\#3&#4\end{array}\right]}}
\newcommand{\qed}{\hspace*{\fill}\rule{1ex}{1ex}}
\newcommand{\forget}[1]{}

\def\qdots{\mathinner{\mkern1mu\raise0pt\vbox{\kern7pt\hbox{.}}\mkern2mu
\raise3.4pt\hbox{.}\mkern2mu\raise7pt\hbox{.}\mkern1mu}}

\newenvironment{proof}{\vspace{1ex}\noindent{\it Proof.}\hspace{0.1em}}
	{\hfill\qed\vspace{2ex}}

\newtheorem{lemma}{Lemma}[section]
\newtheorem{theorem}[lemma]{Theorem}
\newtheorem{corollary}[lemma]{Corollary}
\newtheorem{proposition}[lemma]{Proposition}
\newtheorem{definition}[lemma]{Definition}
\newtheorem{remark}[lemma]{Remark}

\numberwithin{equation}{section}

\newcommand{\bmx}{\left[ \begin{matrix}}
\newcommand{\emx}{\end{matrix} \right] }
\newcommand{\vol}{\mathrm{vol}}
\newcommand{\alg}{\mathrm{alg}}
\newcommand{\fin}{\mathrm{fin}}

\renewcommand{\~}{\widetilde}
\newcommand{\rref}[1]{\romannumeral 0\ref{#1}\relax )}

\newcommand{\I}{{\rm I}}
\newcommand{\M}{{\mathrm M}}

\title{Test vectors and central $L$-values for GL(2)}
\author{Daniel File\footnote{dfile@muhlenberg.edu, Department
of Mathematics, Muhlenberg College, Allentown, PA 18104}, Kimball Martin\footnote{kmartin@math.ou.edu, Department
of Mathematics, University of Oklahoma, Norman, OK 73019} and Ameya Pitale\footnote{apitale@ou.edu, Department
of Mathematics, University of Oklahoma, Norman, OK 73019}}

\begin{document}

\maketitle

\begin{abstract}
We determine local test vectors for Waldspurger functionals for $\GL_2$, in the case where both the representation of $\GL_2$ and the character of the degree two extension are ramified, with certain restrictions. We use this to obtain an explicit version of Waldspurger's formula relating twisted central $L$-values of automorphic representations on $\GL_2$ with certain toric period integrals. As a consequence, we generalize an
average value formula of Feigon and Whitehouse, 
and obtain some nonvanishing results.\footnote{$2000$ Mathematics Subject Classification.
Primary $11F41$, Secondary $11F67, 11F70$} 
\end{abstract}
\setcounter{tocdepth}{1}
\tableofcontents

\section{Introduction}

\subsection{Global results}\label{global-res-sec}
Let $F$ be a number field and $\pi$ be a cuspidal automorphic representation of
$\GL_2(\A_F)$.  
Let $L/F$ be a quadratic extension and $\Omega$ an id\`ele class character of
$L^\times$ such
that $\Omega|_{\A_F^\times} = \omega_\pi$, the central character of $\pi$.  We are interested in
the central value of the
$L$-function
\[ L(s,\pi_L \otimes \Omega) = L(s, \pi \times \theta_\Omega), \]
where $\pi_L$ denotes the base change of $\pi$ to $\GL_2(\A_L)$ and
$\theta_\Omega$ denotes
the theta series on $\GL_2(\A_F)$ associated to $\Omega$.  Note this contains
the following 
interesting special case: when $\Omega$ is trivial, then
$L(s,\pi_L\otimes \Omega) = L(s,\pi)L(s,\pi \otimes \eta)$, where
$\eta=\eta_{L/F}$ denotes the
quadratic character of $\A_F^\times$ associated to $L$ via class field theory.
Assume that $\omega_\pi$ is trivial or $\eta$. Then $\epsilon(1/2, \pi_L \otimes
\Omega) = \pm 1$,
even though $\pi_L \otimes \Omega$ need not be self-dual (cf.\ \cite{JC}).  In
the case
when $\epsilon(1/2, \pi_L \otimes \Omega) =-1$, the central value $L(1/2, \pi_L
\otimes \Omega)=0$. 
Henceforth, assume $\epsilon(1/2, \pi_L \otimes \Omega) = +1$.

Let $D$ be a quaternion algebra over $F$ containing
$L$ such that $\pi$ has a Jacquet--Langlands transfer to an automorphic
representation $\pi'$
of $D^\times(\A_F)$. We allow for the possibility that $D=M_2(F)$ and $\pi'=\pi$,
so there is always
at least one such $\pi'$. Embed $L^\times$ as a torus $T$
inside $D^\times$.  
The period integrals we are interested in are
\begin{equation}\label{klm-PDdef}
 P_D(\phi) = \int\limits_{Z(\A_F) T(F) \backslash T(\A_F)} \phi(t) \Omega^{-1}(t) \,
dt, 
\end{equation}
where $\phi \in \pi'$ and $Z$ denotes the center of $D^\times$ (with $dt$
as in Section \ref{klm-lval-sec}).
If $F=\Q$ and $L$ is imaginary quadratic, then this period simplifies to a finite
sum over certain ``CM points.''

When $\omega_\pi$ is trivial, a beautiful theorem of Waldspurger \cite{wald}
states that 
\begin{equation} \label{intro-waldform}
\frac{|P_D(\phi)|^2}{(\phi,\phi)} = \zeta(2)\prod_v \alpha_v(L,\Omega,\phi) 
\frac{L(1/2, \pi_L \otimes \Omega)}{L(1, \pi, Ad)}
\end{equation}
for any $\phi \in \pi'$.  Here $( \cdot, \cdot )$ is a certain inner product on
$\pi'$
and the factors $\alpha_v(L,\Omega,\phi)$ are certain local integrals which
equal 1 at almost all places.  For all but one $D$, $P_D \equiv 0$ for local reasons.
Namely, the linear functional $P_D$ factors into a product of local linear functionals
$P_{D,v}$.  There is a unique $D \supset L$ for which all $P_{D,v} \ne 0$,
and this $D$ is determined by local epsilon factors in work of Tunnell \cite{tunnell}
and Saito \cite{saito}.
Fixing this $D$, one now gets the nonvanishing criterion: $L(1/2, \pi_L \otimes \Omega) \ne 0$ if
and only if $P_D \ne 0$.

It is useful to have a more explicit version of this formula for certain
applications like equidistribution, nonvanishing, subconvexity, $p$-adic $L$-functions, etc.\
(see, e.g.,
\cite{popa}, \cite{MW},
\cite{FW}, \cite{hsieh}).  
In particular, it is not even obvious from \eqref{intro-waldform}
that $L(1/2, \pi_L \otimes \Omega) \ge 0$, as predicted by the Grand Riemann
Hypothesis.  This positivity result
was subsequently shown by Jacquet and Chen \cite{JC} using a trace formula
identity.

Explicit versions of \eqref{intro-waldform} have been considered by many authors
under various assumptions,
e.g., Gross \cite{gross-svf}, Zhang \cite{zhang}, Xue \cite{xue}, Popa
\cite{popa}, the second author and Whitehouse \cite{MW}, Murase \cite{murase},
Hida \cite{hida-kyoto} and Hsieh \cite{hsieh}.
These explicit formulas rely on picking out a suitable {\em test vector} $\phi$
in \eqref{intro-waldform}.
All of these works rely on the theta correspondence (as did \cite{wald}), except for \cite{MW}
which uses the trace formula identity from \cite{JC}.  The only assumption in
\cite{MW} is that $\pi$ and
$\Omega$ have disjoint ramification, i.e., for any finite place $v$ of $F$, 
$\pi$ and 
$\Omega$ are not both ramified at $v$.  In this case one has a natural choice
for the test vector $\phi$ from  
the work of Gross and Prasad \cite{GP} on \emph{local} test vectors.
 In \cite{MW}, it was noted that this restriction of disjoint ramification is
not essential to the
method and could be removed if one had a reasonable way to define the test
vector $\phi$ in a more general
setting. 

The main local results of this paper (see Theorems \ref{intro-tv-split}, \ref{intro-tv-inert} below) are the existence and characterization of suitable local test vectors in the case of joint ramification under 
certain conditions.
This allows us to extend the formula of \cite{MW} to these cases.  To be precise,
for a finite place $v$ of $F$,
let $c(\pi_v)$ be the (exponent of) the conductor of $\pi_v$ and $c(\Omega_v)$
be the (exponent of) the 
``$F$-conductor'' of $\Omega$ (see \eqref{conductor-of-Omega}). 
Then we make the following assumption:

\begin{equation} \label{tameramcond}
\text{If $v < \infty$ is inert in $L$ and $c(\pi_v)$, $c(\Omega_v) > 0$, then we have $c(\Omega_v) \ge c(\pi_v)$.}
\end{equation}
In particular, if the level $N=\prod_{v < \infty} \varpi_v^{c(\pi_v)}$ of $\pi$
is squarefree, there is no condition on $\Omega$.
We note a consequence of our determination of test vectors is that assumption
\eqref{tameramcond} implies $D$ and $\Omega$ do not have joint ramification at
any finite place.

Theorems \ref{intro-tv-split} and \ref{intro-tv-inert} below give suitable local test
vectors $\phi_v$ under assumption \eqref{tameramcond},
which yields the desired global test vector $\phi$. 
Here suitable essentially means that the local test vectors can be described purely in terms of
ramification data, and do not require more refined information about local representations.  
This is crucial for global applications.
Note that it is not even \emph{a priori} clear if suitable test vectors should exist in general.

Let us now describe
the $L$-value formula more precisely.
Denote the absolute value of the discriminants of $F$ and $L$ by $\Delta$ and
$\Delta_L$.  Let $e(L_v/F_v)$
be the ramification degree of $L_v/F_v$.
Let $S_{\mathrm{inert}}$ be the set of places of $F$ inert in $L$.
Let $S(\pi)$
(resp.\ $S(\Omega)$)
be the set of finite places of $F$ where $\pi$ (resp.\ $\Omega$) is ramified,
$S_1(\pi)$ (resp. $S_2(\pi)$) the set of places in $S(\pi)$ where
$c(\pi_v)=1$ (resp. $c(\pi_v) \geq 2$), and $S_0(\pi) = S_2(\pi) \cup \{ v \in S_1(\pi) : L_v/F_v \text{ is ramified and }
\Omega_v \text{ is unramified} \}$.  Denote by $c(\Omega)$
the absolute norm of the conductor of $\Omega$.

\begin{theorem} \label{intro-lvalthm}
 Let $\pi$ be a cuspidal automorphic representation of $\GL_2(\A_F)$ with trivial
central character
and $\Omega$ a character of $\A_L^\times/L^\times \A_F^\times$.  Assume
$\epsilon(1/2,\pi_L \otimes \Omega)=1$ and that $\pi$ and $\Omega$ satisfy (\ref{tameramcond}).
Then, with the test vector $\phi \in \pi'$ defined in Section \ref{klm-testvec}
and archimedean factors
$C_v(L,\pi,\Omega)$  defined in Section \ref{klm-arch-fact},  we have
\begin{multline*}
 \frac{|P_D(\phi)|^2}{(\phi,\phi)} =
 \frac 12 \sqrt{ \frac{\Delta}{c(\Omega)\Delta_L } } 
 L_{S(\Omega)}(1,\eta) L_{S(\pi) \cup S(\Omega)}(1,\eta) 
 L_{S(\pi) \cap S(\Omega)}(1,1_F)  L^{S(\pi)}(2,1_F) \\
\times   \prod_{v \in S(\pi) \cap S(\Omega)^{c}} e(L_v/F_v)  \prod_{v | \infty}
C_v(L,\pi,\Omega) 
  \cdot  \frac{ L^{S_0(\pi)}(1/2,\pi_L \otimes \Omega) }{L^{S_0(\pi)}(1,\pi,Ad) }.
\end{multline*}
Here $( \cdot, \cdot )$ is the standard inner product on $\pi'$ with respect to
the measure on $D^\times(\mathbb A_F)$
which is the product of local Tamagawa measures.
\end{theorem}

After our paper was originally completed, the paper \cite{CST} appeared, which gives a similar formula using a less explicit choice of test vector.

Note $\phi$ is specified up to a scalars, and the left hand side is invariant under scaling.
As in \cite{MW}*{Thm 4.2}, one can rewrite this formula using the Petersson norm  of
a normalized newform in $ \pi$ instead of $L(1, \pi, Ad)$.  See \eqref{ad-pet-eq} for when
$\pi$ corresponds to a holomorphic Hilbert modular form.
If $F=\Q$ and $\pi$ corresponds to a holomorphic new form of squarefree level $N$ with  $N|c(\Omega)$, then the above formula simplifies
 considerably:
 
  \begin{corollary} Let $f$ be a normalized holomorphic modular eigenform of
weight $k$ and
   squarefree level $N$.
    Let $S$ be the set of primes $p|N$ which split in $L$.
 Let $\Omega$ be any ideal class character of $L$ such that $N|c(\Omega)$ and 
 $\epsilon(1/2,f \times \Omega)=1$.   
 Then
 \[ \frac{|P_D(\phi)|^2}{(\phi,\phi)} =\frac{C_\infty(L,
f,\Omega)}{2^{k+1}\sqrt{c(\Omega)\Delta_L } } 
L_{S(\Omega)}(1,\eta)^2  \prod_{p|N}  (1+p^{-1})^{\epsilon_p}
  \times  \frac{ L^{S}(1/2, f \times \Omega) }{\langle f, f \rangle}, \]
  where $\epsilon_p$ is $+1$ if $p$ splits in $L$ and $-1$ otherwise, and
$\langle \cdot, \cdot \rangle$
  is the Petersson inner product. 
 \end{corollary}
 
In the setting of corollary, $C_\infty(L,f,\Omega)$ is also easier to describe.  If $L$ is real quadratic, 
then $C_\infty(L,f,\Omega) = 2^k$.  If $L$ is imaginary quadratic, it is described by beta functions, and if we also assume $\Omega_\infty$ is trivial, then 
$C_\infty(L,f, \Omega) = \frac{(k/2-1)!^2}{\pi(k-1)!}$.
 
We prove Theorem \ref{intro-lvalthm} by computing local spectral distributions
appearing in the trace formula identity of \cite{JC},
just as in \cite{MW}.  For simplicity, we only do this when $\omega_\pi=1$,
though the
case of $\omega_\pi = \eta$ should be similar.  (One needs either $\omega_\pi=1$
or $\omega_\pi=\eta$ to use the identity from \cite{JC}.)  Note this formula is
considerably more general than the one in \cite{MW} (for trivial
central character) and one expects that it should generalize the applications of the
previously mentioned formulas.  For instance, we obtain the following
generalization
of an average value result of Feigon--Whitehouse \cite{FW}*{Thm 1.1} by computing
the 
geometric side of a certain trace formula.

\begin{theorem} \label{intro-avgval-thm} Let $F$ be a totally real number field with $d=[F:\Q]$.  
Let $\mathcal F(\mathfrak N, 2\mathbf k)$ be the set of cuspidal automorphic representations of $\GL_2(\A_F)$ associated to the
holomorphic Hilbert modular eigen newforms of weight $2 \mathbf k$ and level $\mathfrak N$, with
$\mathbf k = (k_1, \ldots, k_d) \ne (1, \ldots, 1)$ and $\mathfrak N$ squarefree.
Let $L$ be a totally imaginary quadratic extension of $F$, which is inert and unramified above each place $\p | \mathfrak N$.  
Fix a unitary character $\Omega$ of $\A_L^\times/L^\times \A_F^\times$, and let $\mathfrak C$ be the norm of its conductor in $F$.
Suppose $\mathfrak N=\mathfrak N_0 \mathfrak N_1$ and $\mathfrak C = \mathfrak C_0 \mathfrak N_1$
with $\mathfrak N_0$, $\mathfrak N_1$ and $\mathfrak C_0$ all coprime.
Assume $\mathfrak N_1$ is odd, and that the number of primes
dividing $\mathfrak N_0$ has the same parity as $d$. 
Further assume that for each infinite place $v$ of $F$, $k_v > |m_v|$ where $\Omega_v(z) =
(z/\bar z)^{m_v}$. 

 Then, if $|\mathfrak N_0 | > d_{L/F} ( |\mathfrak C_0|/|\mathfrak N_1|)^{h_F}$, where $h_F$ is the class number of $F$, we have
\begin{multline*} 
 \prod_{v | \infty}  \begin{pmatrix} 2k_v-2 \\ k_v-m_v-1 \end{pmatrix}
\sum_{\mathfrak N'} 
\sum_{\pi \in \mathcal F(\mathfrak N', 2\mathbf k)}
 \frac{ L(1/2,\pi_L \otimes \Omega) }{L^{S(\mathfrak N)}(1,\pi,Ad) } \\
= 2^{2-d} \Delta^{3/2} |\mathfrak N| 
L_{S(\mathfrak N_0)}(2,1_{F}) L_{S(\mathfrak N_1)}(1,1_F)L^{S(\mathfrak C_0)}(1,\eta),
\end{multline*}
where $\mathfrak N'$ runs over ideals dividing $\mathfrak N$ which are divisible
by $\mathfrak N_0$, and $S(\mathfrak J)$ denotes the set of all primes dividing $\mathfrak J$.
\end{theorem}

The parity condition guarantees the sign
of the relevant functional equations $\epsilon(1/2, \pi_L \otimes \Omega)$ are 
$+1$ for $\pi \in \mathcal F(\mathfrak N, \mathbf k)$.
Without a condition to the effect that $\mathfrak N$ (or $\mathfrak N_0$) is
large, 
one does not expect a nice explicit formula, but rather just an asymptotic 
in $\mathfrak N$, which miraculously  stabilizes for $\mathfrak N$ large  (cf.
\cite{MR}, \cite{FW}).
Hence the condition above on the size of $\mathfrak N_0$  means we are in the {\em
stable range}.
The other assumptions in the theorem allow for simplifications of the trace
formula
we will use, but are not necessary to express such averages as the geometric
side of an
appropriate trace formula.

Theorem \ref{intro-avgval-thm} specializes to \cite{FW}*{Thm 1.1} in the case that $\mathfrak N$ and
$\mathfrak C$ are coprime, i.e., $\mathfrak N = \mathfrak N_0$.
This case $\mathfrak N = \mathfrak N_0$ is particularly nice  as one can
transfer the problem to a trace formula computation on a 
quaternion algebra that only picks up forms of exact level $\mathfrak N$. 
Additionally,  one can rewrite the formula in terms
of the complete adjoint $L$-value at 1, as in \cite{FW}.  However, this is
impossible to manage in general, and the primary difficulty in going from
Theorem \ref{intro-lvalthm} to Theorem \ref{intro-avgval-thm} is to determine
the contribution to the spectral side of the relevant trace formula
coming from the oldforms.
 (In general, it is not easy to isolate the newforms in such formulas---see,
e.g., \cite{KL} or
 \cite{nelson}---and the issue for us is that the contribution from the oldforms
is now weighted by local adjoint $L$-factors.) 

Still, one can use the above formula together with formulas for smaller levels
to get both explicit bounds and asymptotics for average values over just the
forms of exact level 
$\mathfrak N$.  We do this in the case $\mathfrak N_1$ is prime.
This immediately implies $L(1/2, \pi_L \otimes \Omega) \ne 0$ for some $\pi_L
\in \mathcal F(\mathfrak N, 2 \mathbf k)$.

\begin{theorem}\label{non-van-thm}
With assumptions as in Theorem \ref{intro-avgval-thm}, and further that $\mathfrak N_1 = \mathfrak
p$ is an odd prime
and $|\mathfrak N_0 | > d_{L/F} |\mathfrak C|^{h_F}$, we have
\[ |\mathfrak p| - \frac 1{1-2|\mathfrak p|^{-1}+|\mathfrak p|^{-2}}
 \le \Sigma(\mathfrak N) \le |\mathfrak p| - \frac 1{1+2|\mathfrak
p|^{-1}+|\mathfrak p|^{-2}},  \]
 where $\Sigma(\mathfrak N)$ is equal to
 \[ \frac {2^{d-2}}{\Delta^{3/2} |\mathfrak N_0| 
L(1,1_{F_{\mathfrak p}}) L_{S(\mathfrak N_0)}(2,1_F)  L^{S(\mathfrak C_0)}(1,\eta)}
 \prod_{v | \infty}  \begin{pmatrix} 2k_v-2 \\ k_v-m_v-1 \end{pmatrix}
\sum_{\pi \in \mathcal F(\mathfrak N, 2\mathbf k)} 
 \frac{ L(1/2,\pi_L \otimes \Omega) }{L(1,\pi,Ad) }. \]
 In particular, $\Sigma(\mathfrak N) \sim |\mathfrak p|-1 + O(|\mathfrak p|^{-1})$
as $|\mathfrak N_0 \mathfrak p| \to \infty$ such that $|\mathfrak N_0| > d_{L/F} | \mathfrak C
|^{h_F}$.
Furthermore, with $\mathfrak p$ fixed, we have
 \[ \lim_{|\mathfrak N_0| \to \infty} \Sigma(\mathfrak N) = |\mathfrak p| - 1. \]
In both of these asymptotics, $\mathfrak N_0$ travels along squarefree ideals
 coprime to $\mathfrak C$ which are products of unramified primes and satisfy
our previous
 parity assumption.
\end{theorem}

Note the above theorem implies the nonvanishing of $\Sigma(\mathfrak N)$, and therefore at
least one of these central values, provided $|\mathfrak p| > \frac{3+\sqrt 5}2$
and $|\mathfrak N_0| > d_{L/F} | \mathfrak C|^{h_F}$, or $\mathfrak p$ is
arbitrary and $|\mathfrak N_0|$ is sufficiently large.

We remark that the bounds come from having to estimate the $p$-th
Hecke 
eigenvalues $\{ a_p, a_p^{-1} \}$ of the oldforms of level $\mathfrak N_0$.  The latter
asymptotic
comes from an asymptotic for a weighted analogue of Theorem
\ref{intro-avgval-thm}
in the case of disjoint ramification (see \cite{FW}*{Thm 1.2}) 
to pick off the contribution from the old forms.  One should be able to prove a
version of
Theorem \ref{intro-avgval-thm} involving weighting by Hecke eigenvalues (namely,
extend \cite{FW}*{Thm 6.1} to the case of joint ramification) whereby one could
inductively
obtain asymptotics for the average values $\Sigma(\mathfrak N)$ in the case where
$\gcd(\mathfrak N, \mathfrak C)$
has an arbitrary number of prime factors.
(We remark Sugiyama and Tsuzuki \cite{ST} have recently obtained asymptotics for weighted averages
using a different relative trace formula approach when $\Omega$ is trivial, but $\mathfrak N$
need not be squarefree.)

Note that in previous studies of such averages, $\mathfrak N$ is typically required to
be prime (e.g., \cite{RR}) or have an odd number of prime factors (e.g.,
\cite{FW}) to force
the sign of the functional equation to be $+1$ if, say, $d$ is odd.  However, allowing for joint ramification we can treat levels $\mathfrak N$ with an arbitrary number of
prime divisors, though we do not always get an exact formula in this situation.

Lastly, we include another application of Theorem \ref{intro-avgval-thm} when
$\mathfrak N= \mathfrak N_0$ (i.e., \cite{FW}*{Thm 1.1}).  Here, having an exact
formula for the average value
over newforms allows us to deduce the nonvanishing mod $p$ of the algebraic
part 
$L^\alg(1/2, \pi_L \otimes \Omega)$ (see \eqref{lvalalg-def})
of the central value for $p$ suitably large.  

\begin{theorem} \label{avgval-modp}
With notation and assumptions as in Theorem \ref{intro-avgval-thm}, suppose
 $|\mathfrak N | > d_{L/F} |\mathfrak C| ^{h_F}$,
$\mathfrak N$ is coprime to $\mathfrak C$, and, for each $v | \infty$, $m_v$ is even.
Let
 $p$ be an odd rational prime satisfying $p > q+1$ for all primes $q \in
S(\Omega)$, and
  $\mathcal P$  a prime of $\bar {\mathbb Q}$ above $p$. 
Then  there exists $\pi \in \mathcal F(\mathfrak N, 2\mathbf k)$
such that
\[ L^{\alg}(1/2, \pi_L \otimes \Omega) \not \equiv 0 \mod \mathcal P. \]
\end{theorem}

This generalizes a theorem of Michel and
Ramakrishnan \cite{MR} on the case $F=\Q$ and $\mathfrak N=N$ is prime.
The parity condition on $m_v$ ensures that $\Omega$ is algebraic and that the above central
value is critical.

As in \cite{FW}, one should be able to use Theorem \ref{intro-lvalthm} to get
estimates on more general averages of $L$-values, and apply this
to subconvexity and equidistribution problems, but we do not address this here.
Theorem \ref{intro-lvalthm} has also been used in very recent works of Hamieh \cite{hamieh}
and Van Order \cite{VO}, respectively on valuations of Rankin--Selberg $L$-values in 
anticyclotomic towers and on constructing $p$-adic $L$-functions.

We remark that similar $L$-value formulas have been recently proven
in certain cases of joint ramification
with $L$ totally imaginary,
namely in Hida \cite{hida-kyoto} for $F=\Q$ and in Hsieh \cite{hsieh} for
Hilbert modular forms of square free level
(these works have some additional conditions, but they do not assume trivial
central character).
In general, when the joint ramification does not satisfy \eqref{tameramcond}, this problem appears
considerably more 
complicated. 

\subsection{Local results}\label{local-res}
Now, we pass to the local situation and discuss the local test vectors in some
detail.

Let $F$ be a $p$-adic field and $L$ a quadratic separable extension of $F$
(either a field or $F \oplus F$).   We may then embed $L^\times$ as a torus
$T(F)$ of $\GL_2(F)$.  All such embeddings are conjugate in $\GL_2(F)$, so the choice
of embedding will be merely one of convenience.
Consider an (infinite-dimensional) irreducible  admissible representation $\pi$ of 
$\GL_2(F)$.  We do not assume that the central character $\omega_\pi$ is trivial. A basic question to ask is the following: which characters of $T(F)$ appear as
quotients in $\pi|_{T(F)}$? Let $\Omega$ be a character of $T(F)$.   If $\Omega$
is an irreducible constituent of  $\pi|_{T(F)}$, i.e., if 
\[ \mathrm{Hom}_{T(F)}(\pi, \Omega) \ne 0, \]
then we must have $\Omega|_{Z(F)} = \omega_\pi$, where $Z$ denotes the center of
$\GL_2$.
Hence we will assume $\Omega|_{Z(F)} = \omega_\pi$.

Let $D$ be the unique quaternion division algebra
over $F$, and let $\pi'$ be the Jacquet--Langlands transfer to $D^\times(F)$
when it exists.
If $\pi'$ exists and $T(F)$ embeds into $D^\times(F)$, put $A(\pi) = \{ \pi,
\pi' \}$.  Otherwise, put
$A(\pi) = \{ \pi \}$.  From \cite{wald}, one knows that
\[  \sum_{\tau \in A(\pi)} \dim_{\C}  \mathrm{Hom}_{T(F)}(\tau, \Omega) = 1. \]
In other words, $\Omega$ is a constituent of $\pi|_{T(F)}$ if and only 
if it does not occur in that of $\pi'|_{T(F)}$ (when this makes sense), and it
occurs with multiplicity
at most one.  Further, Tunnell \cite{tunnell} and Saito \cite{saito} gave a
local $\epsilon$-factor criterion:
\begin{equation*}\mathrm{dim_\C \, Hom}_{T(F)} ( \pi,
\Omega)=\frac{1+\varepsilon(1/2,\pi_{L} \otimes \Omega)\omega_\pi(-1)}{2}.\end{equation*}

Applications to a global $L$-value formula (discussed in Section \ref{global-res-sec}) require finer information than this. 
Namely, suppose 
$\dim_{\C}  \mathrm{Hom}_{T(F)}(\pi, \Omega) = 1$ and let 
$\ell \in \mathrm{Hom}_{T(F)}(\pi, \Omega)$ be nonzero.  Then one would like to
 have a
{\em test vector} for $\ell$, i.e., an element $\phi \in \pi$ such that
$\ell(\phi) \ne 0$.  For the applications, we will need $\phi$ to satisfy two further conditions.
\begin{enumerate}
\item $\phi \in V_\pi^K$ for a compact subgroup $K$ of $\GL_2(F)$ with ${\rm dim}(V_\pi^K) = 1$.

\item The compact subgroup $K$ above depends only on the ramification data attached to $\pi$ and $\Omega$.
\end{enumerate}
Let us note that, if $\ell \neq 0$, then some translate of the new vector of $\pi$ is always a test vector for $\ell$. Hence, we can always find a test vector satisfying the first condition above. Under some restriction on the conductors of $\pi$ and $\Omega$, we will obtain a test vector satisfying the second condition as well.

Specifically, let $\mathfrak o$ be the ring of integers of $F$, $\mathfrak p$ its
maximal ideal and $\varpi$ a uniformizer. Let $c(\pi)$ be the exponent of the conductor
of $\pi$ as defined in Section \ref{conductors}, and let
\[ K_1(\mathfrak p^{c(\pi)}) = \{ \mat{a}{b}{c}{d} \in \GL_2(\mathfrak o) : c \in
\mathfrak p^{c(\pi)},  \, d \in 1+\mathfrak p^{c(\pi)} \}. \]  
Let 
$c(\Omega)$ be the conductor of 
$\Omega$ as defined in (\ref{conductor-of-Omega}).  Gross and Prasad
\cite{GP}
 determine a test vector when $c(\pi)=0$ ($\pi$ is
unramified) or $c(\Omega)=0$ ($\Omega$ is unramified).  In particular, when $c(\pi)=0$ so
$A(\pi) = \{ \pi \}$, the vector they obtain can be described
as a translate of the new vector.

We will now describe test vectors when $\pi$ and $\Omega$ are both ramified. We will distinguish the split and field case. 

\subsubsection{The split case} 

\begin{theorem} \label{intro-tv-split} Suppose $L=F \oplus F$ and let $T(F) \cong L^\times$ 
be the diagonal torus in $\GL_2(F)$.
 Let $\pi$ be any infinite dimensional, irreducible, admissible representation of $\GL_2(F)$ with
central character $\omega_\pi$ and conductor $\p^{c(\pi)}$, $c(\pi) \geq 0$. 
Let $\Omega(\mathrm{diag}(x,y)) = \Omega_1(x) \Omega_2(y)$ be a character of $T(F)$ such that
$\Omega_1 \Omega_2 = \omega_\pi$.
Without loss of
generality, assume that $c(\Omega_1) \geq c(\Omega_2)$. Write $\Omega_1= | \cdot
|^{1/2-s_0}\mu$ for some $s_0 \in \C$ and some unitary character $\mu$ of
$F^\times$ such that $\mu(\varpi) = 1$. 
Then $\dim_{\C} \mathrm{Hom}_{T(F)}(\pi, \Omega) = 1$, and for nonzero $\ell \in \mathrm{Hom}_{T(F)}(\pi, \Omega)$, the subgroup $h K_1(\mathfrak p^{c(\pi)}) h^{-1}$ fixes a 1-dimensional space of $\pi$ consisting of test vectors for $\ell$, where
$$h = \begin{cases} \mat{1}{\varpi^{-c(\Omega)}}{0}{1}  & \text{ if } c(\mu) = 0 \text{ or } L(s, \pi \otimes \mu^{-1}) \text{ does not have a pole at } s = s_0;\\
w \mat{1}{\varpi^{-c(\Omega)}}{0}{1} & \text{ if } c(\mu) > 0 \text{ and } L(s, \pi \otimes \mu^{-1}) \text{ has a pole at } s = s_0, \\
& \text{ but } L(1-s, \tilde\pi \otimes \mu) \text{ does not have a pole at } s = s_0.
\end{cases}$$
In particular, if both $\Omega$ and $\pi$ are unitary, then we are always in the first case above.
\end{theorem}
Here, $w = \mat{0}{1}{-1}{0}$.
The proof of the above theorem uses the theory of zeta integrals for $\GL_2$ representations given by their Whittaker models. The zeta integral $Z(s_0, \ast, \mu^{-1})$ (defined in (\ref{zeta-int-defn})) divided by the $L$-value $L(s_0, \pi \otimes \mu^{-1})$ gives a concrete realization of a nonzero $\ell \in \mathrm{Hom}_{T(F)}(\pi, \Omega)$.  One checks that the newform in the Whittaker model translated by the matrix $h$ in the statement of the above theorem is a test vector for $\ell$. 

Note that we do not give a compact that fixes a $1$-dimensional space of $\pi$ consisting of test vectors for $\ell$ when both $L(s, \pi \otimes \mu^{-1})$ and $L(1-s, \tilde\pi \otimes \mu)$ have a pole at $s=s_0$.

\subsubsection{The field case}

\begin{theorem} \label{intro-tv-inert} Suppose $L$ is a field. Let $\pi$ be any infinite dimensional, irreducible, admissible representation of $\GL_2(F)$ with
central character $\omega_\pi$ and conductor $\p^{c(\pi)}$. Let $\Omega$ be a character on
$L^\times$  such that $\Omega |_{F^\times} = \omega_\pi$. Assume that $c(\Omega) \geq c(\pi) > 0$.  
Embed $L^\times$ as a torus $T(F)$  in $\GL_2(F)$ as in Section \ref{torus-notation}. Then 
$\dim_{\C}  \mathrm{Hom}_{T(F)}(\pi, \Omega) = 1$, and for a nonzero $\ell \in \mathrm{Hom}_{T(F)}(\pi, \Omega)$, the subgroup
\[\mat{\OF^\times}{\p^{c(\Omega)}}{\p^{c(\pi)-c(\Omega)}}{1+\p^{c(\pi)}} \cap \GL_2(F) =  h K_1(\mathfrak p^{c(\pi)}) h^{-1}, \quad h = \mat{\varpi^{c(\Omega)-c(\pi)}}{0}{0}{1}w,  \]
fixes a 1-dimensional space of $\pi$ consisting of test vectors for $\ell$.
\end{theorem}

If $\pi$ has trivial central character, then we can replace the compact subgroup in the statement of the above theorem by $\mat{\varpi^{c(\Omega)}}{}{}{1} \mat{\OF^\times}{\OF}{\p^{c(\pi)}}{\OF^\times} \mat{\varpi^{-c(\Omega)}}{}{}{1}$, since the Atkin Lehner element normalizes the group $\mat{\OF^\times}{\OF}{\p^{c(\pi)}}{\OF^\times}$.

The proof of the theorem breaks up into several cases depending on the type of the representation $\pi$. Although the proofs are quite different in all cases, it turns out that one of the key ingredients of the proof is that a function roughly of the form $x \mapsto \Omega(1+x\beta)$ (see Section \ref{L-def-section} for details on notation) is an additive character of $\OF$ of a specific conductor. The condition $c(\Omega) \geq c(\pi)$ is required to make this key ingredient work. Also, in certain cases we obtain test vectors for more general situations than the one mentioned above.

\subsubsection*{Principal series}
If $\pi$ is a principal series representation, then we realize it in its induced model and explicitly define a linear functional $\ell(f) = \int_{Z(F)\backslash T(F)} f(t) \Omega^{-1}(t) dt$. It is easy to see that $\ell \in \mathrm{Hom}_{T(F)}(\pi, \Omega)$. We are able to show, for any $c(\pi), c(\Omega) \geq 0$, that $\ell \neq 0$. See (\ref{the-test-vector}) and (\ref{non-van-A}) for details. It is not clear if the explicit test vector for $\ell$ obtained in (\ref{the-test-vector}) belongs to a one-dimensional subspace of $\pi$ of vectors right-invariant under a compact subgroup.
 It is also not clear how to obtain a component that is right invariant under a conjugate of $K_1(\mathfrak p^{c(\pi)})$. To obtain a test vector with the right invariance mentioned in the statement of the theorem, we evaluate $\ell$ at a translate of the newform of $\pi$ by $h$ and show that that is nonzero. For this, we need $c(\Omega) \geq c(\pi) > 0$. If we replace the $h$ in the statement of the theorem by $h=\mat{\varpi^s}{}{}{1}$, $s = c(\pi) - c(\Omega) -v({\bf a})$, where ${\bf a}$ depends on a particular embedding of $T(F)$ in $\GL_2(F)$, then we can extend the result to the case $c(\Omega) \geq 2c(\chi_1)$ (see Proposition \ref{A-nonvanish-lemma}). Here, $\tau = \chi_1 \times \chi_2$ and $c(\chi_1) \leq c(\chi_2)$. 

\subsubsection*{Twists of Steinberg representation}
If $\pi$ is a twist of the Steinberg representation by a ramified character $\chi$, then realizing it as a sub-representation of the reducible induced representation $\chi |\,\,|^{1/2} \times \chi |\,\,|^{-1/2}$, we see that we get the same linear functional and the same nonvanishing of the translate of newform as in the irreducible principal series case. 

If $\pi$ is a twist of the Steinberg representation by an unramified character $\chi$, then we use the fact that such representations are characterized by the existence of a unique (up to constant) vector that is right invariant under the Iwahori subgroup $\I$ and is an eigenvector of the Atkin-Lehner operator with eigenvalue $-\chi(\varpi)$. If we assume that $c(\Omega) \geq c(\pi)$, then \cite{wald} implies the existence of a nonzero $\ell \in \mathrm{Hom}_{T(F)}(\pi, \Omega)$. As in Section \ref{wald-model}, we can then realize $\pi$ as a sub-representation of the space of smooth functions $B : \GL_2(F) \rightarrow \C$ satisfying $B(tg) = \Omega(t) B(g)$. In this latter space, we look for a vector $B$ with three properties: one that is right invariant under $\I$, is zero when averaged over $\GL_2(\OF)/\I$, and is an eigenvector for the Atkin-Lehner operator with eigenvalue $-\chi(\varpi)$. Using a double coset decomposition for $T(F) \backslash \GL_2(F)/\I$, we obtain in Lemma \ref{various-values-lem} the explicit values of such a $B$ for all $g \in \GL_2(F)$. This gives us $B(h) \neq 0$, for $h$ defined in the statement of the theorem. The advantage of the above method is two-fold. It gives us the explicit values of the newform in the Waldspurger model and it also gives another proof of the uniqueness of the Waldspurger model. One can also obtain an independent proof of existence using the methods of \cite{P}, but we do not do that here.

\subsubsection*{Supercuspidal representations}
In the case that $\pi$ is an irreducible supercuspidal representation we may
appeal to Mackey theory. We begin with the explicit construction of
supercuspidal representations of $\GL_2(F)$ by induction from an open subgroup that is
compact modulo the center. Suppose that $J$ is such a
subgroup and $\pi = c\mbox{--}\mathrm{Ind}_{J}^{\GL_2(F)} \rho$. We first describe the situation when $\pi$ is minimal, i.e.\ when the conductor of $\pi$ cannot be lowered upon twisting by a character.

We say that $\rho$ and $\Omega$ intertwine on $T(F) g J$ if $\mathrm{Hom}_{J
\cap g^{-1}T(F)g}(\rho , \Omega^g) \neq 0$. Understanding $\mathrm{Hom}_{T(F)} (
\pi, \Omega)$ reduces to understanding the double cosets $T(F) \backslash
\GL_2(F) / J$ on which $\rho$ and $\Omega$ intertwine. We do this in two steps.
The first step is to consider a larger subgroup $K_\mathfrak{A} \supseteq J$
where $K_\mathfrak{A}$ is one of two subgroups depending on $J$. There is a
unique double coset $T(F) h_0 K_\mathfrak{A}$ that depends only on $c(\pi)$ and
$c(\Omega)$ containing a $T(F)\backslash \GL_2(F) / J$ double coset on which $\rho$ and $\Omega$ can possibly intertwine. This double
coset decomposes as the disjoint union of finitely many $T(F)\backslash \GL_2(F)
/ J$ double cosets 
\begin{equation*}T(F) h_0 K_\mathfrak{A} =
\bigsqcup\limits_{i} T(F) h_i J.\end{equation*}
When $c(\Omega) > \lfloor c(\pi)/2 \rfloor$,
we describe this decomposition explicitly, show that one may choose the
representatives $h_i$ to be diagonal matrices, and show for each $i$ 
\begin{equation*}
(J \cap h_i^{-1} T(F) h_i) \ker \rho / Z(F) \ker \rho \cong (J \cap \overline{N}) / (\ker \rho \cap \overline{N}),
\end{equation*}
where $\overline{N}$ is the subgroup of lower triangular unipotent matrices. It suffices to examine $\rho |_{J \cap \overline{N}}$ which decomposes as a direct sum of characters.
We show that there is a unique $i_0$ so that $\rho$ and $\Omega$ intertwine on $T(F)h_{i_0} J$. We conclude that there exists a nonzero linear functional $\ell \in \mathrm{Hom}_{T(F)}(\pi , \Omega)$.  We describe the translate of the newvector in the induced model explicitly, and show that this translate is a test vector.

Finally, we deal with the case of an irreducible supercuspidal representation $\tau$ that is not minimal. In this case $\tau \cong \pi \otimes \chi$ where $\pi$ is a minimal supercuspidal representation and $\chi$ is a character of $F^\times$.  We construct a vector $\varphi_\chi \in \pi$ so that $\varphi_\chi \otimes \chi$ is a translate of the newvector in $\tau$. Using the results of the minimal case, we show that $\varphi_\chi$ is a test vector for $\Omega \otimes \chi^{-1}$.

Similar to the irreducible principal series case, if we replace $h$ in the statement of the theorem by $h=\mat{\varpi^s}{}{}{1}, s = c(\pi) - c(\Omega) -v({\bf a})$, then in the {\it minimal} supercuspidal case, we can extend the result to the case $c(\Omega) \geq \lfloor 3c(\pi)/4 \rfloor + 1$.

\subsubsection{Relation to test vectors of Gross--Prasad} 

We recall some results
of Gross--Prasad \cite{GP}.  For simplicity assume $\omega_\pi =1$, $L/F$ is unramified
 and $\dim_{\C}  \mathrm{Hom}_{T(F)}(\pi, \Omega) = 1$.
  For an order $R$ of $M_2(F)$, let $d(R)$ be the
exponent of its reduced discriminant and
$c(R)$ be the smallest $c \ge 0$ such that $\mathfrak o + \varpi^c \mathfrak o_L
\subset R$.  It is clear that
$R^\times$ can only fix a test vector if $c(R) \ge c(\Omega)$.  Moreover, if we want
$R^\times$  to fix a line in $\pi$,
it is reasonable to try $R$ with $d(R) = c(\pi)$.  Thus one might consider orders
with $c(R) = c(\Omega)$ and $d(R) = c(\pi)$.  
If either $c(\Omega)=0$ or $c(\pi)=0$, then there is a unique-up-to-$L^\times$-conjugacy
order $R$ with 
$c(R) = c(\Omega)$ and $d(R) = c(\pi)$, and \cite{GP} shows that $R^\times$ fixes a line
consisting of test vectors.  If $c(\pi)=0$ then $R$ is a maximal order, but in general $R$ is not an Eichler order.

When $c(\Omega) > 0$ and $c(\pi) > 0$, the invariants $c(R)$ and $d(R)$
no longer specify $R$ uniquely up to
conjugacy by $L^\times$.   However, with above assumptions, Theorem \ref{intro-tv-inert} can be
interpreted as follows:  when $c(\Omega) \ge c(\pi)$,
there is an Eichler order $R$ with $c(R) =c(\Omega)$ and $d(R) = c(\pi)$ such that $R^\times$
fixes a line in $\pi$ which consists of
test vectors.  Moreover, this $R$ can be described uniquely up to
$L^\times$-conjugacy as
the intersection of two maximal ideals $R_1$ and $R_2$ with $c(R_1) = c(\Omega)$ and
$c(R_2) = c(\Omega)-c(\pi)$ which are maximal possible
distance apart in the Bruhat--Tits tree, i.e., $d(R_1, R_2) = c(\pi)$. 
This provides an
intrinsic description of our test vectors, i.e., one without reference to a specific embedding of $L^\times$ in $\GL_2(F)$.
   It would be interesting to know whether other Eichler orders $R$ satisfying 
$c(R)=c(\Omega)$ and $d(R)=c(\pi)$ also pick out test vectors.

Note that if $c(\pi) > 2c(\Omega)$, there is no Eichler order
with $c(R) = c(\Omega)$ and $d(R) = c(\pi)$, which suggests
that the case when $\pi$ is highly ramified, in comparison with $\Omega$, is more complicated
than the reverse situation.

\subsection{Outline}
Our paper consists of two parts, one local and one global.

In the first (local) part of the paper we prove our results on local test
vectors, which we treat in three separate cases.  Section
\ref{besselsubgroupsec} contains our local notation and embedding of $L^\times$
into $\GL_2(F)$.
Then in Section \ref{split-tv-sec} we treat the case where $L/F$ is split, using
zeta integrals.  This proves Theorem \ref{intro-tv-split}.    Now assume $L/F$ 
is inert.  In Section \ref{ps-tv-sec}, we treat the case of principal series and
Steinberg representations.  In Section \ref{sc-tv-sec}, we treat the case of
supercuspidal
representations.  These two sections complete Theorem \ref{intro-tv-inert}.
 Finally, in Section \ref{klm-sd-sec} we compute certain local spectral
distributions associated to our local test vectors.

The global part of the paper consists of two sections.  In Section
\ref{klm-lval-sec}, we use the local spectral calculations of Section
\ref{klm-sd-sec} to prove our $L$-value formula (Theorem
\ref{intro-lvalthm}).  
In Section \ref{klm-avg-sec}, we deduce our results on average values
and nonvanishing (Theorems \ref{intro-avgval-thm},  \ref{non-van-thm} and \ref{avgval-modp}).

\medskip
{\bf Acknowledgements.}  We greatly appreciate the referee for a thorough reading of the 
paper,  whose comments and suggestions improved the manuscript and eliminated some
errors.  We would also like to thank Brooke Feigon, Andrew Knightly,
 Phil Kutzko, Ralf Schmidt and David Whitehouse for helpful discussions.
The second author was partially supported by Simons Foundation Collaboration Grant 
240605.  The third author was partially supported by the National Science Foundation 
grant DMS 1100541.

\section{Local setup}\label{besselsubgroupsec}

Let $F$ be a nonarchimedean local field of characteristic zero, $\OF$ its ring
of integers, $\p$ the maximal ideal of $\OF$, and $\varpi$ a generator of $\p$. Let $q$ be the size of the residue field and we let $v$ be the normalized valuation map on $F$.

For a character $\chi$ of $F^\times$, let $c(\chi)$ be the exponent of its conductor, i.e., $c(\chi) \ge 0$ is minimal
such that $\chi$ is trivial on $(1+\p^{c(\chi)})\cap \OF^\times$.

\subsection{Subgroups and representations of $\GL_2$}\label{gl2-rep-info} \label{conductors}
We will be using the following compact subgroups of $\GL_2(F)$.  
Put $K_1(\OF) = K_2(\OF) = \GL_2(\OF)$.  For $n > 0$, put
\begin{equation}\label{K_1}
K_1(\p^n) = 
\mat{\OF^\times}{\OF}{\p^n}{1+\p^n}, 
\end{equation}
\begin{equation}\label{K_2}
K_2(\p^n) =
\mat{1+\p^n}{\OF}{\p^n}{1+\p^n}.
\end{equation}
For $s \in \Z, n \geq 0$,
\begin{equation}\label{K_1-conj}
K_1^{(s)}(\p^n) = \mat{\varpi^s}{}{}{1} K_1(\p^n) \mat{\varpi^{-s}}{}{}{1}.
\end{equation}
We also have the Iwahori subgroup
\begin{equation}\label{Iwahori-sub}
\I = \mat{\OF}{\OF}{\p}{\OF} \cap \GL_2(\OF).
\end{equation}

Let $(\pi, V)$ be an infinite-dimensional, irreducible, admissible  representation of $\GL_2(F)$. For
$n \geq 0$, denote by $V^n$ the subspace of $K_1(\p^n)$-fixed vectors.
By \cite{JPSS},  one knows $V^n \neq 0$ for some $n$. Further, if
$c(\pi)$ is the minimal $n$ such that $V^n \neq 0$, then ${\rm
dim}(V^{c(\pi)}) = 1$. Call the ideal $\p^{c(\pi)}$ the {\it conductor} of
$\pi$.  If $c(\pi)=0$, then $\pi$ is unramified.  

Then $\pi$ is a principal series, twist of Steinberg (special), or supercuspidal representation.
Let $\chi_1, \chi_2$ be two characters of $F^\times$. The representation $\pi = \chi_1 \times \chi_2$ is the standard induced representation of $\GL_2(F)$ consisting of locally constant functions $f : \GL_2(F) \rightarrow \C$ such that
\begin{equation} \label{principal-series-def}
f(\mat{a}{b}{}{d} g) = \chi_1(a) \chi_2(d) |ad^{-1}|^{1/2} f(g), \qquad \text{ for all } g \in \GL_2(F), a, d \in F^\times, b \in F.
\end{equation} 
This is irreducible if and only if $\chi_1 \chi_2 \neq | \cdot |^{\pm1}$, in which case we say $\chi_1 \times \chi_2$
is a principal series representation.
For a character $\chi$ of $F^\times$, the twist of the Steinberg representation by $\chi$, which we denote by $\chi \St_{\GL_2}$, is the unique
irreducible subrepresentation of the induced representation $\chi |\cdot |^{1/2} \times \chi | \cdot |^{-1/2}$.
The supercuspidal representations are described in Section \ref{sc-tv-sec}.

\subsection{The degree-two extension}\label{L-def-section}

As in \cite{Fu}, we fix three elements $\mathbf{a},\mathbf{b},\mathbf{c}\in F$
such that $\mathbf{d}=\mathbf{b}^2-4\mathbf{a}\mathbf{c}\neq0$. We let
$L=F(\sqrt{\mathbf{d}})$ if $\mathbf{d}\notin F^{\times2}$, and $L=F\oplus F$
otherwise. In the latter case we consider $F$ diagonally embedded in $L$. Let
$z\mapsto\bar z$ be the obvious involution on $L$ whose fixed point set is $F$.
We define the Legendre symbol as
\begin{equation}\label{legendresymboldefeq}
 \Big(\frac L\p\Big)=\begin{cases}
                      -1&\text{if $L/F$ is an unramified field extension},\\
                      0&\text{if $L/F$ is a ramified field extension},\\
                      1&\text{if }L=F\oplus F.
                     \end{cases}
\end{equation}
We will make the following assumptions:
\begin{itemize}
 \item $\mathbf{a},\mathbf{b}\in\OF$ and $\mathbf{c}\in\OF^\times$.
 \item If $\mathbf{d}\notin F^{\times2}$, then $\mathbf{d}$ is a generator of
the discriminant of $L/F$.
 \item If $\mathbf{d}\in F^{\times2}$, then $\mathbf{d}\in\OF^\times$.
\end{itemize}
We define elements $\beta$ and $\xi_0$ of $L$ by
\begin{equation}\label{alphadefeq}
 \beta=\left\{\begin{array}{l@{\qquad\text{if }L}l}
 \displaystyle\frac{\mathbf{b}+\sqrt{\mathbf{d}}}{2\mathbf{c}}&\text{ is a
field},\\[2ex]
 \displaystyle\Big(\frac{\mathbf{b}+\sqrt{\mathbf{d}}}{2\mathbf{c}},\frac{
\mathbf{b}-\sqrt{\mathbf{d}}}{2\mathbf{c}}\Big)&=F\oplus F.
 \end{array}\right.
\end{equation}
\begin{equation}\label{xi0defeq}
 \xi_0=\left\{\begin{array}{l@{\qquad\text{if }L}l}
 \displaystyle\frac{-\mathbf{b}+\sqrt{\mathbf{d}}}{2}&\text{ is a field},\\[2ex]
 \displaystyle\Big(\frac{-\mathbf{b}+\sqrt{\mathbf{d}}}{2},\frac{-\mathbf{b}
-\sqrt{\mathbf{d}}}{2}\Big)&=F\oplus F.
 \end{array}\right.
\end{equation}
If $L$ is a field, let $\OF_L$ be its ring of integers, $\varpi_L$ a uniformizer, and $v_L$
the normalized valuation.  If $L = F \oplus F$, put $\OF_L = \OF \oplus \OF$ and $\varpi_L = (\varpi,1)$.
By Lemma 3.1.1 of \cite{PS1}, in either case,
\begin{equation}\label{integralbasiseq}
 \OF_L = \OF +\OF\beta=\OF+\OF\xi_0.
\end{equation}

\begin{lemma}\label{val-lemma} Suppose $L$ is a field.
The possible valuations of $\beta$ and
$\mathbf{a}$ are the following:
\begin{align}
& v_L(\beta)=v(\mathbf{a})=0\qquad\text{ if } \Big(\frac L\p\Big) = -1, \label{inertavaleq} \\
& v_L(\beta)=v(\mathbf{a})\in\{0,1\}\qquad\text{ if } \Big(\frac L\p\Big) = 0. \label{ramavaleq}
\end{align}
\end{lemma}
\begin{proof}
Consider the identity
\begin{equation}\label{alphaavaleq}
 \frac{\mathbf{b}+\sqrt{\mathbf{d}}}{2\mathbf{c}}\cdot\frac{\mathbf{b}-\sqrt{
\mathbf{d}}}{2\mathbf{c}}=\frac{\mathbf{a}}{\mathbf{c}}.
\end{equation}
If $\Big(\frac L\p\Big)=-1$, we get the result by observing that $d$ is a non-square unit. If $\Big(\frac L\p\Big)=0$, we get the result since $1,\beta$ is an integral basis.
\end{proof}

Fix the ideal in $\OF_L$ given by
\begin{equation}\label{ideal defn}\renewcommand{\arraystretch}{1.3}
 \P_L := \p\OF_L = \left\{
                  \begin{array}{l@{\qquad\text{if }}l}
                    \p_L & \big(\frac L{\p}\big) = -1,\\
                    \p_L^2 & \big(\frac L{\p}\big) = 0,\\
                    \p \oplus \p & \big(\frac L{\p}\big) = 1.
                  \end{array}
                \right.
\end{equation}
Here $\p_L$ is the maximal ideal of $\OF_L$ when $L$ is a field. 
We have $\P_L^n\cap\OF=\p^n$ for all $n\geq0$.

Under our assumptions it makes sense to consider the quadratic equation
$\mathbf{c}u^2+\mathbf{b}u+\mathbf{a}=0$ over the residue class field $\OF/\p$.
The number of solutions of this equation is $\big(\frac L\p\big)+1$. In the
ramified case we will fix an element $u_0\in\OF$ such that
\begin{equation}\label{u0defeq}
 \mathbf{c}u_0^2+\mathbf{b}u_0+\mathbf{a}\in\p;
\end{equation}
see Lemma 3.1.1 of \cite{PS1}.
Further, note that in the ramified case we have 
\begin{equation}
\mathbf{b} + 2 \mathbf{c}u_0 \in \p.
\end{equation}
This follows from the fact that $u_0$ is a double root of $\mathbf{c}u^2+\mathbf{b}u+\mathbf{a}$ over $\OF/\p$.

\subsection{The torus}\label{torus-notation}
We will now specify an embedding of $L^\times$ as a torus in $\GL_2$ for convenience of
calculations. 
With $\mathbf{a},\mathbf{b},\mathbf{c}$ as above, let
$$
 S=\mat{\mathbf{a}}{\frac{\mathbf{b}}2}{\frac{\mathbf{b}}2}{\mathbf{c}},\qquad
 \xi=\mat{\frac{\mathbf{b}}2}{\mathbf{c}}{-\mathbf{a}}{\frac{-\mathbf{b}}2}.
$$
Then $F(\xi)=F \cdot I_2 +F \cdot \xi$ is a two-dimensional $F$-algebra isomorphic to
$L$. If $L$ is a field, then an isomorphism is given by $x+y\xi\mapsto
x+y\frac{\sqrt{\mathbf{d}}}2$. If $L=F\oplus F$, then an isomorphism is given by
$x+y\xi\mapsto(x+y\frac{\sqrt{\mathbf{d}}}2,x-y\frac{\sqrt{\mathbf{d}}}2)$. The
determinant map on $F(\xi)$ corresponds to the norm map on $L$. Let
\begin{equation}\label{TFdefeq}
 T(F)=\{g\in\GL_2(F):\:^tgSg=\det(g)S\}.
\end{equation}
One can check that $T(F)=F(\xi)^\times$. Note that $T(F)\cong L^\times$ via the
isomorphism $F(\xi)\cong L$. Under the same isomorphism the group
$T(\OF):=T(F)\cap\GL_2(\OF)$ is isomorphic to $\OF_L^\times$. Note that $T(F)$
consists of all matrices
\begin{equation}\label{TFgeq}
t(x,y)=\mat{x+y\frac{\mathbf{b}}2}{\mathbf{c}y}{-\mathbf{a}y}{x-y\frac{\mathbf
{b}}2},\qquad x, y \in F, \,
\det(g)=x^2-\frac14y^2(\mathbf{b}^2-4\mathbf{a}\mathbf{c})\neq0.
\end{equation}
We will give a useful structural lemma here.

\begin{lemma}\label{toric-decomp-lemma}
Let $L/F$ be a field extension. For any $m, n \geq 0$, we have 
\begin{equation}\label{toric-decomp-eqn}
T(F) \mat{\varpi^{-m}}{}{}{1} K_1(\p^n) = T(F) \mat{\varpi^{m-v({\bf a})}}{}{}{1} w K_1(\p^n).
\end{equation}
\end{lemma}

\begin{proof}
Set  $y = \varpi^{-m}$ and $x = {\bf b} y/2$. Then 
$$\mat{\varpi^{-m}}{}{}{1}  = \frac{-1}{{\bf a}\varpi^{-v({\bf a})}}  t(x,y) \mat{\varpi^{m-v({\bf a})}}{}{}{1} w k,$$
with 
$$k = \mat{\frac{{\bf a}\varpi^{-v({\bf a})}}{\bf c}}{\frac{{\bf b}\varpi^{m-v({\bf a})}}{\bf c}}{}{1} \in K_1(\p^n), \quad \text{ for all } n \geq 0,$$
since $v(\mathbf b) \ge 1$ whenever $v(\mathbf a) = 1$.
\end{proof}

\subsection{The Waldspurger model}\label{wald-model}
Let $\Omega$ be any character of $L^\times$, which we may view as a character of the torus $T(F)$.  Define
\begin{equation}\label{conductor-of-Omega}
c(\Omega) := \text{ min }\{m \geq 0 : \Omega |_{(1+\P_L^m) \cap \OF_L^\times}
\equiv 1\}.
\end{equation}
Note that this is the (exponent of the) conductor of $\Omega$ only in the case $L/F$ is an unramified field extension. Let $\mathcal{B}(\Omega)$ be the space of all locally constant functions $B :
\GL_2(F) \rightarrow \C$ satisfying 
\begin{equation}\label{Wald-trans-prop}
B(tg) = \Omega(t) B(g) \qquad \text{ for all } t \in T(F), g \in \GL_2(F).
\end{equation}
Let $(\pi, V)$ be any infinite dimensional, irreducible, admissible representation of $\GL_2(F)$. We
say that $\pi$ has an {\it $\Omega$-Waldspurger model} if $\pi$ is isomorphic to
a subrepresentation of $\mathcal{B}(\Omega)$. We call a linear functional $\ell$
on $\pi$ an {\it $\Omega$-Waldspurger functional} if it satisfies
\begin{equation}\label{Wald-fnal-defn}
\ell(\pi(t)v) = \Omega(t) \ell(v) \qquad \text{ for all } t \in T(F), v \in V.
\end{equation}
If $\pi$ has an $\Omega$-Waldspurger model then we obtain an $\Omega$-Waldspurger
functional $\ell$ by $\ell(B) = B(1)$. On the other hand, if $\pi$ has an
$\Omega$-Waldspurger functional $\ell$, we obtain an $\Omega$-Waldspurger model for
$\pi$ by the map $v \mapsto B_v$, where $B_v(g) = \ell(\pi(g)v)$. Observe that a necessary condition for an $\Omega$-Waldspurger model or functional to exist is that $\Omega |_{F^\times} = \omega_\pi$, the central character of $\pi$.

If $\pi$ has an $\Omega$-Waldspurger functional $\ell$, we say that $v \in
V$ is a {\it test vector} for $\ell$ if $\ell(v) \neq 0$. From the discussion
above, this is equivalent to $B_v(1) \neq 0$. Suppose $B_0$ is the newform in an $\Omega$-Waldspurger model of $\pi$. Lemma \ref{toric-decomp-lemma} states that, in the field case, for $m \ge 0$, the vector $\pi(\mat{\varpi^{m-v({\bf a})}}{}{}{1} w)B_0$ is a test vector if and only if $\pi(\mat{\varpi^{-m}}{}{}{1})B_0$ is also a test vector. This will be used in the proof of Theorem \ref{intro-tv-inert} below.

Criteria for existence of Waldspurger functionals, which must be unique up to scalars, are given in Section \ref{local-res}.

\section{Zeta integrals and test vectors for split Waldspurger models}
\label{split-tv-sec}

In this section we will show that any irreducible admissible representation
$\pi$ of $\GL_2(F)$ has a split $\Omega$-Waldspurger model for every character
$\Omega$ of $L^\times = F^\times \oplus F^\times$. Under certain restriction on the poles of the $L$-function of $\pi$, we will also determine test vectors for the Waldspurger functional that are right invariant under certain conjugates of the compact group $K_1(\p^{c(\pi)})$. The conjugating elements will depend only on $c(\pi)$ and $c(\Omega)$.

Let $\pi$ be any irreducible, admissible representation of $\GL_2(F)$ with
central character $\omega_\pi$ (not assumed to be trivial).  Let $\pi$ be given
by its Whittaker model $\mathcal{W}(\pi, \psi)$, where $\psi$ is a nontrivial character of $F$ with conductor
$\OF$. For any $W \in
\mathcal{W}(\pi, \psi)$ and a unitary character $\mu$ of $F^\times$, define the
zeta integral
\begin{equation}\label{zeta-int-defn}
Z(s, W, \mu^{-1}) := \int\limits_{F^\times} W(\mat{x}{}{}{1})|x|^{s-1/2}
\mu^{-1}(x) \, d^\times x,
\end{equation}
where $d^\times x$ is the Haar measure on $F^\times$ giving $\mathfrak o^\times$
volume $1-q^{-1}$. Since $\mu$ is unitary, there is an $r \in \R$ not depending
on $\mu$, so that $Z(s, W, \mu^{-1})$ converges absolutely for $\Re(s) > r$. 
 By the theory of $L$-functions, we have
\begin{equation}\label{defn-L-fn}
\frac{Z(s, W, \mu^{-1})}{L(s, \mu^{-1} \otimes \pi)} \in \C[q^{-s}, q^s]
\end{equation}
 and the functional equation
\begin{equation}\label{local-fnal-eqn}
 \frac{Z(1-s,\pi(w)W,\mu\omega_\pi^{-1})}{L(1-s,\mu\otimes\tilde\pi)}
=\varepsilon(s,\mu^{-1}\otimes\pi,\psi)\frac{Z(s,W,\mu^{-1})}{L(s,\mu^{-1}
\otimes\pi)}
\end{equation}
for any $W \in \mathcal{W}(\pi, \psi)$.  Here $w = \mat{}{1}{-1}{}$.
Please refer to Theorem 6.12 of \cite{G} for details.

Let $W_0$ be the unique $K_1(\p^{c(\pi)})$-right invariant vector in
$\mathcal{W}(\pi, \psi)$ such that $W_0(1) = 1$. The formula for
$W_0(\mat{x}{}{}{1})$ in various cases is given in Table \ref{whit-table} (see, e.g., \cite{Sc}). 
\begin{table}
\caption{Whittaker newform values}
$$\begin{array}{|c|c|}
\hline
\pi & W_0(\mat{x}{}{}{1}) \\ \hline \hline
\chi_1 \times \chi_2, \text{ with }
\chi_1, \chi_2 \text{ unramified, }
\chi_1 \chi_2^{-1} \neq |\,\,|^{\pm 1} & |x|^{1/2}
\big(\sum\limits_{k+l=v(x)}\chi_1(\varpi^k) \chi_2(\varpi^l)\big) 1_{\OF}(x) \\ \hline 
&\\
\chi_1 \times \chi_2, \text{ with }
\chi_1\text{ unramified, } \chi_2 \text{ ramified } & |x|^{1/2} \chi_1(x)
1_{\OF}(x) \\\hline
&\\
\chi \St_{\GL_2} \text{ with } \chi \text{ unramified } & |x| \chi(x) 1_{\OF}(x)
\\\hline
&\\
L(s, \pi) = 1 & 1_{\OF^\times}(x)\\\hline
\end{array}
$$
 \label{whit-table}
\end{table}

\begin{proposition}\label{zeta-int-prop} Let $\pi$ be any irreducible,
admissible representation of $\GL_2(F)$ with central character $\omega_\pi$ and
conductor $\p^{c(\pi)}$. Let $W_0$ be the newform in the Whittaker model
$\mathcal{W}(\pi, \psi)$ of $\pi$ such that $W_0(1) = 1$. Let $\mu$ be a unitary character of
$F^\times$.

\begin{enumerate}
\item If $c(\mu) = 0$ then, for any $\pi$, we have
$$Z(s,\pi(\mat{1}{\varpi^{-c(\mu)}}{}{1}) W_0,\mu^{-1}) = (1 - \frac 1q) L(s, \mu^{-1} \otimes \pi).$$

\item Let $c(\mu) > 0$. Then
$$Z(s,\pi(\mat{1}{\varpi^{-c(\mu)}}{}{1}) W_0,\mu^{-1}) = q^{-c(\mu)/2} \mu(\varpi^{-c(\mu)}) \varepsilon(1/2, \mu, \psi).$$

\end{enumerate}
\end{proposition}
\begin{proof} If $c(\mu) = 0$, then the values of the newform $W_0$ from Table \ref{whit-table} and the normalization of the measure gives us i). We have, for any $k \in \Z$ and any $\pi$,
\begin{align*}
 Z(s, \pi(\mat{1}{\varpi^k}{}{1}) W_0,\mu^{-1})
 &=\int\limits_{F^\times}\psi(a\varpi^k)W_0(\mat{a}{}{}{1})|a|^{s-1/2}\mu^{-1}
(a)\,d^\times a\\
  &=\sum_{j\in\Z}\;\int\limits_{\OF^\times}\psi(a\varpi^{j+k})W_0(\mat{a\varpi^j}
{}{}{1})|a\varpi^j|^{s-1/2}\mu^{-1}(a\varpi^j)\,d^\times a\\
 &=\sum_{j\in\Z}q^{-j(s-1/2)}\mu^{-1}(\varpi^j) W_0(\mat{\varpi^j}{}{}{1})
\int\limits_{\OF^\times}\psi(a\varpi^{j+k})\mu^{-1}(a)\,d^\times a.
 \end{align*}
 If $c(\mu) > 0$, then, by the definition of the epsilon factor for $\mu$ (see
\cite{Sc}* {Eqn.\ (5)}),
 we have
 \begin{equation}\label{eps-defn-cond}
 \int\limits_{\OF^\times}\psi(a\varpi^{j+k})\mu^{-1}(a)\,d^\times a =
\begin{cases} q^{-c(\mu)/2} \mu(\varpi^{j+k}) \varepsilon(1/2, \mu, \psi) &
\text{ if } j+k = -c(\mu),\\
 0 & \text{ if } j+k \neq -c(\mu).\end{cases}
 \end{equation}
Now the
proposition follows since $W_0(1)=1$.
\end{proof}

\noindent
{\bf Proof of Theorem \ref{intro-tv-split}}: For any $W \in \mathcal{W}(\pi, \psi)$
define
\begin{equation}
\ell(W) := \frac{Z(s_0, W, \mu^{-1})}{L(s_0, \mu^{-1} \otimes \pi)}.
\end{equation}
The well-definedness of $\ell$ for all $s_0$ and $\mu$ follows
from (\ref{defn-L-fn}). By Theorem 6.12 of \cite{G}, $\ell$ is non-zero. The definition of the zeta integral and $\Omega_1 \Omega_2 = \omega_\pi$ gives us the transformation property 
$$\ell(\pi(\mat{x}{}{}{y})W) = \Omega_1(x) \Omega_2(y) \ell(W), \qquad x, y \in
F^\times.$$
Hence, we get $\mathrm{Hom}_{T(F)}(\pi, \Omega) \neq 0$. The one-dimensionality follows from \cite{wald}. Note that, if $c(\mu) = 0$ or $L(s, \mu^{-1} \otimes \pi)$ does not have a pole at $s=s_0$,  then
we have $\ell(\pi(\mat{1}{\varpi^{-c(\Omega)}}{}{1})W_0) \neq 0$ by  Proposition
\ref{zeta-int-prop}.  If $c(\mu) > 0$ and $L(s, \mu^{-1} \otimes \pi)$ has a pole at $s=s_0$, then $\ell(\pi(\mat{1}{\varpi^{-c(\Omega)}}{}{1})W_0) = 0$ by  Proposition
\ref{zeta-int-prop}. In this case, if we assume that $L(1-s, \mu \otimes \tilde\pi)$ does not have a pole at $s=s_0$, then we can use  the local functional equation
(\ref{local-fnal-eqn}), which  gives us the test vectors for $\ell$. The uniqueness statement follows from the uniqueness of $W_0$. If $\Omega$ and $\pi$ are unitary, then $s_0 = 1/2$ and one can check that $L(s, \mu^{-1} \otimes \pi)$ does not have a pole at $s=1/2$. \qed
\section{Non-supercuspidal representations}
\label{ps-tv-sec}

Here we will assume that $L$ is a field and prove Theorem \ref{intro-tv-inert} when $\pi$ is not supercuspidal.

Let us define Haar measures $dg$ on $\GL_2(F)$ such that $\GL_2(\OF)$ has volume $1$,  
$d^\times x$ on $F^\times = Z(F)$, the center of $\GL_2(F)$, such that $\OF^\times$ has volume $1$ (note this is different from Section \ref{split-tv-sec}), and $dt$ on $T(F) = L^\times$ such that the volume of $\OF_L^\times$ is $1$.

\subsection{Irreducible principal series representation}\label{irr-prin-section}
 Let $\pi$ be a {\it ramified} irreducible principal series representation of $\GL_2(F)$ given
by 
\begin{equation}\label{princ-ser-defn}
\pi = \chi_1 \times \chi_2, \quad \chi_1 \chi_2^{-1} \neq |.|^{\pm 1}, \quad
c(\chi_2) \geq c(\chi_1) \quad c(\pi) = c(\chi_1) + c(\chi_2) > 0, \quad
\omega_\pi = \chi_1 \chi_2.
\end{equation}
Recall that $\pi$ consists of locally constant functions $f$ on $\GL_2(F)$ satisfying
\eqref{principal-series-def}.
The unique, up to scalars, right $K_1(\p^{c(\pi)})$-invariant vector $f_0$ in
$\pi$ is given by the formula
\begin{equation}\label{new-form-formula}
f_0(g) = \begin{cases} |a/d|^{1/2}\chi_1(a) \chi_2(d) & \text{ if } g \in
\mat{a}{\ast}{}{d} \gamma_{c(\chi_2)} K_1(\p^{c(\pi)}),\\
0 & \text{ if } g \not\in B(F) \gamma_{c(\chi_2)} K_1(\p^{c(\pi)}), \end{cases}
\end{equation}
where $\gamma_{c(\chi_2)} = \mat{1}{}{\varpi^{c(\chi_2)}}{1}$ and $B(F)$ is the Borel subgroup of $\GL_2(F)$ consisting of upper triangular matrices. See \cite{Sc} for
details. 

Let $\Omega$ be a character of $L^\times$ such that $\Omega |_{F^\times} = \omega_\pi$.  
Let $\mathcal{B}(\Omega)$ be the space of all
locally constant functions $B : \GL_2(F) \rightarrow \C$ satisfying
(\ref{Wald-trans-prop}) defined in Section \ref{wald-model}. Define an intertwining operator $\mathcal{A} : \pi
\rightarrow \mathcal{B}(\Omega)$ by the formula
\begin{equation}\label{bessel-intertwiner-defn}
(\mathcal{A}(f))(g) := \int\limits_{Z(F) \backslash T(F)} f(tg) \Omega^{-1}(t)
dt, \qquad f \in \pi, \, g \in \GL_2(F).
\end{equation}
Since  $Z(F) \backslash T(F)$ is compact and $\Omega |_{F^\times} = \omega_\pi$, this integral
is well-defined and convergent.  Note $Z(F) \backslash T(F)$ is isomorphic to
$Z(\OF) \backslash T(\OF)$ if $\Big(\frac L{\p}\Big) = -1$, and $Z(\OF)
\backslash T(\OF) \sqcup \varpi_L \big(Z(\OF) \backslash T(\OF)\big)$ if
$\Big(\frac L{\p}\Big) = 0$. 

Next we will show that $\mathcal{A}$ is nonzero for all $\Omega$ and, assuming
$c(\Omega) \geq 2c(\chi_1)$, obtain a $g \in \GL_2(F)$ such that
$(\mathcal{A}(f_0))(g) \neq 0$. First observe that we can write $\GL_2(F) = M_2(F) T(F)$, where $M_2(F) = \{\mat{a}{b}{0}{1} : a, b \in F\} \cap \GL_2(F)$ is the mirabolic subgroup of $\GL_2(F)$ and $M_2(F) \cap T(F) = \{1\}$. Hence, the function $\hat{f}$ defined by
\begin{equation}\label{the-test-vector}
\hat{f}(\mat{a}{b}{0}{1} t) = |a|^{1/2}\chi_1(a) \Omega(t) \qquad \text{ for } \mat{a}{b}{0}{1} \in M_2(F), t \in T(F),
\end{equation}
is a well-defined element of $\pi$ and satisfies, for $t \in T(F), \pi(t) \hat{f} = \Omega(t) \hat{f}$, which implies 
\begin{equation}\label{non-van-A}
\mathcal{A}(\hat{f}) \neq 0.
\end{equation}
For the computation of $\mathcal{A}$ applied to the newvector $f_0$, we need to
know when the argument $tg$ of  $f_0$  is in the support of $f_0$ for certain elements $g \in \GL_2(F)$. We obtain that
information in the following lemma.

\begin{lemma}\label{decomp-lemma}
Let $t = t(x,y) \in T(F)$.
 For $s
\in \Z$, we have the following decomposition of $t\mat{\varpi^{s}}{}{}{1}$ as
$bk$ with $b \in B(F)$ and $k \in \GL_2(\OF)$.
\begin{enumerate}
\item If $x-{\bf b}y/2 \in \varpi^{-l}\OF^\times, l \geq 0, {\bf a}\varpi^{s+l}
y \in \OF$, then
$$t\mat{\varpi^{s}}{}{}{1} = \mat{\det(t)\varpi^s/(x-{\bf
b}y/2)}{\varpi^{-l}{\bf c}y/(x-{\bf b}y/2)}{0}{\varpi^{-l}} \mat{1}{0}{-{\bf
a}\varpi^{s+l} y}{\varpi^l(x-{\bf b}y/2)}.$$

\item If $x-{\bf b}y/2 \in \varpi^{-l}\OF^\times, l \geq 0, {\bf a}\varpi^{s+l}
y \not\in \OF$, then
$$t\mat{\varpi^{s}}{}{}{1} = \mat{\det(t)/({\bf a}y)}{-\varpi^s(x+{\bf
b}y/2)}{0}{{\bf a}\varpi^s y} \mat{0}{1}{-1}{(x-{\bf b}y/2)/({\bf
a}\varpi^sy)}.$$

\item If $x-{\bf b}y/2 \in \p, (x-{\bf b}y/2)/(\varpi^s {\bf a}y) \in \OF$, then
$$t\mat{\varpi^{s}}{}{}{1} = \mat{\det(t)/({\bf a}y)}{-\varpi^s(x+{\bf
b}y/2)}{0}{\varpi^s {\bf a}y} \mat{0}{1}{-1}{(x-{\bf b}y/2)/(\varpi^s {\bf
a}y)}.$$

\item If $x-{\bf b}y/2 \in \p, (x-{\bf b}y/2)/(\varpi^s {\bf a}y) \not\in \OF$,
then
$$t\mat{\varpi^{s}}{}{}{1} = \mat{\det(t)\varpi^s/(x-{\bf b}y/2)}{{\bf
c}y}{0}{x-{\bf b}y/2} \mat{1}{0}{-{\bf a}\varpi^sy/(x-{\bf b}y/2)}{1}.$$
\end{enumerate}
\end{lemma}
\begin{proof} The lemma is obtained by direct computation.
\end{proof}

\begin{proposition}\label{A-nonvanish-lemma}
Let $c(\Omega) \geq 2 c(\chi_1)$ and set $s = c(\pi)-c(\Omega)-v({\bf a})$. Then
$(\mathcal{A}(f_0))(\mat{\varpi^s}{}{}{1}) \neq 0$.
\end{proposition}

\begin{proof}  
Since $\Omega |_{F^\times} = \omega_\pi$ and $c(\Omega) \geq 2 c(\chi_1)$, we have $c(\Omega) > 0$. 
Let us first compute the part of the integral $(\mathcal{A}(f_0))(
\mat{\varpi^{s}}{}{}{1})$ over $Z(\OF) \backslash T(\OF)$. The argument of $f_0$
is given by $t\mat{\varpi^{s}}{}{}{1}$, where $t = \mat{x+{\bf b}y/2}{{\bf
c}y}{-{\bf a}y}{x-{\bf b}y/2} \in T(\OF)$, i.e., $y, x-{\bf b}y/2 \in \OF$ and
$x^2-y^2{\bf d}/4 \in \OF^\times$. We write $t\mat{\varpi^{s}}{}{}{1}$ as $bk$
with $b \in B(F)$ and $k \in \GL_2(\OF)$ according to Lemma \ref{decomp-lemma}.
Since $t \in T(\OF)$, we must have $l = 0$ in parts i) and ii) of Lemma
\ref{decomp-lemma}, and $\mathbf a, y \in \OF^\times$ in parts iii) and iv) of
Lemma \ref{decomp-lemma}. 
The support of $f_0$ is $B(F) \gamma_{c(\chi_2)}K_1(\p^{c(\pi)})$ and an element
$k \in \GL_2(\OF)$ lies in the support if and only if the $(2,1)$ entry of $k$
has valuation $c(\chi_2)$ (respectively, $\geq c(\chi_2)$) if $c(\chi_1) > 0$ (if $c(\chi_1)=0$), which is strictly positive. Hence, the $k$ obtained in parts ii) and iii) of
Lemma \ref{decomp-lemma} is never in the support of $f_0$. Since $s <
c(\chi_2)-v({\bf a})$, the $k$ obtained in part iv) of Lemma \ref{decomp-lemma}
is not in the support of $f_0$ as well.  Hence, the only possibility is part i) of Lemma \ref{decomp-lemma}. First suppose that $c(\chi_1) > 0$. By successive change of variable $x \to x + {\bf b}y/2$ and $y \to xy$, we have
\begin{align}
 \int\limits_{Z(\OF) \backslash T(\OF)} & f_0(t\mat{\varpi^{s}}{}{}{1})
\Omega^{-1}(t) \, dt \label{part-of-int}\\
&= \int\limits_{y \in \varpi^{c(\chi_2)-s-v({\bf a})}
\OF^\times}f_0(\mat{\varpi^s(1+{\bf b}y+{\bf a}{\bf c}y^2)}{{\bf c}y}{0}{1}
\mat{1}{0}{-{\bf a}y\varpi^s}{1}) \Omega^{-1}(1+{\bf c}y \beta) \, dy \nonumber\\
&= q^{-s/2}\chi_1(\varpi^s) \int\limits_{y \in \varpi^{c(\chi_2)-s-v({\bf a})}
\OF^\times}\chi_1(1+{\bf b}y+{\bf a}{\bf c}y^2)f_0(\mat{1}{0}{-{\bf
a}y\varpi^s}{1}) \Omega^{-1}(1+{\bf c}y \beta) \, dy \nonumber\\
&= (1-q^{-1})q^{s/2-c(\chi_2)+v({\bf a})}\chi_1(\varpi^s)
\int\limits_{\OF^\times}\chi_1(1+{\bf b}\varpi^{c(\chi_2)-s-v({\bf a})}y+{\bf
a}{\bf c}\varpi^{2(c(\chi_2)-s-v({\bf a}))}y^2) \nonumber \\
& \qquad \qquad \times f_0(\mat{1}{0}{-{\bf a}y\varpi^{c(\chi_2)-v({\bf
a})}}{1}) \Omega^{-1}(1+{\bf c}\varpi^{c(\chi_2)-s-v({\bf a})} y \beta) \, d^\times
y \nonumber
\end{align}
We get the factor $(1-q^{-1})$ above by the normalization of measures. Now, we have 
$$\mat{1}{0}{-{\bf a}y\varpi^{c(\chi_2)-v({\bf a})}}{1} = \mat{-\varpi^{v({\bf
a})}/({\bf a} y)}{0}{0}{1} \gamma_{c(\chi_2)} \mat{-{\bf a} y \varpi^{-v({\bf
a})}}{}{}{1}.$$
Hence the integral (\ref{part-of-int}) is equal to 
\begin{align}\label{m<n2-v(a)}
& (1-q^{-1})q^{s/2-c(\chi_2)+v({\bf a})}\chi_1(\varpi^s)
\int\limits_{\OF^\times}\chi_1(1+{\bf b}\varpi^{c(\chi_2)-s-v({\bf a})}y+{\bf
a}{\bf c}\varpi^{2(c(\chi_2)-s-v({\bf a}))}y^2) \nonumber\\
& \qquad \qquad \times \chi_1(-\varpi^{v({\bf a})}/({\bf
a} y)) \Omega^{-1}(1+{\bf c}\varpi^{c(\chi_2)-s-v({\bf a})} y
\beta) \, d^\times y
\end{align}
Using $c(\chi_2)-s-v({\bf a}) = c(\Omega)-c(\chi_1)  \geq c(\chi_1)$, we get
$$(1-q^{-1})q^{(c(\chi_1)-c(\chi_2)-c(\Omega)+v({\bf
a}))/2}\chi_1(-\varpi^{c(\pi)-c(\Omega)}/{\bf a})
\int\limits_{\OF^\times}\chi_1^{-1}(y)
\Omega^{-1}(1+{\bf c}\varpi^{c(\Omega)-c(\chi_1)}y \beta) \, d^\times y.$$
Since $c(\Omega) \geq 2c(\chi_1)$ the map $y \to
\Omega^{-1}(1+{\bf c}\varpi^{c(\Omega)-c(\chi_1)}y \beta)$ is an additive character of
$\OF$ of conductor $c(\chi_1)$. This character extends to a character $\hat\psi$
of $F$ with conductor $c(\chi_1)$.

Hence, using (\ref{eps-defn-cond}), we get
\begin{align}\label{m0>0-eqn}
& \int\limits_{Z(\OF) \backslash T(\OF)}
f_0(t\mat{\varpi^{c(\pi)-c(\Omega)-v({\bf a})}}{}{}{1}) \Omega^{-1}(t) \, dt
\nonumber \\
& \qquad =(1-q^{-1})
q^{(-c(\chi_2)-c(\Omega)+v({\bf
a}))/2}\chi_1(-\varpi^{c(\chi_2)-c(\Omega)}/{\bf a}) \varepsilon(1/2, \chi_1,
\hat\psi)
\end{align}

If $c(\chi_1)=0$, the integral is much simpler. We get
\begin{align}\label{part-of-int2}
 \int\limits_{Z(\OF) \backslash T(\OF)}  f_0(t\mat{\varpi^{s}}{}{}{1})
\Omega^{-1}(t) \, dt &= \int\limits_{y \in \p^{c(\Omega)}}f_0(\mat{\varpi^s(1+{\bf b}y+{\bf a}{\bf c}y^2)}{{\bf c}y}{0}{1}) \Omega^{-1}(1+{\bf c}y \beta) \, dy \nonumber\\
& = \chi_1(\varpi^s)q^{-s/2-c(\Omega)}.
\end{align}

If $L/F$ is a ramified field extension then we also have to integrate over
$\varpi_L \big(Z(\OF) \backslash T(\OF)\big)$. Let $t = \mat{x+{\bf b}y/2}{{\bf c}y}{-{\bf a}y}{x-{\bf b}y/2} \in \varpi_L T(\OF)$. Hence, we have 
$$x^2-y^2{\bf d}/4 = (x+{\bf b}y/2)(x-{\bf b}y/2) + {\bf a}{\bf c}y^2 \in \varpi \OF^\times.$$ 
We claim that, for $r < c(\chi_2)-v({\bf a})$, the element $t\mat{\varpi^r}{}{}{1}$ is never in the support of $f_0$. We will look at the four possibilities from Lemma \ref{decomp-lemma}. We know that the values of $x, y$ satisfying conditions of parts ii) and iii) never give elements in the support of $f_0$. 
\begin{itemize}
\item Suppose $x-{\bf b}y/2 \in \varpi^{-l}\OF^\times$ with $l \geq 0$ and ${\bf a}y \varpi^{r+l} \in \OF$. To prove the claim, it is enough to show that $v(y) \leq -l$. Suppose $y \in \p^{-l+1}$. Then we have $\varpi^l y \in \p$. By assumption, we have $\varpi^l (x-{\bf b}y/2) \in \OF^\times$. Hence $\varpi^l (x+{\bf b}y/2) \in \OF^\times$. But then we get $\varpi^{2l}(x^2-y^2 {\bf d}/4) \in \OF^\times$, which is a contradiction.

\item Suppose $x-{\bf b}y/2 \in \p, (x-{\bf b}y/2)/(\varpi^r {\bf a}y) \not\in \OF$. To prove the claim, it is enough to show that $v(y) \leq 0$.
Suppose $y \in \p$. Then $x+{\bf b}y/2 \in \p$. But then we get $(x^2-y^2 {\bf d}/4) \in \p^2$, a contradiction.
\end{itemize}
Hence, for $r < c(\chi_2)-v({\bf a})$, we have 
\begin{equation}\label{no-supp-eqn}
\int\limits_{\varpi_L (Z(\OF) \backslash T(\OF))} f_0(t\mat{\varpi^r}{}{}{1}) \Omega^{-1}(t) dt = 0.
\end{equation}
This completes the proof of the proposition by observing that $s < c(\chi_2)-v({\bf a})$.
\end{proof}

Observe that in the above proof we have used $c(\Omega) \geq 2 c(\chi_1)$ at two crucial steps which simplifies the integral. In the case $c(\Omega) < 2 c(\chi_1)$, it is not clear if the statement of the proposition still remains valid. 

\medskip
\noindent
{\bf Proof of Theorem \ref{intro-tv-inert} for principal series representations}: By the definition (\ref{bessel-intertwiner-defn}) of $\mathcal A$ and (\ref{non-van-A}), the linear functional on $\pi$ given by $ \ell(f) = (\mathcal{A}(f))(1)$ is a nonzero
functional satisfying $\ell(\pi(t)f) = \Omega(t) \ell(f)$ for all $t \in T(F)$
and $f \in \pi$. Hence, $\mathrm{Hom}_{T(F)}(\pi, \Omega) \neq 0$. The one-dimensionality follows from \cite{wald}. Since $c(\Omega) \geq c(\pi)$, we can apply Lemma \ref{toric-decomp-lemma} together with Proposition \ref{A-nonvanish-lemma} to obtain the existence of the required test vector. The uniqueness follows from the uniqueness of $f_0$. \qed

\subsection{Steinberg representation}
Let $\pi = \chi |\cdot|^{1/2} \times \chi |\cdot|^{-1/2}$. 
Let $V_0$ be the unique invariant (infinite-dimensional) subspace of $\pi$, so $\pi|_{V_0}$
is the twisted Steinberg representation $\chi \St_{\GL_2}$.
If we set
$\chi_1 = \chi |\cdot|^{1/2}$ and $\chi_2 = \chi |\cdot|^{-1/2}$, then $V_0$ is
characterized as the kernel of the intertwining operator $M : \chi_1 \times
\chi_2 \to \chi_2 \times \chi_1$, given by
$$(M(f))(g) = \int\limits_F f(\mat{}{-1}{1}{} \mat{1}{x}{}{1} g) \, dx.$$

\subsection*{$\chi$ ramified}
If $\chi$ is a ramified character, then
$f_0$, defined as in (\ref{new-form-formula}), is in $V_0$ and is, in fact, the
unique (up to constant) newform in $\chi \St_{\GL_2}$ (see \cite{Sc}). Hence the proof of Proposition \ref{A-nonvanish-lemma} is valid in this case
without any modification.

\subsection*{$\chi$ unramified}
 If $\chi$ is unramified, then the vector $f_0$, defined as in
(\ref{new-form-formula}), is a spherical vector in $\chi_1 \times \chi_2$, hence
clearly not the newform of $\chi \St_{\GL_2}$, which has conductor $\p$. Any
vector in $\chi \St_{\GL_2}$ which is right $K_1(\p)$-invariant is also right
$\I$-invariant, where $\I$ is the Iwahori subgroup defined in
(\ref{Iwahori-sub}).  
It is known (see \cite{Sc}) that the newform in the induced model is given
by
\begin{equation}\label{new-form-unr-st}
f_0(g) = \begin{cases} |a/d| \chi(ad) q & \text{ if } g \in \mat{a}{\ast}{}{d}
\I, \\
-|a/d| \chi(ad) & \text{ if } g \in \mat{a}{\ast}{}{d} w \I.\end{cases}
\end{equation}

We can try to compute $(\mathcal A(f_0))(g)$ (defined in (\ref{bessel-intertwiner-defn}))
in this case for various values of $g$. But instead, we will use a double coset
decomposition and properties of the Steinberg representation to obtain the test
vector. This has the added advantage of obtaining a new proof of the uniqueness (up to constant) of the Waldspurger functional and also gives us the explicit formula for $B(g)$, where $B$ is the newform in the corresponding Waldspurger model and $g$ is any element of $\GL_2(F)$. By \cite{Su}, Lemma 2-4, there is the disjoint double coset
decomposition
\begin{equation}\label{toricdecompositioneq}
 \GL_2(F)=\bigsqcup_{r=0}^\infty T(F)\mat{\varpi^r}{}{}{1}\GL_2(\OF).
\end{equation}
Note that, by the Iwasawa decomposition of $\SL_2(\OF/\p)$, we have
\begin{equation}\label{single-coset-decomp}
\GL_2(\OF) = w \I \sqcup\bigsqcup_{u\in\OF/\p}\mat{1}{}{u}{1} \I, \qquad w =
\mat{}{1}{-1}{}.
\end{equation}
For $u \in \OF$ and $r \geq 0$ set $\beta_{u,r} := {\bf a}\varpi^{2r} + {\bf b}
\varpi^r u + {\bf c} u^2$. Arguing as in Lemma 3.1 of \cite{P}, we have
\begin{equation}\label{eqn2}
T(F) \mat{\varpi^r}{}{}{1} w \I = T(F) \mat{\varpi^r}{}{}{1} \mat{1}{}{u}{1} \I
\iff \beta_{u,r} \in \OF^\times.
\end{equation}
Lemma 3.2 of \cite{P} tells us exactly when $\beta_{u,r} \in \OF^\times$.
Putting everything together, we get
\begin{proposition}\label{Stein-doub-decomp-prop}
For $r > 0$, we have
$$T(F)\mat{\varpi^r}{}{}{1}\GL_2(\OF) = T(F)\mat{\varpi^r}{}{}{1} \I \sqcup
T(F)\mat{\varpi^r}{}{}{1} w \I.$$
For $r = 0, \Big(\frac L{\p}\Big) = -1$, we have
$$T(F) \GL_2(\OF) = T(F) \I = T(F) w \I.$$
For $r = 0, \Big(\frac L{\p}\Big) = 0$, we have
$$T(F) \GL_2(\OF) = T(F) w \I \sqcup T(F) \mat{1}{}{u_0}{1} \I,$$
where $u_0$ is chosen as in (\ref{u0defeq}).
\end{proposition}
The twisted Steinberg representation is characterized as the representation
$\pi$ with a newform $v_0$ which is invariant under $\I$ and satisfies the
following two conditions.
$$\sum\limits_{\gamma \in \GL_2(\OF)/\I}\pi(\gamma) v_0 = 0, \qquad
\pi(\mat{}{1}{\varpi}{})v_0 = -\chi(\varpi)v_0.$$
These conditions follow from the action of the Atkin-Lehner element and the fact
that $\pi$ does not have a vector invariant under $\GL_2(\OF)$. See Proposition 3.1.2 of \cite{Sc}.  Let $\Omega$ be
a character of $L^\times$ with $\Omega |_{F^\times} = \omega_\pi$.

Let $B : \GL_2(F) \rightarrow \C$ be a function that satisfies $B(tgk) =
\Omega(t)B(g)$ for all $t \in T(F)$, $g \in \GL_2(F)$, $k \in \I$ and 
\begin{align}
&\sum\limits_{u \in \OF/\p} B(g \mat{1}{}{u}{1}) = -B(gw),
\label{new-form-cond}\\
&B(g\mat{}{1}{\varpi}{}) = -\chi(\varpi)B(g) \label{AL-cond}
\end{align}
for all $g \in \GL_2(F)$.
If $\pi$ has a $\Omega$-Waldspurger model, then $B$ will precisely be the  unique
(up to scalars) newform of $\pi$ in the $\Omega$-Waldspurger model; otherwise $B$ will be 0.
\begin{lemma}\label{various-values-lem} 
\begin{enumerate}
\item If $c(\Omega) \geq 2$, then
\begin{equation}
 B(\mat{\varpi^r}{}{}{1}w) = 0 \qquad \text{ for } r \leq c(\Omega)-2.
\label{mw-tower-aut-van}
\end{equation}

\item For $r > 0$, we have
\begin{equation}\label{m-tower-inmw-tower}
B(\mat{\varpi^r}{}{}{1}) = \begin{cases} -q B(\mat{\varpi^r}{}{}{1}w) & \text{
if } r \geq c(\Omega),\\ 0 & \text{ if } r < c(\Omega).\end{cases}
\end{equation}

\item For $r \geq {\rm max}\{c(\Omega)-1, 0\},$ we have
\begin{equation}\label{up-down-tower}
B(\mat{\varpi^{r+1}}{}{}{1}w) = \frac{\chi(\varpi)}{q}
B(\mat{\varpi^r}{}{}{1}w).
\end{equation}

\item If $\Big(\frac L{\p}\Big) = 0$, then
\begin{equation}\label{ram-u0-cond}
B(\mat{1}{}{u_0}{1}) = \begin{cases} -q B(w) & \text{ if } c(\Omega) = 0,\\ 0 & \text{
if } c(\Omega) > 0.\end{cases}
\end{equation}

\item If $c(\Omega) = 0$ and $\Omega = \chi \circ N_{L/F}$,  then 
\begin{equation}\label{van-cond}
B(w) = 0.
\end{equation}
\end{enumerate}
\end{lemma}
\begin{proof} We will illustrate the proof of i) and ii) in detail here. Let $u, v \in \p^{c(\Omega)-1}$ such that $\Omega(1+u+v\beta) \neq 1$. Take
$y=v/\mathbf c, x=1+u+{\bf b}y/2$ and, for $r \leq c(\Omega)-2$, let 
$$k = \mat{1+u}{{\bf a}/{\bf c}v\varpi^r}{-\varpi^{-r}v}{1+u+{\bf b}/{\bf c}v}
\in \I.$$
Then
$$B(\mat{\varpi^r}{}{}{1}w) = B(\mat{\varpi^r}{}{}{1} wk) =
B(t(x,y)\mat{\varpi^r}{}{}{1} w) = \Omega(1+u+v\beta)
B(\mat{\varpi^r}{}{}{1}w).$$
This gives us (\ref{mw-tower-aut-van}) and completes the proof of i). 

Next, we will give the proof of ii). Substitute $g = \mat{\varpi^r}{}{}{1}$ in (\ref{new-form-cond}) to get
$$\sum\limits_{u \in \OF/\p} B(\mat{\varpi^r}{}{}{1} \mat{1}{}{u}{1}) = -B(\mat{\varpi^r}{}{}{1}w).$$
For $u \neq 0$, setting $x={\bf b}/2\varpi^r+{\bf c}u, y=\varpi^r$, we get
$$\mat{\varpi^r}{}{}{1} \mat{1}{}{u}{1} \mat{-{\bf c}}{{\bf b}\varpi^r+{\bf c}u}{}{-\beta_{u,r}} = t(x,y) \mat{\varpi^r}{}{}{1} w.$$
Since $r > 0$ by assumption, $\beta_{u,r} \in \OF^\times$. Hence, for $u \neq 0$ we have
$$B(\mat{\varpi^r}{}{}{1} \mat{1}{}{u}{1}) = \Omega(u+\varpi^r\beta) B(\mat{\varpi^r}{}{}{1}w).$$
This gives us 
$$B(\mat{\varpi^r}{}{}{1}) = -\Big(\sum\limits_{u \in (\OF/\p)^\times} \Omega(u+\varpi^r\beta) + \Omega(1)\Big) B(\mat{\varpi^r}{}{}{1}w).$$
Using (\ref{mw-tower-aut-van}) and the definition of $c(\Omega)$ we get the result for $r \geq c(\Omega)$ or $r \leq c(\Omega)-2$. For $r=c(\Omega)-1$, using Lemma 3.4 of \cite{P}, we see that the expression in the parentheses on the right hand side above is $0$. This completes the proof of ii).

Using (\ref{new-form-cond}), (\ref{AL-cond}), and similar calculations as above, we get the remaining results.
\end{proof}

\noindent
{\bf Proof of Theorem \ref{intro-tv-inert} for twists of Steinberg representations:} 
Let $D$ be the quaternion division algebra over $F$ and $N_{D/F}$
be the reduced norm.  Since $\pi = \chi \St$ corresponds to the one-dimensional
representation 
$\pi' = \chi \circ N_{D/F}$ of $D^\times(F)$, one knows by \cite{wald}
that $\pi$
has an $\Omega$-Waldspurger model if and only if $\Omega \ne \chi \circ
N_{L/F}$.  Since $c(\Omega) \geq c(\pi)$, this must be the case, i.e., $\dim_\C \mathrm{Hom}_{T(F)}(\pi, \Omega) =1$. The $\chi$ ramified case follows exactly as in the
principal series case. For $\chi$ unramified, the result follows from Lemma
\ref{various-values-lem}. \qed

\section{Supercuspidal representations}
\label{sc-tv-sec}

Throughout this section we continue to assume that $L/F$ is a field.

\subsection{Chain orders and strata}\label{chainorders}

This section contains a summary of the facts about chain orders and fundamental
strata that we will use to construct test vectors for the supercuspidal
representations $\pi$ of $\GL_2(F)$, all of which can be found in Chapter 4
of Bushnell--Henniart~\cite{BH}. 

Let $\mathfrak{A}$ be a chain order in $\M_2(F)$. Up to
$\GL_2(F)$-conjugacy, $\mathfrak{A}$ is either $\mathfrak{M}=\M_2(\mathfrak{o})$ or
$\mathfrak{J}=\begin{bmatrix} \mathfrak{o} & \mathfrak{o} \\ \mathfrak{p} &
\mathfrak{o} \end{bmatrix}$, so we always take $\mathfrak{A}$ to be 
$\mathfrak{M}$ or $\mathfrak{J}$.

Write $e_\mathfrak{A}=1$ if $\mathfrak{A} = \mathfrak{M}$ and
$e_\mathfrak{A}=2$ if $\mathfrak{A} = \mathfrak{J}$. For more
intrinsic definitions, see~\cite{BH}.
Let $\mathfrak{P}=\mathrm{rad} \, \mathfrak{A}$, the Jacobson radical of
$\mathfrak{A}$. There is an element $\Pi \in \GL_2(F)$ such that $\mathfrak{P}=\Pi \mathfrak{A}$,
and one has
\begin{equation*}
 \mathrm{rad} \, \mathfrak{M} = \varpi \mathfrak{M}, \quad \mathrm{rad} \,
\mathfrak{J}=\begin{bmatrix} & 1 \\  \varpi & \end{bmatrix} \mathfrak{J}. 
\end{equation*}
Let $\mathfrak{P}^n = \Pi^n \mathfrak{A}$ for $n \in \mathbb{Z}$. Let
$U^0_\mathfrak{A}=U_\mathfrak{A}:=\mathfrak{A}^\times$, 
$U^n_\mathfrak{A}:=1+\mathfrak{P}^n$ for $n\geq 1$, and $K_\mathfrak{A}=\{ g \in \GL_2(F)
: g \mathfrak{A} g^{-1}=\mathfrak{A} \}$. Then
\begin{equation*}
 K_\mathfrak{A}=
\begin{cases}
Z(F) \GL_2(\mathfrak{o}) & \text{if } \mathfrak{A}=\mathfrak{M},\\
\left< \begin{bmatrix} & 1\\ \varpi & \end{bmatrix} \right> \ltimes
\mathfrak{J}^\times & \text{if } \mathfrak{A}=\mathfrak{J}.
\end{cases}
\end{equation*}
We fix a character $\psi_1: F
\rightarrow \mathbb{C}^\times$ so that the conductor of $\psi_1$ is $\p$. For
$\alpha \in \mathrm{M}_2(F)$, define a function of $U_\mathfrak{A}$ by
$\psi_\alpha (x)=\psi_1(\mathrm{Tr} \, \alpha(x-1))$. Then  for $1 \leq m \leq n \leq 2m$,
there is an isomorphism
\begin{align*}
 \mathfrak{P}^{-n} / \mathfrak{P}^{-m} &\rightarrow (U^{m+1}_\mathfrak{A} /
U^{n+1}_\mathfrak{A})^{\wedge},\\
\alpha+\mathfrak{P}^{-m} &\mapsto \psi_\alpha.
\end{align*}
The normalized level of $\pi$ is defined to be 
\begin{equation*}
 \ell(\pi) = \min \{ n/e_\mathfrak{A} : \pi|_{U_\mathfrak{A}^{n+1}} \, \,
\text{contains the trivial character}\}.
\end{equation*}

A stratum in $\mathrm{M}_2(F)$ is a triple $(\mathfrak{A}, n, \alpha)$ where
$\mathfrak{A}$ is a chain order in $\mathrm{M}_2(F)$ with radical
$\mathfrak{P}$, $n$ is an integer, and $\alpha \in \mathfrak{P}^{-n}$. For $n
\geq 1$ one associates to a stratum the character $\psi_\alpha$ of
$U_\mathfrak{A}^n$ which is trivial on $U_\mathfrak{A}^{n+1}$.

We say that a smooth representation $\pi$ contains the stratum  $(\mathfrak{A},
n, \alpha)$ if $\pi|_{U^{n}_\mathfrak{A}}$ contains $\psi_\alpha$. 
A fundamental stratum is one such that $\alpha + \mathfrak{P}^{1-n}$ contains no
nilpotent elements. If an irreducible smooth representation $\pi$ of $\GL_2(F)$
contains a stratum $(\mathfrak{A}, n, \alpha)$, then  $(\mathfrak{A}, n,
\alpha)$ is fundamental if and only if $\ell(\pi) = n/e_\mathfrak{A}$
~\cite{BH}*{12.9 Theorem}. 

Suppose that  $(\mathfrak{A}, n, \alpha)$ is a fundamental stratum with
$e_\mathfrak{A}=1$. Write $\alpha=\varpi^{-n} \alpha_0$ for $\alpha_0 \in
\mathfrak{A}$. Let $f_\alpha(t) \in \OF[t]$ be the characteristic polynomial of
$\alpha_0$, and let $\tilde{f}_\alpha \in \mathbf{k}[t]$ be its reduction modulo
$\p$. Here $\mathbf{k}$ is the residue class field. 
If $\tilde{f}_\alpha$ has two solutions in $\mathbf{k}$, then $(\mathfrak{A}, n,
\alpha)$ is said to be a split fundamental stratum.
If $\tilde{f}_\alpha$ is irreducible, then the stratum  $(\mathfrak{A}, n,
\alpha)$ is said to be unramified simple. 
On the other hand, if  $(\mathfrak{A}, n, \alpha)$ is a fundamental stratum with
$e_\mathfrak{A}=2$, and $n$ odd, then  $(\mathfrak{A}, n, \alpha)$ is said to be
ramified simple.
A simple stratum is either a simple unramified stratum or a simple ramified
stratum. Suppose that $(\mathfrak{A}, n, \alpha)$ is a simple stratum with $\alpha_0$ as above and let $E=F[\alpha_0]$. Bushnell-Henniart define what it means for $\alpha$ to be minimal (\cite{BH}*{13.4 Definition}), and when this is the case $\OF_E=\OF[\alpha_0]$ (\cite{BH}*{13.4 Lemma}).  If $(\mathfrak{A}, n, \alpha)$ is a simple stratum with $\mathfrak{A}=\mathfrak{M}$, then $\alpha_0 \in \mathfrak{M}$ but $\alpha_0 \notin \mathfrak{P}$.

Define $\pi$ to be \emph{minimal} if, for all characters $\chi$ of $F^\times$, 
$\ell(\pi \otimes \chi) \geq \ell(\pi)$. Every irreducible supercuspidal
representation of $\GL_2(F)$ is either minimal, or isomorphic to the twist of a
minimal irreducible supercuspidal representation.
Every minimal irreducble smooth representation of $\GL_2(F)$ contains exactly one of
the following: a ramified simple stratum, an unramified simple stratum, or a
split fundamental stratum~\cite{BH}*{13.3 Corollary}.
 If $\pi$ contains a split fundamental stratum, then $\pi$ is not
supercuspidal. 

\subsection{Construction of minimal supercuspidals}
\label{construction}

In this section we review the construction of minimal irreducible supercuspidal
representations. See \cite{BH}*{Section 19} for more details. In each case we
describe a distinguished vector $v_0$ in the inducing representation. This
vector $v_0$ will be used to construct a test vector for $\pi$.

We remark that if a representation $\pi$ contains a simple stratum
$(\mathfrak{A}, n, \alpha)$, then it contains all $\GL_2(F)$ conjugates of
$(\mathfrak{A}, n, \alpha)$. Therefore, we may always take $\mathfrak{A}$ to be
either $\mathfrak{M}$ or $\mathfrak{J}$. Since $K_\mathfrak{A}$ normalizes
$U_\mathfrak{A}$, we may also consider $\alpha$ up to $K_\mathfrak{A}$
conjugacy.

For the rest of Section~\ref{sc-tv-sec} we assume that all supercuspidal representations are irreducible.
\subsubsection{$\mathfrak{A}=\mathfrak{M}$, $\ell(\pi)=2r+1$} \label{oddM}
Suppose that $\pi$ is a minimal supercuspidal representation containing the
simple stratum given by ${( \mathfrak{M}, 2r+1, \alpha)}$. Then $E=F[\alpha]$ is an
unramified quadratic extension of $F$, and $\pi \cong c \mbox{--}
\mathrm{Ind}_{J_\alpha}^{\GL_2(F)} \lambda$, where $J_\alpha=E^\times
U_\mathfrak{M}^{r+1}$ and $\lambda$ is a character.

We have that $\lambda|_{U_{\mathfrak{M}}^{r+1}} =\psi_\alpha$ with $\alpha \in
\mathfrak{P}_\mathfrak{M}^{-2r-1}$ and $\alpha$ is minimal.
One may take $\alpha_0= \varpi^{2r+1} \alpha$ to be in rational canonical form,
i.e.
\begin{equation}\label{rationalcanonical}
 \alpha_0= \begin{bmatrix}0 & 1\\ a_0 & a_1 \end{bmatrix},
\end{equation}
for $a_i \in \OF$, $i=0,1$.
Then
 $$ 1+ \begin{bmatrix} \p^{r+1} & \p^{2r+2} \\ \p^{2r+2} & \p^{2r+2}
\end{bmatrix}  \subseteq \ker \psi_\alpha .$$

\subsubsection{$\mathfrak{A}=\mathfrak{J}$, $\ell(\pi)=\frac{2r+1}{2}$}
\label{oddJ}
Suppose that $\pi$ is a minimal supercuspidal representation containing the
simple stratum $( \mathfrak{J}, 2r+1, \alpha)$, and let $E=F[\alpha]$.
In this case $E/F$ is a ramified extension, and  $\ell(\pi) e(E/F)=2r+1$. Then
$\pi=c \mbox{--}\mathrm{Ind}_{J_\alpha}^{\GL_2(F)}\lambda$ where
$J_\alpha=E^\times U_\mathfrak{J}^{r+1}$ and $\lambda$ is a character.
Observe
\[ U_\mathfrak{J}^{r+1} = 1 + {\mathfrak P}^{r+1} =
\begin{cases} 
1+\bmx {\mathfrak p}^{r/2+1} & {\mathfrak p}^{r/2} \\ {\mathfrak p}^{r/2+1} & {\mathfrak p}^{r/2+1} \emx & \text{if } r \text{ is even} \\
1+\bmx {\mathfrak p}^{(r+1)/2} & {\mathfrak p}^{(r+1)/2} \\ {\mathfrak p}^{(r+3)/2} & {\mathfrak p}^{(r+1)/2} \emx & \text{if } r \text{ is odd}. 
\end{cases} \]

Note that 
\begin{equation*}
 U_\mathfrak{J}^{2r+2}=1+ \begin{bmatrix} \p^{r+1} & \p^{r+1}\\ \p^{r+2}
& \p^{r+1} \end{bmatrix}  \subseteq \ker \lambda.
\end{equation*}
Let $\alpha_0=\varpi^{r+1} \alpha$ be of the form \eqref{rationalcanonical},
where now $a_0 \in \varpi \OF^\times$ and $a_1 \in \p$, and $k:=\left[\frac{r}{2}\right]+1$.  
Then
\begin{equation*}
 1+ \begin{bmatrix} \p^{k} & \p^{r+1}\\ \p^{r+2} & \p^{r+1}
\end{bmatrix}  \subseteq \ker \lambda.
\end{equation*}

\subsubsection{$\mathfrak{A}=\mathfrak{M}$, $\ell(\pi)=2r > 0$}\label{evenlevel}
Now suppose $\pi$ contains an unramified simple stratum $(\mathfrak{M}, 2r,
\alpha)$ for some $\alpha \in \mathfrak{P}^{-2r}$ so that $ \ell(\pi)=2r>0$ and $e(E/F)=1$,
where as before $E=F[\alpha]$.   Continue to assume that $\alpha_0 = \varpi^{2r}\alpha$ has the form \eqref{rationalcanonical}. In
this case, $\pi$ is not induced from a character, and $E$ is a
unramified quadratic extension of $F$. We will describe a representation $\rho$
of $J_\alpha=E^\times U_\mathfrak{M}^{r}$ so that $\pi$ is compactly induced
from $\rho$, and we follow Kutzko \cite{Kutzko2}*{\S 1} since his construction is more convenient for our applications. 

Write $U_E^1=U_\mathfrak{M}^1 \cap E^\times$. Since $\alpha_0 \in \mathfrak{M} \smallsetminus \mathfrak{P}$ and $\OF_E = \OF[\alpha_0]$ (see Section~\ref{chainorders}), a simple argument shows $U_E^1 \subset 1+\p_E$. The opposite inclusion is obvious; therefore, $U_E^1 =1+\p_E$.  We similarly note that $E^\times \cap U_\mathfrak{M}=:U_E \cong \OF_E^\times$.
There is a character $\chi$ of $E^\times$ such that $\chi(1+x)=\psi_1 \circ
\mathrm{Tr}_{E/F}(\alpha x)$ for all $x \in \p_E^{r+1}$. Define a character
$$\lambda: H_\alpha^1:=U_E^1 U_\mathfrak{M}^{r+1} \rightarrow \C^\times$$ by
$\lambda(ux)=\chi(u)\psi_\alpha(x)$ for $u \in U_E^1$ and $x \in
U_\mathfrak{M}^{r+1}$. 

Let $A=\left\{ \begin{bmatrix} x & \\ & 1 \end{bmatrix} : x \in F^\times \right\}$, and set $A^n = A \cap U_\mathfrak{M}^n$ for $n \geq 0$.
The character $\lambda$ can be extended to a character $\tilde{\lambda}$ of $A^r
H_\alpha^1$ by $\tilde{\lambda}(yx)=\lambda(x)$ for $y \in A^r
U_\mathfrak{M}^{2r+1}$ and $x \in H_\alpha^1$ (\cite{Kutzko2}*{Definition 1.8}).
 
Let $J_\alpha^1=U_E^1 U_\mathfrak{M}^r$. Define $\eta= \mathrm{Ind}_{A^r
H^1_\alpha}^{J_\alpha^1} \tilde{\lambda}$. Then $\eta$ is an irreducible
representation of $J_\alpha^1$ of dimension $q$. There is an irreducible
representation $\rho$ of $J_\alpha$ such that $\pi \cong
\mathrm{Ind}_{J_\alpha}^{\GL_2(F)} \rho$, and $\rho|_{J_\alpha^1}  \cong
\eta$~\cite{Kutzko2}*{Lemma 1.10 and Proposition 1.15}.
Note that $U_\mathfrak{M}^{2r+1} \subset \ker \rho$. We must compute
$\rho|_{A^r} \cong \eta|_{A^r}$. We have $[J^1_\alpha : A^r H^1_\alpha]=q$, and an
irredundant set of coset representatives is given by $\left\{ \begin{bmatrix} 1
& \\a & 1 \end{bmatrix} : a \in \p^r / \p^{r+1} \right\}$. 

It is a simple computation to show that $\eta|_{A^r} = \mathrm{Ind}_{A^r H^1_\alpha}^{J^1_\alpha}
\tilde{\lambda}|_{A^r}$ is isomorphic to the regular representation of
$A^r/A^{r+1}$. In particular it contains the trivial character with multiplicity
one, so there is vector $v_0 \in \rho$ that is unique up to scalars with the
property that it is fixed by $1+\begin{bmatrix} \p^r & \p^{2r+1} \\
\p^{2r+1} & \p^{2r+1} \end{bmatrix}$. 

Sometimes it will be convenient to consider the corresponding vector $f_0 \in
\eta$ given by
\begin{equation}\label{f0}
f_0(k)=\begin{cases}\tilde{\lambda}(k) & \text{if } k \in A^r H^1_\alpha, \\
0 & \text{otherwise}.
\end{cases}
\end{equation}

\subsubsection{Depth zero supercuspidals}\label{depthzero}
Now, consider a depth zero supercuspidal representation, i.e. $\ell(\pi)=0$.
Then $\pi$ is induced from a representation $\rho$ of $K_\mathfrak{M}$ that is
inflated from a cuspidal representation $\tilde{\rho}$ of $\GL_2(\OF / \p)$, i.e.\
$\rho$ is trivial on $U_\mathfrak{M}^1=1+ \p \mathrm{M}_2(\OF)$ and it factors
through $\tilde{\rho}$. The cuspidal representations $\tilde{\rho}$ are
parameterized by Galois conjugacy classes of regular characters
$\theta: \mathbb{F}_{q^2}^\times \rightarrow
\mathbb{C}^1$.
 Such a character $\theta$ can also be regarded as a character of $\OF_E^\times$
that is trivial on $1+\p_E$, where $E/F$ is the unique unramified quadratic
extension. Embed $\OF_E^\times$ in $\GL_2(\OF)$, and identify $\mathbb{F}_{q^2}^\times$ with the image of $\OF_E^\times$ under the reduction
map modulo $\p$. The following
proposition gives the character table for $\tilde{\rho}$ which is a well-known
result (see, e.g., \cite{BH}*{6.4.1}).
 \begin{proposition}\label{char-table-depth-zero}
 The character table of $\tilde{\rho}$ is given by
\begin{align*}
 \mathrm{Tr} \, \tilde{\rho}(z)  = & (q-1) \theta(z), \,  z \in Z;\nonumber\\
 \mathrm{Tr} \, \tilde{\rho}(zu)  =& -\theta(z), \,  z\in Z, \, u\in N, \, u\neq
1;\nonumber\\
 \mathrm{Tr} \, \tilde{\rho}(y)  = & - ( \theta(y)+ \theta^q(y) ),  \, y \in
\mathbb{F}_{q^2}^\times \smallsetminus Z. 
\end{align*}
If $g$ is not conjugate to an element of $\mathbb{F}_{q^2}^\times \cup ZN$, then
$\mathrm{Tr} \, \tilde{\rho} (g)=0$. 
\end{proposition}
From the character table one sees that the restriction of $\rho$ to $A^0$ is isomorphic to the regular representation of $A^0/A^1$. In particular there is a vector $v_0 \in
\rho$ such that, for $a \in A^0$, $\rho(a)v_0=v_0$.

\subsection{Remarks on minimal supercuspidals}
\label{sec:minsc-rem}

We consider a minimal
supercuspidal representation $\pi=c \mbox{--}\mathrm{Ind}_{J_\alpha}^{\GL_2(F)}
\rho$, where $E=F[\alpha]$ for $\alpha \in \mathfrak{P}^{-n}$ such that $\pi$
contains the simple stratum $( \mathfrak{A}, n, \alpha)$. In the case when
$e(E/F) \ell(\pi)$ is odd, then $\rho=\lambda$ is a character  which restricts
to $\psi_\alpha$. When $e(E/F) \ell(\pi)$ is even, then $\rho$ is not a
character, and if additionally $\ell(\pi) > 0$ we will sometimes identify $\rho|_{J_\alpha^1}$ with $\eta$ as
described above.

From the discussion in the previous section, we may always take $J=J_\alpha$ of
the form $E^\times(1+ \mathfrak P^{[\frac{e_{\mathfrak A} \ell(\pi)+1}2]} )^\times$, where we
take $E=F$ if $\ell(\pi) = 0$.  In all cases, we have $E^\times \subset K_{\mathfrak A}$ so
$J \subset K_{\mathfrak A}$.

\begin{definition} \label{v0def}
Suppose $\pi=c\mbox{--}\mathrm{Ind}_J^{\GL_2(F)} \rho$ is an irreducible minimal
supercuspidal representation. 

\begin{enumerate}
 \item If $\rho=\lambda$ is a character, define $v_0$ to be the unique vector up to scalar multiple in $\rho$. That is $\rho(k)v_0=\lambda(k)$ for $k
\in J$. Then according to the constructions in Sections~\ref{oddM} and
\ref{oddJ}, for $a \in A \cap J$, $\rho(a)v_0=\lambda(a)=1$.
 \item Suppose $\dim_\C \rho >1$. Define $v_0 \in \rho$ to be the vector
described in Sections~\ref{evenlevel} and \ref{depthzero} such that, for
$a \in A \cap J$, $\rho(a)v_0=v_0$. 
\end{enumerate}
\end{definition}
Let $N= \left\{ \begin{bmatrix} 1 & x \\ & 1\end{bmatrix} \right\} \subset
\GL_2(F)$,  $\overline{N}= \left\{ \begin{bmatrix} 1 &  \\x & 1\end{bmatrix}
\right\} \subset \GL_2(F)$, and $\overline{N}^r = \overline{N} \cap U_\mathfrak{M}^r$ for $r \geq 0$.

\begin{lemma} \label{restrictionlemma}
 Suppose $\pi$ is a minimal supercuspidal representation and $\ell(\pi)=2r$.  Write
$\pi= c \mbox{--} \mathrm{Ind}_{J}^{\GL_2(F)} \rho$ where $\rho$ is not a
character. 
 \begin{enumerate}
  \item If $\ell(\pi)=0$, then $\rho|_{\overline{N}^0}= \bigoplus
\limits_{i=1}^{q-1} {\psi_i}$ where $\psi_i$ runs over all the nontrivial
characters of $\overline{N}^0/\overline{N}^1$.

 \item If $\ell(\pi)>0$ and $J=J_\alpha$, then $\rho|_{\overline{N} \cap
J_\alpha}= \bigoplus \limits_j {\psi_j}$, where the sum runs over all
characters $\psi_j$ of $\overline{N} \cap J_\alpha$ so that
$\psi_j|_{\overline{N} \cap H_\alpha}  =\psi_\alpha|_{\overline{N} \cap
H_\alpha}$, and $H_\alpha:=E^\times U_\mathfrak{M}^{r+1}$.
 \end{enumerate}
\end{lemma}
\begin{proof}
The first part may be deduced from Proposition~\ref{char-table-depth-zero}. Now, suppose that $\ell(\pi)>0$. Since $J_\alpha= E^\times U_\mathfrak{M}^r$,
we have $\overline{N} \cap
J_\alpha=\overline{N}^r$ and similarly $\overline{N} \cap H_\alpha=\overline{N}^{r+1}$.
Also, recall that $\rho|_{J_\alpha^1} \cong \eta= \mathrm{Ind}_{A^r
H_\alpha^1}^{J_\alpha^1} \tilde{\lambda}$ where $\tilde{\lambda}$ is obtained from
$\lambda$ by extending it trivially to $A^r$. 

A set of irredundant coset representatives for $A^r H_\alpha^1 \backslash
J_\alpha^1$ is given by  $\left\{ \begin{bmatrix} 1 &  \\ a & 1\end{bmatrix} | a
\in  \p^r / \p^{r+1} \right\}$.
Let $\psi'$ be one of the characters $\psi_j$ of $\overline{N}^r$ as in ii).
For $a \in \p^r / \p^{r+1}$, define
\begin{equation}
 f_a(y)= \begin{cases}
 \psi'(\bmx 1 & \\ a & 1 \emx ) \tilde{\lambda}(x)  & \text{if } y=x \begin{bmatrix}1 & \\a &
1\end{bmatrix}, \quad x\in A^r H_\alpha^1, \\
0 & \text{otherwise.}
\end{cases} \label{fa}
\end{equation}
This is well defined since $\psi'$ and $\tilde{\lambda}$ agree on
$\overline{N}^{r+1}$. Note the $f_a$ span $\eta$, and when $a=0$, \eqref{fa} agrees with
\eqref{f0}. Define $f_{\psi'}= \sum
\limits_{a \in \p^r/\p^{r+1} } f_a$.  Then we have
\begin{equation*}
 \eta ( \begin{bmatrix} 1 &  \\x & 1 \end{bmatrix} ) f_{\psi'} =
\psi'(x) f_{\psi'}.
\end{equation*}
From the explicit basis we computed for $\eta$ one sees that each of these
characters appear with multiplicity one.
This proves the lemma. 
\end{proof}

Later, it will be useful to have a case-by-case description of the kernel of $\rho$.
We summarize what we know about the kernel from the previous section
in the following table.

\begin{center}
\begin{tabular}{c|c|c|c|c}
$\ell(\pi)$ & $J$ & $\subset \ker \rho$ & $[\frac{\ell(\pi)+3/2}2]$ & $[\ell(\pi)+3/2] - [\frac{\ell(\pi)}2] - 1$ \\
\hline
$2r + 1$ & $E^\times (1 + {\mathfrak P}^{r+1})$ & $1 + 
\begin{bmatrix} 
{\mathfrak p}^{r+1} & {\mathfrak p}^{2r+2} \\
{\mathfrak p}^{2r+2} & {\mathfrak p}^{2r+2} \end{bmatrix}$ 
& $r+1$ & $r+1$\\
$2r > 0$ & $E^\times (1 + {\mathfrak P}^{r})$ & $1 + 
\begin{bmatrix} 
{\mathfrak p}^{r+1} & {\mathfrak p}^{2r+1} \\
{\mathfrak p}^{2r+1} & {\mathfrak p}^{2r+1} \end{bmatrix}$ 
& $r$ & $r$ \\
$0$ & $Z(F) \GL_2(\mathfrak o)$ & $1 + \mathfrak P$ 
& $0$ & $0$ \\
$\frac{2r+1}2$ & $E^\times (1 + {\mathfrak P}^{r+1})$ & $1 + 
\begin{bmatrix} 
{\mathfrak p}^{[r/2]+1} & {\mathfrak p}^{r+1} \\
{\mathfrak p}^{r+2} & {\mathfrak p}^{r+1} \end{bmatrix}$ 
& $[ \frac r2] + 1$ & $r-[\frac r2] +1$\\
\end{tabular}
\end{center}
The quantities in the latter two columns will be denoted $i$ and $i'$ 
in Propositions \ref{nonzeroprop}
and \ref{twist}, and are included here for the later convenience of the reader.

\subsection{Mackey theory} \label{mackey} 

In this section we will describe the strategy to obtain the desired test vector for $\pi$. Consider a minimal supercuspidal
 representation $\pi$ of $\GL_2(F)$.  There is an open subgroup $J$ of $\GL_2(F)$ that
contains the center $Z(F)$, is compact modulo $Z(F)$, and has an irreducible
representation $\rho$ of $J$ such that $\pi \cong
c\mbox{--}\mathrm{Ind}_{J}^{\GL_2(F)}\, \rho$. 
As before, let $\Omega : T(F)
\rightarrow \C^\times$ be a character such that  $\Omega|_{Z(F)} = \omega_\pi$.

Consider the space
\begin{align}
 \mathrm{Hom}_{T(F)} ( \pi , \Omega) \cong \mathrm{Hom}_{\GL_2(F)} \left(
c\mbox{--} \mathrm{Ind}_{J}^{\GL_2(F)} \rho \, , \,
\mathcal{B}(\Omega) \right). \label{besselhom}
\end{align}
See Section \ref{wald-model} for definition of $\mathcal{B}(\Omega)$ and details of the above isomorphism. Following the proof of Proposition 1.6 of Bushnell--Henniart~\cite{BHwhittaker}
and Kutzko~\cite{Kutzko}, 
define $\mathscr{H}( \GL_2(F), \, \rho , \, \Omega)$ to
be the space of functions 
\begin{equation*}   f: \GL_2(F) \rightarrow \mathrm{Hom}_\C( \rho, \C)
\end{equation*}
satisfying
\begin{equation*}
  f(tgk)=\Omega(t) f(g) \circ \rho(k), \, t \in T(F), \, g \in \GL_2(F), \, k
\in J .
\end{equation*}
Then for $\varphi \in c \mbox{--}\mathrm{Ind}_{J}^{\GL_2(F)} \, \rho$, and $f\in
\mathscr{H}(\GL_2(F), \, \rho, \, \Omega)$ the convolution $f*\varphi$ defined
as
\begin{equation*}
 f*\varphi (g) =\int \limits_{\GL_2(F) / Z(F) } f(x) \varphi(x^{-1} g) d\bar{x},
\qquad g \in \GL_2(F)
\end{equation*}
gives a function in the space $\mathcal{B}(\Omega)$.
Furthermore, $\GL_2(F)$ acts on $\mathscr{H}( \GL_2(F), \, \rho , \, \Omega)$
through the convolution by $(g \cdot f) *\varphi=f*(g \cdot \varphi)$, and there
is a $\GL_2(F)$ homomorphism
\begin{align*}
 \mathscr{H}(\GL_2(F), \rho, \Omega) &\rightarrow \mathrm{Hom}_{\GL_2(F)} (
c\mbox{--}\mathrm{Ind}_{J}^{\GL_2(F)} \, \rho \, , \,
\mathcal{B}(\Omega))\\
f & \mapsto ( \varphi \mapsto f*\varphi).
\end{align*}
This is in fact an isomorphism. By \cite{wald},  $\mathrm{dim}_\C \mathrm{Hom}_{\GL_2(F)} (
c\mbox{--}\mathrm{Ind}_{J}^{\GL_2(F)} \, \rho \, , \,
\mathcal{B}(\Omega)) \leq 1$.  Hence, there is at most one double
coset $T(F) h_0 J$ which has non-trivial intersection with the support of any $f \in \mathscr{H}(\GL_2(F), \, \rho,
\, \Omega)$, and each $f$ is uniquely determined by its
value at $h_0$ (see 1.8 of~\cite{BHwhittaker}). Suppose $f \in
\mathscr{H}(\GL_2(F), \, \rho, \, \Omega)$ has support in a double coset
$T(F) h_0 J$, and $f(h_0)=\ell_0 \in \mathrm{Hom}(\rho, \C)$. For $k \in J \cap
h_0^{-1} T(F) h_0$, define $\Omega^{h_0}(k)=\Omega( h_0 k h_0^{-1})$.
Then $\ell_0$ has the property that for $k \in J \cap h_0^{-1} T(F) h_0$,
\begin{equation*}\ell_0(\rho(k) v) =\Omega^{h_0}(k)\, \ell_0(v).
\end{equation*} Therefore,
\begin{equation}
 \ell_0 \in \mathrm{Hom}_{J \cap h_0^{-1} T(F) h_0} ( \rho , \, \Omega^{h_0}).
\label{schom}
\end{equation}
When the Hom space in \eqref{schom} is not $0$, we say that $\pi$ and $\Omega$
\emph{intertwine} on $h_0$. If this is the case, then the double coset $T(F)h_0
J$ supports a nonzero function in $\mathscr{H}(\GL_2(F), \, \rho, \, \Omega)$,
and the $\mathrm{Hom}$ space in \eqref{besselhom} is not zero.

\subsection{Test vectors for minimal supercuspidal
representations}\label{sec-test-vec-minimal}
By \cite{Henniart}*{A.3}, if $\pi$ is a minimal representation with level $\ell(\pi)$, then $c(\pi)=2\ell(\pi)+2$. Let $\pi=c\mbox{--}\mathrm{Ind}_J^{\GL_2(F)} \rho$ as described above. Let $v_0
\in \rho$ be the vector described in Definition~\ref{v0def}.  Assume that $c(\Omega) \geq c(\pi)$. Set 
\begin{equation} 
m_0 = [\ell(\pi) + 3/2] - c(\Omega) - v(\mathbf a). \label{m0}
\end{equation}

In the next proposition, we determine a double coset representative $h_0$ of $T(F) \backslash \GL_2(F)/J$ such that $\mathrm{Hom}_{J \cap h_0^{-1} T(F) h_0} ( \rho , \, \Omega^{h_0}) \neq 0$.  We remark that this result depends on our choice of inducing subgroup $J$, and in particular the quadratic extension $E=F[\alpha]=F[\alpha_0]$ where $\alpha_0$ is always assumed to be of the form \eqref{rationalcanonical}. For $m \in \Z$ and $z \in \OF^\times$ we define $g(m,z):=
\begin{bmatrix} z \varpi^{m} & \\ & 1 \end{bmatrix}$.

\begin{lemma}\label{intersectionTJ}
 For $z \in \OF^\times$, $T(F) \cap g(m_0,z) J g(m_0,z)^{-1}=F^\times
\left(1+ \mathfrak{P}_L^{c(\Omega)-[\ell(\pi)/2]-1}\right)$.
\end{lemma}
\begin{proof}
We give the details when $\mathfrak{A}=\mathfrak{M}$.
First, suppose $\ell(\pi) = 0$. In this case $J = Z(F) \GL_2(\OF)$ and $m_0 = 1 - c(\Omega) - v(\mathbf{a})$. Furthermore, for $z \in \OF^\times$, 
$g(m_0, z ) J g(m_0, z)^{-1} = g(m_0, 1) J g(m_0, 1)^{-1}$.
Let $t^\prime \in T(F) \cap g(m_0, 1) J g(m_0, 1)^{-1}$. Since $J = Z(F) \GL_2(\OF)$, there is a unique integer $k$ so that $\varpi^k t^\prime \in T(F) \cap g(m_0, 1) \GL_2(\OF) g(m_0, 1)^{-1}$. 
Let $t = \varpi^k t^\prime = \begin{bmatrix} x+\mathbf{b}y & \mathbf{c}y \\ \mathbf{a} y & x \end{bmatrix}$.
Then 
\begin{equation*}
 g(m_0,1)^{-1} t g(m_0,1)= \begin{bmatrix} x+\mathbf{b}y & \mathbf{c}y \varpi^{-m_0} \\ \mathbf{a} y \varpi^{m_0} & x \end{bmatrix} \in \GL_2(\OF).
\end{equation*}
Therefore, $y \in \p^{-m_0 -v(\mathbf{a}) } = \p^{c(\Omega) - 1} \subset \p$. Since $g(m_0,1)^{-1} t g(m_0,1) \in \GL_2(\OF)$ , then $x \in \OF^\times$. This proves that $T(F) \cap g( m_0, z) J g(m_0, z)^{-1} \subseteq F^\times \left( 1 + \mathfrak{P}_L^{c(\Omega) - 1}\right)$. The other inclusion is straightforward. This completes the proof for $\ell(\pi) = 0$.

Now, assume $\ell(\pi) >0$. 
Note that $t \in T(F)\cap g(m_0, z) J g(m_0, z)^{-1}$ if and only if for all $w \in Z(F)$, $wt \in T(F) \cap g(m_0, z) J g(m_0, z)^{-1}$. 

Suppose $t^\prime \in T(F) \cap g(m_0, z) J g(m_0, z)^{-1}$, $t^\prime= \begin{bmatrix} x^\prime +\mathbf{b} y^\prime & \mathbf{c}y^\prime \\ - \mathbf{a}y^\prime & x^\prime \end{bmatrix}$. Let $k=\max \{ -v(x^\prime), -v(y^\prime)-m_0-v(\mathbf{a})\}$, $x=x^\prime \varpi^k$, $y=y^\prime \varpi^k$, and $t=\varpi^k t^\prime$. Then 
\begin{equation*}
 g(m_0,z)^{-1} t g(m_0,z)= \begin{bmatrix} x+\mathbf{b}y & z^{-1}\mathbf{c}y \varpi^{-m_0} \\ -z \mathbf{a} y \varpi^{m_0} & x \end{bmatrix}.
\end{equation*}
Let $i=[(\ell(\pi)+1)/2]$, so $J=E^\times U_\mathfrak{M}^i$. There is a $u \in U_\mathfrak{M}^i$ so that $g(m_0, z)^{-1} t g(m_0, z) u  \in E^\times$. Since $g(m_0, z)^{-1} t g(m_0, z) \in \Mat_2(\OF)$ and $u \in U_\mathfrak{M}^i$, this implies that 
 $a_0 z^{-1} \mathbf{c}y \varpi^{-m_0} \equiv -z \mathbf{a} y \varpi^{m_0} \mod \p^i$ (see (\ref{rationalcanonical}) for $a_0$). Since $y \varpi^{-m_0} \in \p^{-2m_0 -v(\mathbf{a})}$, then $y \in \p^{i-m_0-v(\mathbf{a})}=\p^{c(\Omega)-[\ell(\pi)/2]-1}$. But this means that $v(\mathbf{a} y \varpi^{m_0})>0$, and hence $v(x)=0$ by our choice of $k$. Therefore, $t \in \OF^\times (1+ \mathfrak{P}_L^{c(\Omega)-[\ell(\pi)/2]-1})$. The discussion above shows that $t^\prime \in F^\times (1+ \mathfrak{P}_L^{c(\Omega)-[\ell(\pi)/2]-1})$. The inclusion  $T(F) \cap g(m_0,z) J g(m_0,z)^{-1} \supseteq F^\times
(1+ \mathfrak{P}_L^{c(\Omega)-[\ell(\pi)/2]-1})$ is straightforward.

\end{proof}

\begin{proposition} \label{nonzeroprop}
Let $i=[(\ell(\pi)+3/2)/2]$. There is a unique $z_0 \in \OF^\times / (1+\p^{i})$ such that for $g(m_0,z_0):=
\begin{bmatrix} z_0 \varpi^{m_0} & \\ & 1 \end{bmatrix}$, we have 
 \begin{equation*}
  \mathrm{Hom}_{J \cap g(m_0,z_0)^{-1} T(F) g(m_0,z_0)}( \rho \, , \Omega^{g(m_0,z_0)})\neq
0.
 \end{equation*}
\end{proposition}
\begin{proof}
We give the details when $\ell(\pi)>0$ and $\mathfrak{A}=\mathfrak{M}$.
In this case $m_0={\ell(\pi) +1-c(\Omega)-v(\mathbf{a})}$ and $i=[\frac{\ell(\pi)+1}{2}]$. By Lemma~\ref{intersectionTJ}, $T(F) \cap g(m_0,z) J g(m_0,z)^{-1} = F^\times (1+ \mathfrak{P}_L^{c(\Omega)-[\ell(\pi)/2]-1})$.  
Since $\Omega|_{F^\times} = \omega_\pi$, intertwining only depends on $\Omega|_{1+
\mathfrak{P}_L^{c(\Omega)-[\ell(\pi)/2]-1}}$. Recall the definition of $\xi$
from Section~\ref{besselsubgroupsec}. The function $y
\mapsto \Omega(1+y(\xi-\mathbf{b}/2))$ is an additive character of
$\p^{c(\Omega)-[\ell(\pi)/2]-1} / \p^{c(\Omega)}$.

On the other hand, we have an isomorphism
\begin{align*}
 \p^{c(\Omega)-[\ell(\pi)/2]-1} / \p^{c(\Omega)} \longrightarrow \overline{N}^{i}U_\mathfrak{M}^{\ell(\pi)+1} / U_\mathfrak{M}^{\ell(\pi)+1}\\
 y \mapsto g(m_0, z)^{-1}(1+y(\xi-\mathbf{b}/2))g(m_0, z)
\end{align*}
Therefore, $\Omega^{g(m_0, z)}$ determines a character of $\overline{N}^{i} U_\mathfrak{M}^{\ell(\pi)+1} /
U_\mathfrak{M}^{\ell(\pi)+1} \cong \overline{N}^{i}
/\overline{N}^{\ell(\pi)+1}$. 
As $z$ runs over $\OF^\times/
(1+\p^{\ell(\pi)+1-i})$, the character determined by $\Omega^{g(m_0, z)}$ runs over all characters of
$\overline{N}^{i} /\overline{N}^{\ell(\pi)+1}$ that are non-trivial on $\overline{N}^{\ell(\pi)}$. In an abuse of notation we also refer to the character of $\overline{N}^{i} /\overline{N}^{\ell(\pi)+1}$ by $\Omega^{g(m_0, z)}$.

If $\ell(\pi)$ is odd, then
$\rho$ is a character of $J$.  Then
 $\rho|_{\overline{N}^{i} }  = {\psi^\prime}$,
i.e.\ $\rho$ restricts to a single character, and $\ell(\pi)+1-i=i$, giving the
proposition in this case. 

Otherwise $\ell(\pi)>0$ is even and, according to
Lemma~\ref{restrictionlemma}, $\rho|_{\overline{N}^{i} } $ is
the direct sum of characters all of which restrict to the same character $\psi'$
on $\overline{N}^{i+1}$. In this case there is a unique $z_0 \in
\OF^\times / (1+\p^{i})$ such that
\begin{equation*}
 \rho|_{\overline{N}^{i} } \cong \bigoplus\limits_{\substack{z \in \OF^\times / (1+\p^{i+1}) \\ z
\equiv z_0 \pmod{\p^i}}}
{\Omega^{g(m_0, z)}},
\end{equation*}
proving the proposition. The other
cases follow similarly.
\end{proof}

\begin{remark}
Let us comment on the choice of the $m_0$ in \eqref{m0}.  
Put $g_m = g(m,1)$.
One can exhibit the following double coset decomposition: 
$$\GL_2(F) = \begin{cases} \bigsqcup\limits_{m \geq v({\bf a})} T(F) g_{-m} K_{\mathfrak{M}} & \text{ if } \mathfrak{A} = \mathfrak{M}; \\
\bigsqcup\limits_{m \geq 0} T(F) g_{-m} K_{\mathfrak{J}} & \text{ if } \mathfrak{A} = \mathfrak{J}, \Big(\frac L{\p}\Big) = -1 \\
\bigsqcup \limits_{m\geq v(\mathbf{a})} T(F) g_{-m} 
K_\mathfrak{J} \sqcup T(F) \begin{bmatrix} 1 & \\ u_0 & 1 \end{bmatrix}
K_\mathfrak{J} & \text{if } \mathfrak{A} = \mathfrak{J}, \left( \frac{L}{\p} \right)=0.
\end{cases}
$$
Then, still assuming $c(\Omega) \geq c(\pi)$, one can prove that if $f \in \mathscr{H}(\GL_2(F), \rho, \Omega)$ is nonzero and is supported on the double coset $T(F) g_{-m} K_{\mathfrak{A}}$, one must have $-m = m_0$.  Thus it makes sense to look for intertwining on an element of $T(F)g_{-m} K_{\mathfrak A}$.
The decomposition above involves negative powers of the uniformizer in the double coset representatives whereas \eqref{toricdecompositioneq} uses positive exponents in the representatives. The difference in the indices occurs because for $m \geq v( {\bf a} )$, $g_{-m}$ and $g_{m-v({\bf a})}$ represent the same double coset.
\end{remark}

Next, we will define a vector in $\pi$ which will be a test vector for a $\Omega$-Waldspurger functional and will have the desired right-invariance.
Recall $v_0$ from Definition \ref{v0def}.
 Define $\varphi_0 \in \pi$ by
\begin{equation}
 \varphi_0(g)= \begin{cases} \rho(k_1)v_0 & \text{if } g=k_1 g_{m_0}^{-1} k_2, \quad k_1\in J, \quad k_2\in
K_1^{(m_0+[\ell(\pi)+1])}(\p^{2 \ell(\pi)+2}),  \\
0 & \text{ otherwise}. \end{cases}
\label{newvector}
\end{equation}
 
See (\ref{K_1-conj}) for the definition of $K^{(s)}_1(\p^n)$.  The vector $\varphi_0$ is well defined because of the inclusion $J \cap g_{m_0}^{-1}
K_1^{(m_0+[\ell(\pi)+1])}(\p^{2 \ell(\pi) +2}) g_{m_0} \subseteq Stab(v_0)$.
Since $\varphi_0$ is a translate of the newform in $\pi$, we see that $\varphi_0$ is the unique (up to scalar) $K_1^{(m_0+[\ell(\pi)+1])}(\p^{c(\pi)})$ fixed vector in $\pi$.

For $z \in \OF^\times$, define 
\begin{equation}
 \varphi_z(g)= \begin{cases} \rho(k) v_0  & \text{if } g=k g(m_0,z)^{-1}, \quad k\in J,  \\
0 & \text{otherwise.}
\end{cases}
\label{phiz} \end{equation}

\begin{proposition} Suppose $\pi$ is a minimal supercuspidal representation. Let $i$ and $z_0$
 be as in Proposition \ref{nonzeroprop}.  Then
\begin{enumerate}
\item
$$\varphi_0= \sum \limits_{z \in \OF^\times / (1+\p^{i})} \varphi_z; \qquad \text{and}$$

\item
$\ell(\varphi_0) = \ell(\varphi_{z_0})$ for $\ell \in \mathrm{Hom}_{T(F)}( \pi \, , \Omega)$.
\end{enumerate}
\end{proposition}
\begin{proof}
The space $ \left( g_{m_0} J  g_{m_0}^{-1} \cap
K_1^{(m_0+[\ell(\pi)+1])}(\p^{2 \ell(\pi)+2}) \right)  \backslash
K_1^{(m_0+[\ell(\pi)+1])}(\p^{2 \ell(\pi)+2}) $ has an irredundant set of coset
representatives given by $\left\{ g(0,z) | z
\in \OF^\times/ (1+\p^{i}) \right\}$. This shows i). It is a straightforward computation to show that the double cosets $T(F) g(m_0,z) J, z \in \OF^\times /(1+\p^i)$ are disjoint. Hence, $z_0$ is the unique element in $\OF^\times/ (1+\p^{i})$ such that the double coset  $T(F) g(m_0, z_0) J$ is in the support of a nonzero
$f \in \mathscr{H}(\GL_2(F), \, \rho,\, \Omega)$.
By the discussion in Section \ref{mackey}, this gives ii).
\end{proof}

\begin{proposition}\label{min-sup-hom-prop}
 Let $\pi$ be a minimal supercuspidal representation. There is a nonzero $\ell \in
\mathrm{Hom}_{T(F)}( \pi \, , \Omega)$ satisfying $\ell(\varphi_0)\neq 0$.
\end{proposition}
\begin{proof}
Let $z_0 \in \OF^\times /(1+\p^i)$ be as in Proposition \ref{nonzeroprop} and $\ell_0$ be a nonzero element of the space $\mathrm{Hom}_{J \cap g(m_0,z_0)^{-1} T(F) g(m_0,z_0)  }( \rho \, , \Omega^{g(m_0,z_0)})$. Define $\xi = 1_{T(F) g(m_0, z_0) J} \otimes \ell_0 \in
\mathscr{H}( \GL_2(F), \, \rho , \, \Omega)$. 
As in Section~\ref{mackey}, define
$\ell(\varphi)= \xi * \varphi(1) \in
\mathrm{Hom}_{T(F)}( \pi \, , \Omega)$. After appropriate normalization of measures, $\ell(\varphi_0) = \ell(\varphi_{z_0})=\ell_0(v_0)$.

When $\rho$ is a character, it follows immediately that $\ell(\varphi_{0})
\neq 0$. However, when $\rho$ is not a character, it must be shown that $v_0
\notin \ker \ell_0$. 

Suppose that $\ell(\pi)=2r>0$ and write $J=J_\alpha$.
Recall that under the identification $\rho|_{J_\alpha^1} \cong \eta=
\mathrm{Ind}_{A^r H_\alpha^1}^{J_\alpha^1} \tilde{\lambda}$ the vector $v_0$ is
identified with $f_0$ defined by \eqref{f0}. Consider any character $\psi'$ which
is a summand of $\rho|_{\overline{N}^r} $, and the vectors $f_a \in \eta$
defined by \eqref{fa} with respect to $\psi'$.  We may take $\ell_0$ to be given by
\begin{equation*}
 \ell_0 ( \sum \limits_{a \in \p^r / \p^{r+1} } c_a f_{a} ) := \sum
\limits_{a \in \p^r / \p^{r+1}} c_a.
\end{equation*}
Indeed, with this definition, note that for $f \in \eta$ and $x \in \p^r$,
\[ \ell_0 ( \eta (
\begin{bmatrix} 1 & \\ x & 1\end{bmatrix} )  f ) =\psi'(\bmx 1 & \\ x & 1 \emx) \ell_0 (f) 
=\Omega^{g(m_0,z)} (\bmx 1 & \\ x & 1 \emx) \ell_0 (f), \]
viewing $\Omega^{g(m_0,z)}$ as a character of $\overline{N}^{r}$ for some choice of $z \in \OF^\times /( 1 + \p^{i+1})$ corresponding to $\psi^\prime$ as in the proof of Proposition \ref{nonzeroprop}.
Hence $\ell_0(v_0)=\ell_0(f_0)=1 \neq 0$.

Finally, suppose that $\ell(\pi)=0$, so $c(\pi)=2$ and
$s=1-c(\Omega)-v(\mathbf{a})<0$. Let $h=g_s$. The linear functional $\ell_0$ is
the projection onto one of the irreducible summands of $\rho|_{\overline{N}^0}$.
Let $a \in \OF^\times$ and denote by $\psi_a$ the character of $\overline{N}^0$
given by $\psi_a ( \begin{bmatrix} 1 &  \\u & 1 \end{bmatrix} ) =
\psi_1(a u)$.
Denote by $v_a$ the vector in $\rho$ such that $\rho ( \begin{bmatrix} 1 & 
\\u & 1 \end{bmatrix} ) v_a = \psi_a (u) v_a$.

Now, write $v_0= \sum  c_a v_a$, where $c_a
\in \mathbb{C}$ and $a$ runs over $\OF^\times / (1+\p)$. For $b \in \OF^\times$,  we have $\rho (g(0,b)) v_0 = v_0$. However, $\rho ( g(0,b) ) v_{a} = v_{ba}$. Therefore, $c_a=c_{ba}$ for all
$b \in \OF^\times$. Therefore, $v_0$ has a nonzero component in each summand of
$\rho|_{\overline{N}^0}$.
\end{proof}

\noindent
{\bf Proof of Theorem~\ref{intro-tv-inert} for minimal supercuspidal representations:} 
Since $c(\Omega) \geq c(\pi)$, we can apply Lemma \ref{toric-decomp-lemma} together with Proposition \ref{min-sup-hom-prop} and the definition \eqref{newvector} to see that $\mathrm{Hom}_{T(F)}( \pi \, , \Omega) \neq 0$ and that $\varphi_0$ is a test vector with the required properties.

\subsection{Non-minimal representations}
In this section we consider a non-minimal supercuspidal representation $\tau$ and let $\Omega$ be a character of $T(F)$ such that
$\Omega|_{Z(F)} = \omega_\tau$ and $c(\Omega) \ge c(\tau)$. There exists a minimal supercuspidal
representation $\pi$ and a (ramified) character $\chi$ of $F^\times$ so that
$\tau \cong \pi \otimes \chi$. Identify $\tau=\pi \otimes \chi$. Since $\pi$ is minimal and $\tau$ is not, Proposition 3.4 of \cite{Tunnell1978} tells us 
$$c(\tau) = 2 c(\chi) > c(\pi).$$
Then $c(\Omega \otimes \chi^{-1}) \geq c(\Omega) > c(\pi)$. 

Observe that the considerations of the previous section guarantee the existence of a vector in $\pi$ that is a test vector for a $(\Omega \otimes \chi^{-1})$-Waldspurger functional and is the unique vector (up to scalars)  in $\pi$ that is right invariant under the corresponding conjugate of $K_1(\p^{c(\pi)})$. To get the desired test vector for $\tau$, we actually need a vector in $\pi$ with respect to $\Omega \otimes \chi^{-1}$, but which transforms according $\chi^{-1} \circ \det$ under right translation by a conjugate of $K_1(\p^{c(\tau)})$. In the next proposition, we obtain a vector $\varphi_\chi$ with the correct right-transformation property, and then will show that it is a test vector for the appropriate linear functional.

\begin{proposition}\label{twist}
 Suppose $\pi = c\mbox{--}\mathrm{Ind}_J^{\GL_2(F)} \rho$. Let
$s=2c(\chi)-c(\Omega)-v(\mathbf{a})$, $b=[\ell(\pi)+3/2]-c(\chi)$, $m_0 = [\ell(\pi)+3/2]-c(\Omega) - v(\mathbf a)$, $i=[(\ell(\pi)+3/2)/2]$, and $i' = [\ell(\pi)+3/2] - [\ell(\pi)/2]-1$.
\begin{enumerate}
 \item There is a unique $u \in \OF^\times/((1+\p^i) \cap \OF^\times)$,
and $v_\chi \in \rho$ depending on $u$, which is unique up to scaling,
such that for all 
$x \in \p^{i'}$,
\begin{equation}
\rho \big( \begin{bmatrix}1+x\varpi^{-b} & 0 \\xu^{-1}  & 1 \end{bmatrix} \big) v_\chi =
\chi^{-1}(1+ x \varpi^{-b})v_\chi. \label{varphi10}
\end{equation} 
 \item Suppose $u$ and $v_\chi$ satisfy \eqref{varphi10}. Let
 \begin{equation*}
 \varphi_\chi(g)=\left\{ \begin{array}{ll} (\chi^{-1} \circ \det)(k_1) \rho(k_2)
v_\chi & \text{if } g=k_2 g_\chi k_1, \, k_2 \in J, \, k_1 \in
K_1^{(s)}(\p^{2c(\chi)}) \\
0 & \text{otherwise,} \end{array} \right.
\end{equation*} where $g_\chi = \begin{bmatrix} \varpi^{-m_0} & 0\\u^{-1}
\varpi^{c(\chi)-s} & 1 \end{bmatrix}$. Then $\varphi_\chi$ is well defined, and
is the unique vector (up to scalars) in $\pi$ such that, for $k \in
K_1^{(s)}(\p^{2c(\chi)})$, $\pi(k) \varphi_\chi = (\chi^{-1}\circ \det) (k)
\varphi_\chi$.
\end{enumerate}
\end{proposition}

Note that $i=i'$ when $\ell(\pi) \in \Z$ or $\ell(\pi) = \frac{2r+1}2$ with $r$ even; otherwise they are 
off by 1 (cf.\ table in Section \ref{sec:minsc-rem}).

\begin{proof}
Observe that for $\begin{bmatrix} a_{11} & a_{12} \\ a_{21} & a_{22} \end{bmatrix} \in K_1^{(s)}(\p^{2c(\chi)})$, we have 
\begin{equation}
 g_\chi \begin{bmatrix} a_{11} & a_{12} \\ a_{21} & a_{22} \end{bmatrix} g_\chi^{-1} \in \begin{bmatrix} a_{11} & 0 \\ (a_{11}-1) u^{-1}\varpi^b & 1 \end{bmatrix} + \mathfrak{P}^{\ell(\pi)e_\mathfrak{A}+1}.
 \label{df-gchi-conj}
\end{equation}
We remark $b \le 0$.
If $g_\chi \begin{bmatrix} a_{11} & a_{12} \\ a_{21} & a_{22} \end{bmatrix} g_\chi^{-1} \in J$, then $a_{11} \equiv 1 \mod \p^{i'-b}$.
 To show that $\varphi_\chi$ is well defined, we must check that $\rho^{g_\chi}$ and $\chi$ agree on $g_\chi^{-1} J g_\chi \cap K_1^{(s)}(\p^{2c(\chi)})$. This is precisely the condition \eqref{varphi10}. Once this is established, uniqueness then follows since $\varphi_\chi \otimes \chi$ is a translate of the newform for $\tau$. Therefore, part ii) of the proposition will follow from part i).

First, suppose $\rho = \lambda$ is a character.
As in Section~\ref{construction} we have
\begin{equation*}
 \rho \big( \begin{bmatrix}1+x\varpi^{-b} & 0 \\xu^{-1}  & 1 \end{bmatrix} \big) v_\chi = \lambda \big( \begin{bmatrix}1 & 0 \\xu^{-1}  & 1 \end{bmatrix} \big).
\end{equation*}
As a function of $x$, both sides of \eqref{varphi10} are non-trivial characters of $\p^{i'} / \p^{[\ell(\pi)+3/2]}$ (cf.\ table at end of Section~\ref{sec:minsc-rem}). Therefore, there is a unique $u \in \OF^\times / (1+\p^i)$ so that \eqref{varphi10} holds. This proves part i) when $\rho$ is a character.

If $\rho$ is not a character, then $b<0$. By Section~\ref{construction},
\begin{equation*}
 \rho \big( \begin{bmatrix}1+x\varpi^{-b} & 0 \\xu^{-1}  & 1 \end{bmatrix} \big) v_\chi = \rho \big( \begin{bmatrix}1 & 0 \\xu^{-1}  & 1 \end{bmatrix} \big)v_\chi.
\end{equation*}
Suppose that $\pi$ is a depth zero supercuspidal representation. Let
$u=1$. By Lemma~\ref{restrictionlemma} there is a unique up to scalar
$v_\chi \in \rho$ so that, for $x \in \OF$,
\begin{equation*}
\rho \big( \begin{bmatrix}1 & 0 \\ x  & 1 \end{bmatrix} \big) v_\chi =
\chi^{-1}(1+x \varpi^{c(\chi)-1}) v_\chi.
\end{equation*}

Finally, suppose that $\ell(\pi)=2r>0$, and $\pi=c
\mbox{--}\mathrm{Ind}_{J}^{\GL_2(F)} \rho$ where $\rho$ is not a
character of $J$.  By Lemma~\ref{restrictionlemma}, $\rho|_{\overline{N}
\cap H_\alpha}$ is a multiple of $\psi_\alpha|_{\overline{N} \cap H_\alpha}$.
There exists a unique $u \in \OF^\times / (1+\p^{i})$
so that
\eqref{varphi10} holds for $x \in \p^{i'}$.
 By Lemma~\ref{restrictionlemma}, with this choice of $u$ there is a
unique up to scalar multiple $v_\chi \in \rho$ so that \eqref{varphi10} holds. This completes the proof of i) of the proposition.
\end{proof}

The next lemma gives a double coset decomposition of the support of $\varphi_\chi$.
\begin{lemma}\label{support-lemma}
Let $g_\chi = \begin{bmatrix} \varpi^{-m_0} & 0\\u^{-1}
\varpi^{c(\chi)-s} & 1 \end{bmatrix}$ as in Proposition \ref{twist}. Then
$$Jg_\chi K_1^{(s)}(\p^{2c(\chi)}) = \bigsqcup_{z \in \OF^\times / (1+\p^{c(\chi)-[\ell(\pi)/2]-1})} J g_\chi \bmx z & \\ & 1 \emx.$$
\end{lemma}
\begin{proof} Recall
$$K_1^{(s)}(\p^{2c(\chi)}) = \mat{1}{\p^s}{\p^{c(\Omega)+v({\bf a})}}{1+\p^{2c(\chi)}} \mat{\OF^\times}{}{}{1}.$$
We have the following, for $z \in \OF^\times$,
\begin{align}
g_\chi \mat{1}{\p^s}{\p^{c(\Omega)+v({\bf a})}}{1+\p^{2c(\chi)}} g_\chi^{-1} &= \mat{1+\p^{c(\chi)}}{\p^{2c(\chi)-[\ell(\pi)+3/2]}}{\p^{[\ell(\pi)+3/2]}}{1+\p^{c(\chi)}} \label{eq11}\\
g_\chi \mat{z}{}{}{1} g_\chi^{-1} &= \mat{z}{}{u^{-1}\varpi^{c(\chi)-s+m_0}(z-1)}{1} \label{eq22}
\end{align}
Using the description of $J$ in the table at the end of Section \ref{sec:minsc-rem}, we see that the right hand side of \eqref{eq11} lies in $J$. Also, the right hand side of \eqref{eq22} lies in $J$ if and only if $z \in 1+\p^{c(\chi)-[\ell(\pi)/2]-1}$. This completes the proof of the lemma.
\end{proof}

The next two lemmas give a useful decomposition of $g_\chi$. Fix $g_\chi$ and $u$ to be as in Proposition~\ref{twist}, and set $y_0=-u^{-1}\mathbf{a}^{-1} \varpi^{c(\Omega)+v(\mathbf{a})-c(\chi)}$.  For $y \in F$, define
$t_y = t(1+\mathbf b y/2, y) = \begin{bmatrix}1 + \mathbf b y & \mathbf c y \\ - \mathbf a y  & 1 \end{bmatrix}$.
 
\begin{lemma} \label{gchi_decomp}
We have $g_\chi=k_0
g_{m_0}^{-1} t_{y_0}$, where $k_0 \in U_\mathfrak{A}^{\ell(\pi)e_\mathfrak{A}+1}$.
 \end{lemma}
\begin{proof} 
Write $g_\chi= g_{m_0}^{-1} g$ where $g = \begin{bmatrix} 1 & \\ -\mathbf{a}y_0 & 1 \end{bmatrix}$. Let
$k_0^{-1}= g_{m_0}^{-1} t_{y_0} g^{-1} g_{m_0}.$ 
We see that
 \begin{equation*}
  k_0^{-1}= \begin{bmatrix}1+\mathbf{b}y_0 + \mathbf{a}\mathbf{c} y_0^2
& \mathbf{c}y_0 \varpi^{-m_0}  \\
0 & 1 \end{bmatrix}.
 \end{equation*}
So $k_0 \in U_\mathfrak{A}^{\ell(\pi) e_\mathfrak{A}+1}$  and $g_\chi=k_0 g_{m_0}^{-1} t_{y_0}$.
\end{proof}

\begin{lemma}\label{gchi_decomp2}
For each $z\in \OF^\times$, there exists $k_z \in U_\mathfrak{A}^{\ell(\pi)e_\mathfrak{A}+1}$ such that
 \begin{equation}
   g_{m_0}^{-1} t_{y_0}
\begin{bmatrix} z & \\ & 1 \end{bmatrix} = k_z g_{m_0}^{-1} \begin{bmatrix}
z & \\ & 1\end{bmatrix} t_{zy_0}.
 \end{equation}
\end{lemma}

\begin{proof}
Write $g_\chi= k_0 g_{m_0}^{-1} t_{y_0}$ as in Lemma~\ref{gchi_decomp}. Then 
\begin{align*}
 g_{m_0}^{-1}  t_{y_0} \begin{bmatrix} z & \\ & 1 \end{bmatrix}
 &= k_0^{-1} g_\chi \begin{bmatrix}z & \\ & 1 \end{bmatrix}
 = k_0^{-1} g_{m_0}^{-1} \begin{bmatrix} 1 & \\ -\mathbf{a} y_0 & 1 \end{bmatrix} \begin{bmatrix} z & \\ & 1 \end{bmatrix}\\
 &= k_0^{-1}\begin{bmatrix} z & \\ & 1 \end{bmatrix} g_{m_0}^{-1} \begin{bmatrix} 1 & \\ -\mathbf{a} z y_0 & 1 \end{bmatrix} 
 = k_0^{-1}\begin{bmatrix} z & \\ & 1 \end{bmatrix} k_0' g_{m_0}^{-1} t_{zy_0} 
 = k_z \begin{bmatrix} z & \\ & 1\end{bmatrix} g_{m_0}^{-1} t_{zy_0}.
\end{align*}
For the second to last equality we have used a decomposition similar to Lemma~\ref{gchi_decomp}, and $k_0'$ is the corresponding element of $U_\mathfrak{A}^{\ell(\pi) e_\mathfrak{A}+1}$. For the last equality we use the fact that the subgroup $U_\mathfrak{A}^{\ell(\pi) e_\mathfrak{A}+1}$ is normalized by $A^0$.
\end{proof}

Let us remark here that $U_\mathfrak{A}^{\ell(\pi)e_\mathfrak{A}+1}$ lies in the kernel of $\rho$ (see the table at the end of Section \ref{sec:minsc-rem} for details). For any $g \in \GL_2(F)$ and $v \in \rho$, define 
$$\varphi_{g, v}(h)= \left\{ \begin{array}{ll} \rho(k) v & \text{if } h=k
g, \, k \in J, \\
0 & \text{otherwise.} \end{array} \right.$$
Note that, for any $z \in \OF^\times$, the support of $\pi(\mat{z^{-1}}{}{}{1})\varphi_{g_{\chi}, v_\chi}$ is exactly $Jg_\chi \mat{z}{}{}{1}$. 
Then
\begin{align}
\varphi_\chi &= \sum \limits_{z \in \OF^\times / (1+\p^{c(\chi)-[\ell(\pi)/2]-1})} \chi^{-1}(z)
\, \pi (  \begin{bmatrix} z^{-1} & \\ & 1 \end{bmatrix} ) \varphi_{g_{\chi},
v_\chi} \nonumber \\
 &= q^{-[\ell(\pi)/2]-1} \sum \limits_{z \in \OF^\times / (1+\p^{c(\chi)})} \chi(z)
\, \pi (  \begin{bmatrix} z & \\ & 1 \end{bmatrix} ) \varphi_{g_{\chi},
v_\chi} \nonumber \\
&= q^{-[\ell(\pi)/2]-1} \sum \limits_{z \in \OF^\times / (1+\p^{c(\chi)})} \chi(z) \, \pi( t_{z^{-1} y_0}^{-1} )
\varphi_{g(m_0,z)^{-1}, v_\chi}. \label{ellvarphichi}
\end{align}
The first equality follows from the double coset decomposition in Lemma \ref{support-lemma}. To get the second equality, note that for $z \in 
1+\p^{c(\chi)-[\ell(\pi)/2]-1}$, we can apply (\ref{varphi10}) to the right hand side of (\ref{eq22}). Finally, the third equality follows from Lemmas \ref{gchi_decomp} and \ref{gchi_decomp2}.

Write $c_0=[ \frac{c(\chi)+1}{2}]$.  For $x \in \p^{c_0}$, $x \mapsto \chi(1+x)$
is an additive character of $\p^{c_0}/\p^{c(\chi)}$. 

\begin{proposition} \label{non-minimal-prop-2}
 Suppose $\ell \in \mathrm{Hom}_{T(F)}( \pi , \Omega \otimes \chi^{-1})$.  
 There is a unique $w_0 \in \OF^\times / (1+\p^{c(\chi)-c_0})$ satisfying the following conditions:
 \begin{enumerate}
 \item If \begin{equation*}\sum \limits_{z \in (1+\p^{c_0}) / (1+\p^{c(\chi)})}
\chi(z w) \, \ell\left( \pi\left( t_{(zw)^{-1}y_0}^{-1} \right)
\varphi_{g(m_0, zw)^{-1}, v_\chi} \right) \neq 0,\end{equation*} then $w \equiv
w_0 \mod \p^{c(\chi)-c_0}$.
 \item $\ell(\varphi_{g(m_0, zw_0)^{-1}, v_\chi}) \neq 0$.
 \end{enumerate}
\end{proposition}
\begin{proof} Recall $y_0=-u^{-1}\mathbf{a}^{-1} \varpi^{c(\Omega)+v(\mathbf{a})-c(\chi)}$. First, we note that $\chi \circ \det$ is trivial on elements $t_{zy_0}$ for $z \in
\OF^\times$. We can define an additive character of $\OF$ by $\psi_{\Omega}(x):=\Omega
\left( t_{xy_0}
\right)$. 
By \eqref{ellvarphichi}, we have 
\begin{equation}
 \ell(\varphi_\chi)=q^{-[\ell(\pi)/2]-1} \sum \limits_{w \in \OF^\times / (1+\p^{c_0})} \sum
\limits_{z \in (1+\p^{c_0})/(1+\p^{c(\chi)})} \chi^{-1}(wz) \, \psi_{\Omega}^{-1} (wz) \,
\ell(\varphi_{g(-m_0, zw), v_\chi}) \label{ellvarphichi2}
\end{equation} 
If $z \in 1+\p^{c_0}$, then $\rho(g(0,z^{-1})) v_\chi = v_\chi$, and
$\varphi_{g(-m_0, zw), v_\chi}=\varphi_{g(-m_0 , w), v_\chi}$.  Consider the inner sum of \eqref{ellvarphichi2}:
 \begin{align}
  \sum \limits_{z \in (1+\p^{c_0})/(1+\p^{c(\chi)})} \chi^{-1}(wz) \, \psi_\Omega^{-1} (wz) 
  =\chi^{-1}(w) \, \psi_\Omega^{-1} ( w) \sum \limits_{x \in \p^{c_0}/ \p^{c(\chi)} }
\chi^{-1}(1+x) \, \psi_\Omega^{-1} (wx). \label{innersum}
 \end{align}
This sum does not equal zero if and only if $\chi^{-1}(1+x)= \psi_\Omega (xw)$ for
all $x \in \p^{c_0}$, and this occurs for exactly one element $w=w_0 \in \OF^\times / (1+\p^{c(\chi)-c_0})$.
This proves the first part. 

Consider an element $t \in T(F)  \cap g_{m_0}J g_{m_0}^{-1}= F^\times(1+\mathfrak{P}_L^{c(\Omega)-[\ell(\pi)/2]-1})$ (cf.\ Lemma \ref{intersectionTJ}). 
Since $2 c(\chi) > c(\pi) = 2\ell(\pi) + 2$, we see 
\begin{align*}
( c( \Omega ) -[ \ell(\pi)/2 ]-1 ) - (c(\Omega) - c(\chi) ) = c(\chi) - [\ell(\pi)/2 ] - 1 > c_0.
\end{align*}
Therefore, there exists $x \in \p^{c_0}$ and $z \in F^\times$ such that $t = z t_{xy_0}$, and, by the remarks after \eqref{innersum}, 
\begin{equation*}
\chi^{-1}(1+w_0^{-1}x)= \psi_\Omega (x) = \Omega( t_{xy_0}) = \Omega(z)^{-1} \Omega(t).
\end{equation*}
By \eqref{varphi10} we have
\begin{equation*}
 \rho ( \begin{bmatrix} 1 &  \\ u^{-1} x & 1 \end{bmatrix} ) v_\chi =
\chi^{-1}( 1 + x \varpi^{-b} ) v_\chi.
\end{equation*}
Therefore, for all $t \in T(F)  \cap g_{m_0}J g_{m_0}^{-1} = T(F)  \cap g_{m_0, w_0}J g_{m_0, w_0}^{-1}$,
\begin{equation*}
\rho( g_{m_0, w_0}^{-1}\, t \, g_{m_0, w_0} ) v_\chi = \Omega( t ) v_\chi =\Omega( t ) (\chi^{-1}\circ \det)(t) v_\chi.
\end{equation*}
This implies that there is $\ell_0 \in \mathrm{Hom}_{g(m_0, w_0)^{-1} T(F)g(m_0, w_0) \cap
J}(\rho, (\Omega\otimes\chi^{-1})^{g(m_0, w_0)}) $ such that $\ell_0 ( v_\chi) \neq
0$, and, from the discussion in Section~\ref{mackey}, after a normalization 
$\ell(\varphi_{g(m_0, w_0)^{-1}, v_\chi})=\ell_0(v_\chi)$.
\end{proof}

\noindent
{\bf Proof of Theorem~\ref{intro-tv-inert} for non-minimal supercuspidal representations:}
By Proposition~\ref{non-minimal-prop-2} we compute 
\begin{align*}
 \ell(\varphi_\chi) &=q^{-[\ell(\pi)/2]-1} \ell(\varphi_{g(m_0 , w_0)^{-1}, v_\chi}) \sum \chi^{-1}(w
w_0^{-1}) \psi_\Omega^{-1}( w w_0^{-1})\\
 &= q^{-[\ell(\pi)/2]-1-\left[ \frac{c(\chi)}{2}\right]} \psi_\Omega^{-1}(w_0^{-1}) G(
\chi, \psi_\Omega^{-1}) \ell(\varphi_{g(m_0 , w_0)^{-1}, v_\chi})\neq 0,
\end{align*}
where the sum is over $w \in (1+\p^{c(\chi)-c_0})/(1+\p^{c_0})$, $w_0$ is the
unique element of $\OF^\times / (1+\p^{c(\chi)-c_0})$ such that $\chi^{-1}(1+z)=
\psi_\Omega(z w_0^{-1})$ for all $z \in \p^{c_0}$, and $G(\chi,
\psi_\Omega^{-1})$ is Gauss sum for the pair $\chi$, $\psi_\Omega^{-1}$.
For the last equality see \cite{BH}*{23.6 Proposition}. This shows that $\mathrm{Hom}_{T(F)}( \tau \, , \Omega) \neq 0$. The one-dimensionality follows from \cite{wald}. Since $c(\Omega) \geq c(\tau)$, we can apply Lemma \ref{toric-decomp-lemma} to obtain the test vector with required properties. The uniqueness of the test vector follows from the uniqueness of the newform in $\pi$.

\section{Local spectral distributions} \label{klm-sd-sec}

Now we return to the setting where $F$ is a $p$-adic field and $L$ is a quadratic separable extension as in Section \ref{L-def-section}.  Let
$\pi$ be an infinite-dimensional, irreducible, admissible
representation of $\GL_2(F)$, and $\Omega$ a character of $L^\times$ such that
$\Omega|_{F^\times} =
\omega_\pi$.    In this section, we will
calculate certain local
spectral distributions $\~J_\pi(f)$ defined by Jacquet--Chen \cite{JC} for
certain test functions 
$f \in C_c^\infty(\GL_2(F))$.  These are used in Section
\ref{klm-lval-sec} to generalize the global $L$-value formula previously
obtained by the second author
and Whitehouse in \cite{MW}.  For simplicity, we will prove this global
$L$-value formula when the
central character of our automorphic representation is trivial, so we may as
well assume $\omega_\pi=1$ in
this section also. We will also assume that $\pi$ and $\Omega$ are unitary, since the global objects in the following sections are unitary as well.

The calculation of $\~J_\pi(f)$
is contained in \cite{MW} in the cases when $F$ is archimedean, $L/F$ is split,
or $\pi$ and
 $\Omega$ have
disjoint ramification.  Hence, we will assume
 $L/F$ is a quadratic extension of nonarchimedean fields and $c(\pi), c(\Omega)>0$.  In particular, either $L(s,\pi)=1$ or $\pi = \chi \St_{\GL_2}$, where, for the
rest of this section, $\chi$ denotes
an unramified quadratic character.  Further, we assume $c(\Omega) \ge c(\pi)$ to use our
determination of test vectors.

Write $L^\times = F(\xi)^\times$ where $\xi = \frac{\sqrt{\mathbf d}}2$.  Let
$T=T(F)$ be the torus in $\GL_2(F)$ isomorphic to $L^\times$ defined in \eqref{TFdefeq}.  Here
it is
convenient to take a slightly different parameterization for $T$ than the one given by $t(x,y)$ 
in \eqref{TFgeq}.  Namely, we map
\begin{equation} \label{klm-T-param}
 x+y\xi_0\mapsto \bmx x & \mathbf cy \\ - \mathbf ay & x-\mathbf b y\emx,
\end{equation}
where $\xi_0 = \xi - \mathbf b/2 = (\sqrt {\mathbf d} - \mathbf b)/2$.
By \cite{tunnell} and \cite{saito} or Theorem~\ref{intro-tv-inert}, 
the assumption $c(\Omega) \ge c(\pi)$ implies $\dim_{\C} 
\mathrm{Hom}_T(\pi,\Omega) = 1$. 
Fix a nonzero linear functional 
$\ell \in \mathrm{Hom}_T(\pi,\Omega)$.

Consider the Kirillov model for $\pi$ and the inner product on $\pi$ given by
\[ (\phi_1, \phi_2) = \int_{F^\times} \phi_1(a) \overline{\phi_2(a)} \, d^\times
a, \]
where $d^\times a$ is the Haar measure giving $\vol(\mathfrak o^\times) = 1$. 
This inner product is $\GL_2(F)$-invariant.
Let $e$ be the unique (up to scalars) vector in $\pi$ such that $\pi(t)e =
\Omega(t)e$ for $t \in T$, which
we normalize so that $(e,e)=1$.  
Let $dg$ denote the local Tamagawa measure on $\GL_2(F)$.
Then the local distribution we are interested in is defined in \cite{JC} by
\begin{equation} \label{klm-jc-def}
 \~ J_\pi(f) =(\pi(f)e,e) =  \int_{\GL_2(F)} f(g) ( \pi(g)e, e ) \, dg, \qquad f
\in C_c^\infty(\GL_2(F)).
\end{equation}
Put $s= c(\Omega)-c(\pi)$, $h =  \bmx \varpi^s & \\ & 1 \emx w$, and
\begin{equation} \label{klm-subgp-def}
 K' = h K_1(\mathfrak p^{c(\pi)}) h^{-1} = \bmx 1+\mathfrak
p^{c(\pi)} & \mathfrak p^{c(\Omega)} \\
\mathfrak p^{c(\pi)-c(\Omega)} & \mathfrak o^\times \emx.
\end{equation}
Then Theorem~\ref{intro-tv-inert} says
there is a unique-up-to-scalars test vector $\phi \in \pi$ which is
right invariant by $K'$
 such that $\ell(\phi) \ne 0$. 
Let $\phi_0$ be the newvector in $\pi$ normalized so that $\phi_0(1)=1$.  Then
we may take $\phi = \pi(h) \phi_0$.

Observe that $\Omega$ is trivial on $T \cap Z K'$, where $Z=Z(T)$, since $\phi$ is fixed by
$ZK'$.
Consider the vector $e' \in \pi$ given by
\begin{equation}
 e' = \sum_{t \in T/(T \cap Z K')} \Omega^{-1}(t) \pi(t) \phi.
\end{equation}
Note the index set for the sum is finite so $e'$ is well defined.
Then for any $t \in T$, we have
$ \pi(t)e' = \Omega(t)e'$, and $\ell(e') \neq 0$. In other words, we may assume
\[ e=\frac{e'}{(e',e')^{1/2}}. \]
We will take for our test function $f = 1_{K'}/\vol(K')$, so our calculations do not in fact depend on the normalization of $dg$ in \eqref{klm-jc-def}.
Then
\[ \~ J_\pi(f) =  \vol(K')^{-1} \int_{K'} ( \pi(k) e, e ) \, dk
= \vol(K')^{-1}  \int_{K'} \frac { ( \pi(k) e', e' )} { ( e', e' )} \, dk.  \]
Note, using the $\GL_2(F)$-invariance of the inner product, we get 
\[  ( e', e' ) = \sum_{t \in T/(T \cap ZK')} \Omega^{-1}(t) (  \phi, \pi(t^{-1})
e' )
= |T/(T \cap ZK')| (\phi, e' ). \]
Since $\pi(f)$ is simply orthogonal projection onto $\left< \phi \right> =
\pi^{K'}$,
\[ \vol(K')^{-1} \int_{K'} ( \pi(k) e', e' ) \, dk = ( \pi(f) e', e' ) =  \frac{(
e', \phi )(\phi, e' )}{( \phi, \phi )}. \]
Hence
\begin{equation} \label{klm-jpif2}
 \~J_\pi(f) = \frac 1{|T/(T \cap ZK')|} \frac{(e',\phi)}{(\phi,\phi)}.
\end{equation}
Note that
\begin{equation} \label{klm-ipcalc}
(\phi, \phi) = (\phi_0, \phi_0) = \begin{cases}
L(2,1_F) & \text{if } \pi = \chi \St_{\GL_2}, \\
1 & \text{if } L(s,\pi)=1,
\end{cases}
\end{equation}
so it remains to compute $|T/(T \cap ZK')|$ and $(e',\phi)$.  
(Recall $\chi$ denotes an unramified character.)
 Only the latter
computation is involved.  This requires
knowing some facts about values of the Whittaker newform and determining a set
of representatives for $T/(T \cap ZK')$.
We first tackle these two tasks, and then compute $(e',\phi)$, and hence
$\~J_\pi(f)$, under our above assumptions.

\subsection*{Whittaker values}

Assume $\pi$ has trivial central character and 
let $\psi$ be a nontrivial additive character of $F$ of conductor $\OF$. 
 Let $\mathcal W(\pi,\psi)$ be the Whittaker
model for $\pi$ with respect to $\psi$. Let $W_0$ be the newform normalized so that $W_0(1) = 1$, and therefore $\phi_0(a) =
W_0 ( \bmx a & \\ & 1 \emx )$.  

We will be interested in certain values of the Whittaker newform when the local
$L$-factor of $\pi$ has
degree 1 or 0.  For this, we recall (cf.\ Table \ref{whit-table} in Section \ref{split-tv-sec})
that 
$\phi_0(a) = \chi(a) |a| 1_{\mathfrak o}(a)$ when $\pi = \chi \St_{\GL_2}$ and
$\phi_0(a) = 1_{\mathfrak o^\times}(a)$
when $L(s,\pi)=1$.
From this, one obtains the following result on Whittaker newform values.

\begin{lemma} \label{klm-whit-val}
\begin{enumerate}
\item \label{klm-whit-val1}
If $u, v \in \mathfrak o^\times$, then
\[ W_0 ( g \bmx u & \\ & v \emx w ) = W_0 ( gw ). \]

\item \label{klm-whit-val2}
If $\pi = \chi \St_{\GL_2}$ with $\chi$ unramified, then for $j \in \Z$,
\[ W_0 ( \bmx \varpi^j & \\ & 1 \emx w ) = 
\begin{cases}
- \chi(\varpi)^{j} q^{-j-1} & j \ge-1, \\
0 & \text{else}. 
\end{cases}  \]

If $L(s,\pi)=1$, then for any $j \in \Z$,
\[ W_0 ( \bmx \varpi^j & \\ & 1 \emx w ) = 
\begin{cases}
\epsilon(1/2,\pi) & j=-c(\pi), \\
0 & \text{else}. 
\end{cases}  \]

\item \label{klm-whit-val4}
If $\pi=\chi\St_{\GL_2}$ with $\chi$ unramified, for $j \ge 0 \ge k$, we have
\[ \int_{\mathfrak o^\times}  W_0 ( \bmx \varpi^j u & \\ & 1 \emx \bmx 1 &
\\ \varpi^k & 1 \emx ) \, d^\times u= 
-q^{-1} (\chi(\varpi)q^{-1})^{j-2k}.  \]

If $L(s,\pi)=1$ then for all $j,k \in \Z$, 
\[ \int_{\mathfrak o^\times}  W_0 ( \bmx \varpi^j u & \\ & 1 \emx \bmx 1 &
\\ \varpi^k & 1 \emx ) \, d^\times u 
= \begin{cases}
1 & j=0 \text{ and } k \ge c(\pi), \\
(1-q)^{-1} & j=0 \text{ and } k=c(\pi)-1, \\
0 & \text{else}.
\end{cases} \]
\end{enumerate}
\end{lemma}

\begin{proof} \rref{klm-whit-val1} follows simply from the facts that $W_0$ is
right invariant by $K_1(\mathfrak p^{c(\pi)})$ and $\omega_\pi=1$. The proof of \rref{klm-whit-val2} and \rref{klm-whit-val4} follows from the functional equation (\ref{local-fnal-eqn}) with $\mu = 1$ by comparing coefficients of $q^s$.
\end{proof}

\subsection*{The toric quotient}

We identify $t=x+y\xi_0 \in L^\times$ with its image in $T$ via \eqref{klm-T-param}.
Since we have assumed that $c(\Omega) \geq c(\pi)$, we have $t=x+y \xi_0 \in K'$ if and only if $x \in 1 + \mathfrak p^{c(\pi)}$ and
$y \in \mathfrak p^{c(\Omega)}$. 

\begin{lemma} \label{klm-rep-lem}
We have
\begin{equation} \label{klm-repcount}
|T/(T \cap ZK') | = \begin{cases} q^{c(\Omega)}(1+q^{-1}) & L/F \text{ unramified,} \\
2q^{c(\Omega)} & L/F \text{ ramified}.
\end{cases}
\end{equation}
Furthermore, if $L/F$ is unramified or $v(\mathbf a)=1$, then
a complete set of representatives of $T/(T \cap Z K')$ is given by 
\begin{equation} \label{klm-repset-unr+va1}
 \{ 1+y\xi_0 : y \in \mathfrak o/\mathfrak p^{c(\Omega)} \} \cup 
\{ x+\xi_0 : x \in \mathfrak p/\mathfrak p^{c(\Omega)+v(\mathbf a)} \}.
\end{equation}
If $L/F$ is ramified and $v(\mathbf a)=0$, then 
a complete set of
representatives of $T/(T \cap Z K')$ is given by 
\begin{equation} \label{klm-repset-ram-va0}
 \{ 1+y\xi_0 : y \in \mathfrak o/\mathfrak p^{c(\Omega)}, \, y \not \equiv u_0' \mod
\mathfrak p \} \cup
 \{ 1+(u_0'+y)\xi_0 : y \in \mathfrak p/\mathfrak p^{c(\Omega)+1}\} \cup 
\{ x+\xi_0 : x \in \mathfrak p/\mathfrak p^{c(\Omega)} \} 
\end{equation}
where $u_0' = -u_0/{\bf a} \in \mathfrak o^\times$ with $u_0$ as in
\eqref{u0defeq}.
\end{lemma}

\begin{proof}
We will obtain the set of representatives of $T/(T \cap Z K')$, from which \eqref{klm-repcount} follows.
Given an arbitrary $t=x+y\xi_0
\in T$, we may multiply $t$ by an element of $Z$ to assume that
both $x, y \in \mathfrak o$ and either $x=1$ or $y=1$.  Further, if $x$ and $y$
are both units, we may
assume $x=1$.  So we may consider a set of representatives of the form $x +
\xi_0$ and $1+y \xi_0$
where $x \in \mathfrak p$ and $y \in \mathfrak o$.
Observe $\xi_0^2 = - \mathbf{ac} -\mathbf b \xi_0$.  For $t, t' \in T$, write $t \sim t'$ if $t = t_0 t'$ for some $t_0 \in T \cap
ZK'$.  

First we observe that $x+\xi_0 \sim 1+y \xi_0$, where $x \in \mathfrak p$
and $y \in \mathfrak o$, is not possible.  If it were, there would exist $u \in
1+\mathfrak p^{c(\pi)}$, 
$r \in \mathfrak \p^{c(\Omega)}$ and $z \in F^\times$ such that
\[ zx+z\xi_0 = (u+r\xi_0)(1+y\xi_0) = u-\mathbf{ac}ry + (uy + r -\mathbf b ry)
\xi_0. \]
Since $u-\mathbf{ac}ry \in \mathfrak o^\times$, we see $v(z) < 0$, but $z=uy + r
-\mathbf b ry \in \mathfrak o$,
a contradiction.  

Now consider $x_1 + \xi_0 \sim x_2 + \xi_0$ for $x_1, x_2 \in \mathfrak p$. 
Then, for some 
$u \in 1+\mathfrak p^{c(\pi)}$,  $r \in \p^{c(\Omega)}$ and $z \in F^\times$, we have
\[ zx_1+z\xi_0 = (u+r\xi_0)(x_2+\xi_0) = ux_2-\mathbf{a c} r + (u+rx_2-\mathbf b
r)\xi_0. \]
Hence $z=u+rx_2-\mathbf b r \in 1 + \mathfrak p^{c(\pi)}$ and
\[ zx_1 =ux_1+rx_1x_2-\mathbf b rx_1 = ux_2-\mathbf{a c} r, \]
which implies
\[ u(x_2-x_1) = r(\mathbf{ac} - \mathbf b x_1 + x_1 x_2). \]
In particular, we must have $x_1 \equiv x_2 \mod \mathfrak p^{c(\Omega)+v(\mathbf a)}$.
 
In fact, if $v(\mathbf a)=0$, we have $x_1 + \xi_0 \sim x_2 + \xi_0$ iff $x_1
\equiv x_2 \mod \mathfrak p^{c(\Omega)}$.
Similarly, if $v(\mathbf a)=1$, then $v(\mathbf b) > 0$ and
$x_1 + \xi_0 \sim x_2 + \xi_0$ iff $x_1 \equiv x_2 \mod \mathfrak p^{c(\Omega)+1}$. The rest of the cases are computed similarly.
\end{proof}

Let us remark that the coset representatives in the previous lemma depend only on $c(\Omega)$ since we are in the case $c(\Omega) \geq c(\pi)$. 

\subsection*{Projection onto the test vector}

Put $e(L/F)=1$ if $L/F$ is unramified and $e(L/F)=2$ if $L/F$ is ramified.
Denote by $\eta$ the quadratic character of $F^\times$ associated to $L/F$. 
\begin{proposition} \label{klm-Jpif-prop}
If $c(\pi) \ge 2$, then
\begin{equation}
 \~J_\pi(f)  = q^{-c(\Omega)} \frac{L(1,1_F)L(1,\eta)}{e(L/F)}.
\end{equation}
If $c(\pi)=1$, then
\begin{equation}
 \~J_\pi(f)  = q^{-c(\Omega)} \frac{L(1,1_F)L(1,\eta)}{e(L/F)L(2,1_F)}.
\end{equation}
\end{proposition}
\begin{proof}
By \eqref{klm-jpif2}, \eqref{klm-ipcalc}, and \eqref{klm-repcount}, this
proposition is equivalent to the
statement that
$$
(e',\phi) = L(1,1_F).
$$
To show this, first observe
$$(e',\phi) = \sum_{t \in T/(T \cap ZK')} \Omega^{-1}(t)
(\pi(t)\phi,\phi)
= \sum_{t \in T/(T \cap ZK')}  \Omega^{-1}(t) (\pi(h^{-1}th)\phi_0,\phi_0).
$$
Recall that $h = \mat{\varpi^s}{}{}{1} w$ with $s = c(\Omega)-c(\pi)$.  
We will give the details of the case $c(\pi) \geq 2$ here. The other case is computed similarly. Hence, assume that $c(\pi) \geq 2$ so $L(s,\pi)=1$.  Then, for $g \in \GL(2)$,
$$  (\pi(g)\phi_0, \phi_0) =  \int_{\mathfrak o^\times} W_0 ( \bmx  u & \\ & 1
\emx g ) \, d^\times u.
$$
First  suppose $t=x+\xi_0$ where $x \in \mathfrak p$ and $v(x) \leq c(\Omega)+v(\mathbf
a)$.  
Note
$$ h^{-1} t h = \bmx x-\mathbf b & \varpi^s \mathbf a \\ -\varpi^{-s} \mathbf c &
x \emx 
= \varpi^{-s}\mathbf c \bmx 1 & (\mathbf b-x)\varpi^s/\mathbf c \\ & 1 \emx
\bmx \det(t)\varpi^{2s}/\mathbf c^2 & \\ & 1 \emx w
\bmx 1 & -\varpi^{s} x/\mathbf c \\ & 1 \emx.
$$
Since the rightmost matrix lies in $K_1(\mathfrak p^{c(\pi)})$, we have
\[ (\pi(h^{-1}th)\phi_0,\phi_0) = W_0 ( \bmx \det(t)\varpi^{2s}/\mathbf c^2 & \\
& 1 \emx w ) = 0, \]
where the last equality follows from Lemma~\ref{klm-whit-val} \ref{klm-whit-val2}).
Now suppose $t=1+y\xi_0$ where $y \in \mathfrak o$.  If $y=0$, then 
$$  (\pi(h^{-1}th)\phi_0,\phi_0) = (\phi_0,\phi_0) = 1.
$$
Otherwise, assume $v(y) < c(\Omega)$ and write
$$ h^{-1} t h = \bmx 1 & \varpi^{s} \mathbf a y \\ & 1 \emx \bmx \det(t) & \\ & 1
\emx
\bmx 1 & \\ -\varpi^{-s} \mathbf c y & 1 \emx.
$$
Then, by Lemma~\ref{klm-whit-val} \ref{klm-whit-val4}),
\[ 
  (\pi(h^{-1}th)\phi_0,\phi_0) = \nonumber
   \int_{\mathfrak o^\times} W_0 ( \bmx \det(t) u & \\ & 1 \emx
\bmx 1 & \\ -\varpi^{-s} \mathbf c y & 1 \emx ) \, d^\times u \\
= \begin{cases}
(1-q)^{-1} & \text{if } v(y)=c(\Omega)-1, \\
0 & \text{else}.
\end{cases} \]
Observe that $\int_{1+\varpi^k \mathfrak o_L}
\Omega^{-1}(u) \, d^\times u = 0$ for
$0 < k < c(\Omega)$, together with $\Omega^{-1}|_{\mathfrak o^\times} = 1$, implies 
\[ \sum_{y \in \mathfrak o/\mathfrak p^{c(\Omega)} : v(y) \ge k} \Omega^{-1}(1+y\xi_0) = 0.
\]
Hence, for $0 < k \le c(\Omega)$, we have
\[ \sum_{y \in \mathfrak o/\mathfrak p^{c(\Omega)} : v(y)=k} \Omega^{-1}(1+y\xi_0) = 
\begin{cases}
0 & \text{if }  0 < k < c(\Omega)-1,  \\
-1 &  \text{if }k=c(\Omega)-1 \text{ and } c(\Omega) > 1,\\
1 &  \text{if } k = c(\Omega).
\end{cases} \]
Summing up gives the desired calculation
$$(e',\phi) = 1+ (1-q)^{-1} \sum_{y \in \mathfrak o/\mathfrak p^{c(\Omega)} : v(y) = c(\Omega)-1}
\Omega(1+y\xi_0)^{-1} = 
\frac 1{1-q^{-1}},
$$
since here $c(\Omega) \ge 2$.
\end{proof}
\section{A central-value formula}
\label{klm-lval-sec}

In this section we will work globally.  Specifically, let $L/F$ be a quadratic
extension of number fields,
$\A$ the ad\`eles of $F$ and $\A_L$ the ad\`eles of $L$.  
Let $\Delta$ and $\Delta_L$ be the absolute values of the discriminants of
$F$ and $L$, and let $\eta=\eta_{L/F}$ be the quadratic id\`ele class character
associated to $L/F$ via class field theory.

Let $G=\GL(2)/F$, $\pi$ a cuspidal automorphic representation of $G(\A)$ with
trivial central character, and $\Omega$
a unitary character of $\A_L^\times/L^\times \A^\times$. 
Assume the sign of the functional equation $\epsilon(1/2, \pi_L \otimes
\Omega)=1$, where $\pi_L$ is the base change of $\pi$
to $L$.  Then by work of Waldspurger \cite{wald}, Tunnell \cite{tunnell}, and
Saito \cite{saito}, one knows that there is a unique
quaternion algebra (possibly the split matrix algebra) $D/F$ in which $L$
embeds, such that $\pi$ has
a Jacquet--Langlands transfer to a representation
$\pi'$ of $D^\times(\A)$ and the local Hom spaces
\[ \mathrm{Hom}_{L_v^\times}(\pi_v', \Omega_v) \ne 0 \]
for all places $v$, and in fact have dimension 1.  Fix this $D$ and $\pi'$, and
write $G'$ for $D^\times$, regarded as an algebraic group over $F$.
 Let $T$ be a torus in
$G'$ whose $F$-points
are isomorphic to $L^\times$, and view $\Omega$ as a character of
$T(\A)/Z(\A)$, where $Z$ is the center of $G'$.

Let $\psi$ be the standard additive character on $\A/F$, i.e., the composition
of the trace map
with the standard additive character on $\A_{\Q}$.
Let $S$ be a finite set of places of $F$ containing all archimedean places, such
that, for all $v \not \in S$,
$\psi$, $\pi$ and $\Omega$ are unramified  and $L$ is not ramified at or above $v$.  

Put on $G'(\A)$ the product of the local Tamagawa measures times
$L^S(2,1_F)$, i.e., take the
local Tamagawa measure $dg_v$ for $v \in S$ and $dg_v$ normalized so that 
$G(\mathfrak o_v) \cong G'(\mathfrak o_v)$ has volume 1 if $v \not \in S$ (see,
e.g., 
\cite{JC}*{Section 2} for the definition of local Tamagawa measures).  Note we
will renormalize our measure on $G'(\A)$ later in Section
\ref{klm-finalform-sec}.

In \cite{JC}, Jacquet and Chen prove a formula for a distribution appearing on
the spectral side of the relative trace formula,
\begin{equation}
 J_{\pi'}(f) = \sum_{\phi} \int_{T(\A)/Z(\A)T(F)} \pi'(f) \phi(t) \Omega(t)^{-1}
\, dt  
\int_{T(\A)/Z(\A)T(F)} \overline{ \phi(t) \Omega(t)^{-1}} \, dt,
\end{equation}
where $\phi$ runs over an orthonormal basis for the space of $\pi'$.  Here
$T(\A)$ and $Z(\A)$ are given the product of local Tamagawa measures,
$T(F)$ has the counting measure, and $dt$ is the quotient measure.

Let $S_\mathrm{inert}$ be the set of finite places $v$ in $S$ such that
$L_v/F_v$ is inert
(ramified or unramified).
For $v \in S_\mathrm{inert}$, as in \eqref{klm-jc-def}, define 
\[ \~J_{\pi'_v}(f_v) =  \int_{G'(F_v)} f_v(g) (\pi'_v(g)e_v', e_v') \, dg_v, \]
where  $e_v'$ is a norm 1 vector such that $\pi_v'(t)e_v' = \Omega_v(t) e_v'$
for all $t \in T(F_v)$.
For $v \in S-S_\mathrm{inert}$, set
\[ \~J_{\pi'_v}(f_v) = \sum_W \int\limits_{F_v^\times} \pi_v'(f_v) W  ( \bmx a &
\\ & 1  \emx ) \Omega ( \bmx a & \\ & 1 \emx )^{-1} \, d^\times
a
\overline{ \int\limits_{F_v^\times} W ( \bmx a & \\ & 1  \emx ) \Omega
( \bmx a & \\ & 1 \emx )^{-1}  \, d^\times a}, \]
where $d^\times a$ is the local Tamagawa measure and
$W$ runs over an orthonormal basis for the local Whittaker model $\mathcal
W(\pi_v,\psi_v)$.

With the above normalizations, the formula of Jacquet and Chen is

\begin{theorem}[\cite{JC}]  \label{JC-thm}
Let $S$ be a set of places containing all infinite places and all places at
which $L$, $\pi$ or
$\Omega$ is ramified.  Let $f=\prod f_v \in C_c^\infty(G'(\mathbb A_F))$ with
$f_v$ the unit element of the Hecke algebra for $v \not \in S$.  Then
\[ J_{\pi'}(f) = \frac 12 \prod_{S} \~ J_{\pi'_{v}}(f_{v}) \prod_{v \in
S_\mathrm{inert}} 2\epsilon(1,\eta_{v},\psi_{v}) L(0,\eta_{v})
\times \frac{L_{S}(1,\eta) L^S(1/2,\pi_L \otimes \Omega)}{L^{S}(1,\pi,Ad)}. \]
\end{theorem}

Note that if $\pi'(f)$ is orthogonal projection onto a 1-dimensional subspace
$\langle \phi \rangle$, then
\begin{equation} \label{klm-Jpif}
 J_{\pi'}(f) = \frac{\left|  \int_{ T(\A)/Z(\A)T(F)} \phi(t) \Omega(t)^{-1} \,
dt \right|^2}{(\phi,\phi)}.
\end{equation}
We wrote this expression so that it is invariant under replacing $\phi$ by a
scalar multiple.
 
\subsection{Choice of test vector} \label{klm-testvec}

To obtain an explicit $L$-value formula, 
we will choose $f = \prod f_v$ so that it picks out a global test vector $\phi =
\otimes \phi_v$ as follows.  

\medskip
First suppose $v$ is a finite place of $F$.  We denote by $\mathfrak o_v$, $\mathfrak o_{L_v}$,
 $\mathfrak p_v$, and $\mathfrak \varpi_v$ what was denoted in previous sections
by these 
 symbols without the subscript $v$ for the local field $F_v$.  
 Since we have assumed that the central
character is trivial, we may work with the
 congruence subgroups
 \[ K_{0,v}(\mathfrak p_v^n) = \left\{ \bmx a & b \\ c & d \emx \in G(\mathfrak
o_v) : c \in \mathfrak p_v^n \right\}. \]

We assume that at any finite $v \in S_{\mathrm{inert}}$ such
that $c(\Omega_v) > 0$, we have $c(\Omega_v) \ge c(\pi_v)$. 
Recall that, if $L_v/F_v$ is split or $0 \leq c(\pi_v) \leq c(\Omega_v)$, then we
can identify $G'(F_v) = G(F_v)$.

For $v \not \in S$, let $f_v$ be the characteristic function of $G(\mathfrak
o_v)$.  Then $\pi_v \cong \pi'_v$,
and $\pi'_v(f_v)$ is orthogonal projection onto the local newvector $\phi_v$.   

Let $v \in S-S_{\mathrm{inert}}$.  Take $g_v \in G(F_v)$ such that $g_v^{-1}
T(F_v) g_v$ is the diagonal 
subgroup of $G(F_v)$.  Let $f_v$ be the characteristic function of the subgroup
of $G(F_v)$ given by
\[ g_v^{-1} \bmx 1 & -\varpi_v^{-c(\Omega_v)} \\ & 1 \emx K_{0,v}(\mathfrak
p_v^{c(\pi_v)})
 \bmx 1 & \varpi_v^{-c(\Omega_v)} \\ & 1 \emx  g_v \]
divided by its volume.  Then $\phi_v$ is the unique (up to scalar multiples) vector in $\pi_v$ fixed by this subgroup.

Consider $v \in S_{\mathrm{inert}}$. 

Suppose $c(\pi_v)=0$ or $c(\Omega_v)=0$. Let $R(\pi'_v)$ be an order in $D(F_v)$ of
reduced discriminant $\mathfrak p_v^{c(\pi_v)}$ such that 
$R(\pi'_v) \cap L_v = \mathfrak o_v + \varpi_v^{c(\Omega_v)} \mathfrak o_{L_v}$
(cf.\ \cite{Gross}*{Proposition 3.4}).  Note $R(\pi'_v)$  is  unique up to
$T(F_v)$-conjugacy.
In this case, we take $f_v$ to be the characteristic function of
$R(\pi_v')^\times$ divided by its volume.  Then $\pi'_v(f_v)$
acts as orthogonal projection onto the local Gross-Prasad test vector $\phi_v$
\cite{GP}, 
except in the case that $c(\pi_v) \ge 2$
and $L_v/F_v$ is ramified.  (Note \cite{GP} also assumes $F_v$ has odd residual
characteristic if $\pi_v$
is supercuspidal because of this restriction in \cite{tunnell}, but this
hypothesis is no longer needed due to
\cite{saito}.)
  When $c(\Omega_v)=0$, $c(\pi_v) \ge 2$, and $L_v/F_v$
is ramified, $\pi'_v(f_v)$ acts as orthogonal projection onto a two-dimensional
space containing a vector $\phi_v$
which satisfies $\pi'_v(t_v)\phi_v=\Omega_v(t_v)\phi_v$ for all $t_v \in T(F_v)$
\cite{GP}*{Remark 2.7}; hence on this space any linear form in
$\Hom(\pi_v, \Omega_v)$ is simply a multiple of the map $\phi_v' \mapsto (\phi_v', \phi_v)$.

If $0 < c(\pi_v) \le c(\Omega_v)$, take $g_v$ so that $g_v^{-1} T(F_v)g_v$ is of the form 
\eqref{TFdefeq}, and let $K_v$ be such that $g_v K_v g_v^{-1}$ is
 the subgroup in \eqref{klm-subgp-def}. 
Let $f_v$ be the characteristic function of $K_v$ divided by its
volume, so
$\pi_v(f_v)$ acts as orthogonal projection onto the line generated by $\phi_v$,
the unique (up to scalar multiples) vector in $\pi_v$
fixed by $K_v$.

\medskip
Lastly, suppose $v$ is an infinite place of $F$.  
Let $K_v$ be a maximal compact subgroup of $G'(F_v)$ whose restriction to
$T(F_v)$
remains maximal compact.  Let $\phi_v$ be a vector of minimal weight such that
$\pi'_v(t_v)\phi_v = \Omega_v(t_v)\phi_v$
for $t_v \in K_v \cap T(F_v)$.  Choose $f_v$ so that $\pi'_v(f_v)$ is orthogonal
projection onto $\langle \phi_v \rangle$.

\medskip
Take $f=\prod f_v$ and $\phi = \otimes \phi_v$, so $\pi(f)$ acts as orthogonal
projection onto a finite-dimensional space $V$ containing $\phi$. 
Local considerations show the toric period integral is vanishing on the
orthogonal complement of $\langle \phi \rangle$ in $V$,
and hence one has \eqref{klm-Jpif}.

\subsection{Archimedean factors}
\label{klm-arch-fact}

Here we recall certain archimedean constants $C_v(L,\pi,\Omega)$ from
\cite{MW}. 
Let $v$ be an infinite place of $F$.  By assumption, $\Omega_v$ is a unitary
character of
$L_v$.

First suppose $F_v=\R$ and $L_v = \R \oplus \R$. Write
\[ \Omega_v(x_1,x_2) = \left|\frac{x_1}{x_2}\right|^{it} \sgn^{m_v}\left(\frac{x_1}{x_2}\right), \]
where $t \in \R$ and $m_v$ is 0 or 1.  If $\pi_v=\mu_v \times \mu_v^{-1}$ is
principal series with Laplacian eigenvalue $\lambda_v$ let $\epsilon_v \in \{ 0,
1 \}$
such that $\mu_v \Omega_v = | \cdot |^r \sgn^{\epsilon_v}$ for some $r$.  Then
we put
\[ C_v(L,\pi,\Omega) = \left( \frac{8\pi^2}{\lambda_v} \right)^{\epsilon_v}. \]
If $\pi$ is discrete series of weight $k_v$, put
\[ C_v(L,\pi,\Omega) = 2^{k_v}. \]

Now suppose $F_v=\R$ and $L_v=\C$.  Write $\Omega_v(z) = (z/\bar z)^{\pm m_v}$
where $m_v \in \frac 12 \Z_{\ge 0}$.
If $\pi_v= \mu_v \times \mu_v^{-1}$ is a principal series where $\mu_v$ is of the
form $|\cdot|^{r_v} \sgn^{\epsilon_v}$, then
\[ C_v(L,\pi,\Omega) =   (2\pi)^{2m_v} \prod_{j=0}^{m_v-1} (\lambda_v +
j(j+1))^{-1}, \]
where $\lambda_v=\frac 14 - r_v^2$.  If $\pi_v$ is discrete series of weight
$k_v$, then
\[ C_v(L,\pi,\Omega) = \frac{1}{\pi B(k_v/2+m_v, k_v/2-m_v)} \]
if $m_v < \frac{k_v-1}2$ and
\[ C_v(L,\pi,\Omega) = \frac{(2\pi)^{2m_v-k_v}k_v!}{m_v!
B(k_v/2+m_v,1-k_v/2+m_v)} \]
if $m_v \ge \frac{k_v-1}2$.  Here $B(x,y)$ denotes the beta function.

Lastly suppose $F_v=\C$ so $L_v = \C \oplus \C$.  Write $\Omega_v$ in the form
\[ \Omega_v(z_1,z_2) =  (z_1 \bar z_1)^{it} \left( \frac{z_1}{\bar z_1}
\right)^{m_v} (z_2 \bar z_2)^{-it} \left( \frac{z_2}{\bar z_2} \right)^{-m_v},  \]
where $t \in \R$ and $m_v \in \frac 12 \Z_{\ge 0}$.  Then $\pi_v$ is principal
series.  Let $k_v$ be its weight, $\lambda_v$ the Laplacian eigenvalue and
$\ell_v = \max(k_v,m_v)$.  Then
\[ C_v(L,\pi,\Omega) = \left( \frac 12 + \ell_v \right) \begin{pmatrix} 2\ell_v
\\ |k_v-m_v| \end{pmatrix} \prod_{j=k_v+1}^{\ell_v} \frac{4\pi^2}{4\lambda_v +
j^2-1}. \]

\subsection{Proof of Theorem \ref{intro-lvalthm}}
\label{klm-finalform-sec}

We consider a measure on $G'(\A)$ which is the product of local
Tamagawa measures.
Write $\Delta = \Delta_{\mathrm{inert}} \Delta_{\mathrm{split}}$, where
$\Delta_{\mathrm{inert}}$ is the part of $\Delta$ coprime to
every place over which $L/F$ splits.
Then note that
\[  \prod_{v \in S_\mathrm{inert}} 2\epsilon(1,\eta_{v},\psi_{v}) L(0,\eta_{v})
= \frac 1{\sqrt{c(\eta)c(\psi)}} \prod_{v \in S_\mathrm{inert}} e(L_v/F_v) =
\sqrt{ \frac{\Delta_{\mathrm{inert}}}{\Delta_L } } \prod_{v \in
S_\mathrm{inert}} e(L_v/F_v). \]

Let $v \in S$ be finite.  The calculations of $\~J_{\pi_v'}(f_v)$ below for when
$L_v/F_v$ is split,
$v$ is infinite, or at most one of $\pi_v$ and $\Omega_v$ is ramified are taken
from \cite{MW}.

Suppose $L_v = F_v \oplus F_v$.  Then
\[ \~J_{\pi'_v}(f_v) = \begin{cases}
q_v^{-c(\Omega_v)} \frac{L(1/2,\pi_{L_v} \otimes
\Omega_v)}{(W_{\pi_v},W_{\pi_v})} & \text{if $\Omega_v$ is unramified}, \\
q_v^{-c(\Omega_v)} \frac{L(1,1_{F_v})^2}{(W_{\pi_v},W_{\pi_v})} & \text{if
$\Omega_v$ is ramified}, 
\end{cases}\]
where $W_{\pi_v}$ is the normalized Whittaker newvector.  Further,
$\vol(\mathfrak o_v^\times) (W_{\pi_v},W_{\pi_v})$ is  equal to
$L(1,\pi_v,Ad)L(1,1_{F_v})/L(2,1_{F_v})$ if $\pi_v$ is unramified,
$L(1,\pi_v,Ad)=L(2,1_{F_v})$ if $c(\pi_v)=1$, and $1$ if $c(\pi_v) > 1$.
Since we are using local Tamagawa measures, the product over all such $v$
of $\vol(\mathfrak o_v^\times)$ is $\sqrt{\Delta_{\mathrm{split}}}$.

Suppose now $L_v/F_v$ is inert.  If $\pi_v'$ is unramified, then
$\~J_{\pi'_v}(f_v)$ is
\[ \frac{q_v^{-c(\Omega_v)}}{e(L_v/F_v)} \frac{L(1/2,\pi_{L_v} \otimes
\Omega_v)L(2,1_{F_v})}{L(1,\pi_v,Ad)} L(1,\eta_v)^{\delta_v}, \] 
where $\delta_v = -1$ if $\Omega_v$ is unramified and $\delta_v = 1$ if
$\Omega_v$ is ramified.
If $\pi_v$ is ramified and $\Omega_v$ is unramified, then $\~J_{\pi'_v}(f_v)=1$.
When both $\pi_v$ and $\Omega_v$ are ramified, $\~J_{\pi'_v}(f_v)$ is calculated
in Proposition
 \ref{klm-Jpif-prop}.

Summing up, if $\pi_v$ is unramified, then, up to factors of the form
$\vol(\mathfrak o_v^\times)$ and $e(L_v/F_v)$, $\~J_{\pi'_v}(f_v)$ is 
\[ q^{-c(\Omega_v)} \frac{L(1/2,\pi_{L_v} \otimes \Omega_v)
L(2,1_{F_v})}{L(1,\pi_v,Ad)} L(1,\eta_v)^{\delta_v}. \]

If $c(\pi_v)=1$, then, up to factors of the form $\vol(\mathfrak o_v^\times)$
and $e(L_v/F_v)$,
 $\~J_{\pi'_v}(f_v)$ is 
 \[ \frac{L(1/2,\pi_{L_v} \otimes \Omega_v)}{L(1, \pi_v,Ad)}  \]
  if $\Omega_v$ is unramified and $L_v/F_v$ is split or unramified;
 $1$ if $\Omega_v$ is unramified and $L_v/F_v$ is ramified;  and
  \[ q^{-c(\Omega_v)} \frac{L(1/2,\pi_{L_v} \otimes \Omega_v)}{L(1,\pi_v,Ad)}
L(1,1_{F_v}) L(1,\eta_v) \]  
  if $\Omega_v$ is ramified.

If $c(\pi_v) \ge 2$, then, up to factors of the form $\vol(\mathfrak
o_v^\times)$ and $e(L_v/F_v)$,
 $\~J_{\pi'_v}(f_v)$ is $1$ if $\Omega_v$ is unramified and
$q^{-c(\Omega_v)}L(1,1_{F_v})L(1,\eta_{v})$ if $\Omega_v$ is ramified.

Now suppose $v|\infty$.  Then from \cite{MW} one has
\[ \~J_{\pi'_v}(f_v) = \frac{C_v(L,\pi,\Omega)}{e(L_v/F_v)} \frac{L(1/2,\pi_{L_v}
\otimes \Omega_v) L(2,1_{F_v})}{L(1,\pi_v,Ad) L(1,\eta_v)}. \]

Combining the above calculations completes the proof of Theorem \ref{intro-lvalthm}.

\begin{remark}
When $S(\pi) \cap S(\Omega) = \emptyset$, Theorem \ref{intro-lvalthm} is exactly the main theorem of
\cite{MW},
though the choice of measure on $G'(\A)$ is slightly different in \cite{MW}.  Our set
$S_0(\pi)$ is denoted by $S'(\pi)$ in that paper.  

As in \cite{MW}, one can  rewrite this formula using the Petersson norm 
$(\phi_\pi,\phi_\pi)$ of the new vector $\phi_\pi \in \pi$ instead of
$L(1,\pi,Ad)$.  The formula in
\cite{MW} is also valid when $\omega_\pi = \eta$, and one could treat that case
here similarly.
The restriction that $\omega_\pi \in \{ 1, \eta \}$ is not inherent in the
method, but is due to
this assumption in \cite{JC}.
\end{remark} 

\begin{remark} For many applications, one would like a formula for the 
 {\em complete} ratio of $L$-values $L(1/2, \pi_L \otimes \Omega)/L(1,\pi, Ad)$. Theorem \ref{intro-lvalthm} of course gives this
when $S_0 = \emptyset$ (e.g., if the conductor $c(\pi)$ of $\pi$ is squarefree
and $\pi$ and $L/F$ have disjoint ramification).  In general, one can of course
multiply both sides
by the appropriate local factors, but then the rest of the formula will depend
on more than just the
ramification of $\pi$ and $\Omega$ together with their infinity types. 
Specifically, 
for $v \in S_1(\pi)$ and $\pi_v = \chi_v \St_v$,
the local factor $L(1/2, \pi_{L_v} \otimes \Omega_v)$ depends on the sign of
$\chi_v$ when
$L_v/F_v$ is ramified.  Similarly, for $v \in S_2(\pi)$, the local factor $L(1,
\pi_v, Ad)$ depends on 
more than just the ramification of $\pi_v$. 
\end{remark}

\section{An average-value formula}
\label{klm-avg-sec}

In this section, we will prove Theorems \ref{intro-avgval-thm}, \ref{non-van-thm} and \ref{avgval-modp}.
Fix notation as in the first paragraph of Theorem \ref{intro-avgval-thm}.

\subsection{The trace formula} \label{sec71}

Let $D/F$ be the quaternion algebra which is ramified precisely at the infinite
primes and the primes dividing $\mathfrak N_0$.  
Set $G' = D^\times$, $G=\GL(2)/F$, and let $Z$ denote the center of either of
these.  Let $\epsilon$ be an element of the normalizer of $T(F)$ inside $G'(F)$
which does not lie 
in $T(F)$,
so $\epsilon^2 \in Z(F)$ and $D(F)= L \oplus \epsilon L$.  Then we may write an
element of $G'(F)$ in the form
\[ \bmx \alpha & \beta \epsilon \\ \bar \beta & \bar \alpha \emx, \quad \alpha,
\beta \in L. \]
With this representation,
\[ T = \{ \bmx \alpha & 0 \\ 0 & \bar \alpha \emx \}. \]

As in Section \ref{klm-lval-sec}, let $\psi$ be the standard additive character of $\A/F$, and
take the product of the local Tamagawa measures on $T(\A)$, $G'(\A)$, $G(\A)$
and $Z(\A)$.
For a cuspidal automorphic representation $\pi'$ of $G'(\A)$, let JL$(\pi')$
denote its Jacquet--Langlands transfer to $G(\A)$.
Denote by $\mathcal F'(\mathfrak N,2 \mathbf k)$ the set of cuspidal automorphic
representations
$\pi'$ of $G'(\A)$ such that JL$(\pi') \in \mathcal F(\mathfrak N,2\mathbf k)$. 
We call
$\mathfrak N$ the conductor of $\pi'$ and write $c(\pi') = \mathfrak N$.
Subject to assumption \eqref{tameramcond}, we note that our choice of $D$ guarantees
$\mathrm{Hom}_T(\pi',\Omega) \ne 0$ for all $\pi' \in \mathcal F'(\mathfrak N,
2\mathbf k)$.

We now recall Jacquet's relative trace formula for $G'$ from \cite{Jac}.
This is an identity of the form
\begin{equation} \label{avgval-rtf}
 I(f) = J(f),
\end{equation}
where $I(f)$ is a certain geometric distribution, and $J(f)$ is a certain spectral distribution. 
Specifically, let $f = \prod f_v \in C_c^\infty(G'(\A))$.
 The geometric (relative) orbital integrals of
$f$ are defined by
\[ I(0, f) = \int_{T(\A)} f(t) \Omega(t) \, dt \]
\[ I(\infty, f) = \int_{T(\A)} f (t \bmx 0 & \epsilon \\ 1 & 0 \emx ) \Omega(t)
\, dt \]
and
\[ I(b, f) = \int_{T(\A)/Z(\A)} \int_{T(\A)} f ( s \bmx 1 & \epsilon \beta \\
\bar \beta & 1 \emx t ) \Omega(st) \, ds \, dt, \]
where $b = \epsilon N(\beta)$ for $\beta \in L^\times$.  Note this  latter
integral only depends on $b$ and not the choice of a specific $\beta$.
Then the left hand (geometric) side of \eqref{avgval-rtf} is 
\begin{equation}
 I(f) = \vol(T(\A)/Z(\A)T(F)) ( I(0,f) + \delta(\Omega^2) I(\infty, f) ) +
\sum_{b \in \epsilon N(L^\times)} I(b,f),
\end{equation}
where $\delta(\chi)=1$ if $\chi$ is trivial and $\delta(\chi) = 0$ otherwise.  

We now describe $J(f)$, but for simplicity only in the situation that is
relevant for us.  
Namely, for each $v|\infty$, fix an embedding $\iota_v: G'(F_v)
\hookrightarrow \GL_2(\C)$ and let $\pi'_{2k_v}$ be the irreducible
$(2k_v-1)$-dimensional representation of $G'(F_v)$ given by 
$\pi'_{2k_v} = (\mathrm{Sym}^{2k_v-2} \otimes \det^{1-k_v}) \circ \iota_v $. 
Hence
JL$(\pi'_{2k_v})$ is the holomorphic discrete series of weight $2k_v$ on
$G(F_v)$.
The assumption that $|m_v| < k_v$ implies that there is a 1-dimensional subspace of 
$\pi'_{2k_v}$ consisting of vectors $w_v$ such that $\pi'_{2k_v}(t) w_v =
\Omega_v(t) w_v$ for all
$t \in T(F_v)$.  Fix such a vector $w_v \in \pi'_{2k_v}$ which satisfies $(w_v, w_v)
= 1$.  
For all $v| \infty$, we may take $f_v \in C_c^\infty(G'(\R))$ as in Section
\ref{klm-testvec}, so that
\[ \int_{Z(F_v)} f_v(zg) \, dz =\frac{2k_v-1}{\vol(G'(F_v)/Z(F_v))}
\overline{(\pi'_{2k_v}(g)w_v, w_v)} \]
(cf.\ \cite{FW}*{Lemma 3.4}).

For a cuspidal automorphic representation $\pi'$ of $G'(\A)/Z(\A)$, we consider
the spectral
distribution
\[ J_{\pi'}(f) = \sum_{\phi} P_D(\pi'(f)\phi) \overline{P_D(\phi)}, \]
where $\phi$ runs over an orthonormal basis for $\pi'$ and $P_D$ is defined as
in \eqref{klm-PDdef}.
In general, the spectral side $J(f)$ of \eqref{avgval-rtf} is a sum over all $\pi'$ of
$J_{\pi'}(f)$ plus a non-cuspidal contribution.  However, things simplify greatly for our choice
of $f$.

We already specified $f_v$ for $v | \infty$.  Now let $v < \infty$ and put $m_v
= c(\Omega_v)$.
  For such a $v$, as in 
Section \ref{klm-testvec}, we will take $f_v$ to be the characteristic function
of $R_v^\times$
divided by its volume, for an order $R_v$ of $G'(F_v)$ chosen as follows.   
If $v \nmid \mathfrak N$, then $G'(F_v) \cong G(F_v)$ and we take $R_v$
to be a maximal order optimally containing $\mathfrak o_v + \varpi_v^{m_v} \mathfrak o_{L_v}$.
If $v | \mathfrak N_0$, then $G'(F_v)$ is not split and we take $R_v$ to be a
maximal order
containing $\mathfrak o_{L_v}$.  If $v | \mathfrak N_1$, then $G'(F_v) \cong
G(F_v)$ and, 
at least when $v$ is odd, we can take
\begin{equation}\label{Rv-defn}
R_v = \{ \bmx \alpha & \beta \epsilon_v \\ \bar \beta & \bar \alpha \emx : 
tr(\alpha), tr(\beta) \in \mathfrak o_v, \,
\alpha, \beta \in \mathfrak p_v^{1-m_v} \mathfrak o_{L_v},\,
\alpha - \beta \in \mathfrak o_v+ \mathfrak p_v^{m_v} \mathfrak o_{L_v}  \}. 
\end{equation}
Note that for each $v \nmid \mathfrak N_0$, this agrees with our choice of test
functions in
Section \ref{klm-testvec}.  
The difference of the present choice of $f_v$ for 
$v | \mathfrak N_0$ 
is simply out of convenience so we can directly apply local calculations from
\cite{FW}.
What is important is that one still has $\pi_v(f_v)$ being orthogonal projection
onto
our local test vector for $v | \mathfrak N_0$ (cf. \cite{FW}*{Lemma 3.3}).

Consequently, for this $f$, assuming $k_v > 1$ for some $v | \infty$,
the spectral side of \eqref{avgval-rtf} is given by
\begin{equation} \label{avgval-rtfspec}
 J(f) = \sum_{\mathfrak N'} \sum_{\pi' \in \mathcal F'(\mathfrak N', 2\mathbf k)}
J_{\pi'}(f),
\end{equation} 
where $\mathfrak N'$ runs over ideals which divide 
$\mathfrak N$ and are divisible by $\mathfrak N_0$.  This is because, for our
choice of $f'$,
$\pi'(f')$ is zero unless $\pi'$ is of weight $2\mathbf k$ and has conductor
dividing $\mathfrak N$.
Furthermore, by our choice of $D$, $J_{\pi'}(f)$ vanishes for local reasons if
the conductor of 
$\pi'$ is not divisible by $\mathfrak N_0$ (cf.\ \cite{FW}*{Lemmas 3.6 and 3.7}).
(The avoidance of the case $k_v=1$ for all $v | \infty$ is purely for
simplicity, 
for in this case there is also contribution from the residual spectrum, which
one would treat as in
\cite{FW}.)

\subsection{Spectral calculations}
Here we compute the spectral expansion \eqref{avgval-rtfspec}. For $\pi' \in \mathcal F'(\mathfrak N,2 \mathbf k)$,
 we see that  $J_{\pi'}(f)=|P_D(\phi)|^2/(\phi,\phi)$.
Hence Theorem
\ref{intro-lvalthm} implies
\begin{align} \label{avgval-Jf}
J_{\pi'}(f) &= \frac 12 \sqrt{\frac{\Delta}{c(\Omega) \Delta_L}}
 L^{S(\mathfrak N_0)}(2,1_F) L_{S(\mathfrak N)}(1,\eta) L_{S(\mathfrak C_0)}(1,\eta)^2 \nonumber \\
 & \qquad \qquad 
\prod_{v | \infty}  \frac{2k_v-1}{\pi} \begin{pmatrix} 2k_v-2 \\ k_v-m_v-1
\end{pmatrix}
 \frac{ L(1/2,\pi_L \otimes \Omega) }{L(1,\pi,Ad) }.
\end{align}

We now need to extend this equality to general $\pi' \in \mathcal F'(\mathfrak
N', 2
\mathbf k)$ where $\mathfrak N'$ divides $\mathfrak N$ and is divisible by
$\mathfrak N_0$.  
For $v |  (\mathfrak N')^{-1} \mathfrak N$,
let $R_v'$ be the maximal order of $G'(F_v) \cong G(F_v)$ which contains $R_v$
given by (\ref{Rv-defn}). Let $f' = \prod f'_v$, where $f'_v = f_v$ if 
$v \nmid (\mathfrak N')^{-1} \mathfrak N$,
and $f'_v$ is the characteristic function of $(R_v')^\times$ divided by its
volume
if $v |  (\mathfrak N')^{-1} \mathfrak N$.
Now, $f'$ agrees with our choice of test function for $\pi'$ in Section
\ref{klm-testvec}, and Theorem \ref{intro-lvalthm} gives
\begin{multline} \label{avgval-Jfprime}
J_{\pi'}(f') = \frac 12 \sqrt{\frac{\Delta}{c(\Omega) \Delta_L}}
 L^{S(\mathfrak N_0)}(2,1_F) L_{S(\mathfrak N)}(1,\eta) L_{S(\mathfrak C_0)}(1,\eta)^2 \\
\times \prod_{v | (\mathfrak N')^{-1} \mathfrak N } L(1, \eta_v)
 \prod_{v | \infty}  \frac{2k_v-1}{\pi} \begin{pmatrix} 2k_v-2 \\
k_v-m_v-1 \end{pmatrix}
 \frac{ L(1/2,\pi_L \otimes \Omega) }{L(1,\pi,Ad) }.
\end{multline}
From Theorem \ref{JC-thm}, we see that
\[ J_{\pi'}(f) = J_{\pi'}(f') \prod_{v |  (\mathfrak N')^{-1} \mathfrak N} 
\frac{\~ J_{\pi'_v}(f_v)}{\~ J_{\pi'_v}(f'_v)}. \]
From \cite{MW}*{Section 2.2.4}, we know 
\[ \~ J_{\pi'_v}(f'_v) = q_v^{-m_v} L(2,1_{F_v}) L(1,\eta_v) \frac
1{L(1,\pi_v,Ad)}, \]
so it remains to compute $\~ J_{\pi'_v}(f_v)$.  Here $v |  (\mathfrak N')^{-1}
\mathfrak N \supset \mathfrak N_1$,
so $\pi'_v = \pi_v$ is unramified and $m_v = 1$.  
We may write $\pi_v = \chi \times \chi^{-1}$ where 
$\chi=\chi_v$ is an unramified (unitary) character of $F_v^\times$.

Note $\pi_v(f_v)$ is orthogonal projection onto $\pi_v^{R_v^\times}$.  
Embedding $L_v$ in $M_2(F_v)$ as in \eqref{TFdefeq}, we may write 
\[ R_v^\times = K_v:= \bmx \mathfrak o_v^\times & \mathfrak p_v^{m_v} \\
 \mathfrak p_v^{1-m_v} & \mathfrak o_v^\times \emx = 
 \bmx \mathfrak o_v^\times & \mathfrak p_v \\
 \mathfrak o_v & \mathfrak o_v^\times \emx. \]
Note
\[ K_v = 
h_v\GL_2(\mathfrak o_v) h_v^{-1} \cap h_v \bmx \varpi_v^{-1} & \\ & 1 \emx  
\GL_2(\mathfrak o_v) \bmx \varpi_v & \\ & 1 \emx  h_v^{-1}, \quad \text{where } h_v = \bmx \varpi_v & \\ & 1 \emx. \]
 So if we put $\phi_0$ to be a newvector in $\pi_v$ and
$\phi_0' = \pi_v (h_v^{-1}) \phi_0$, then
\[ \pi_v^{K_v} = \langle \pi_v(h_v) \phi_0, \, \pi_v(h_v) \phi_0' \rangle. \]
Normalize $\phi_0$ so that $(\phi_0, \phi_0) = 1$.

\begin{lemma} We have
\begin{equation} \label{klm-avg-ip}
 (\phi_0, \phi_0') = (\phi_0', \phi_0) 
 = \frac {q_v^{-1/2}}{1+q_v^{-1}} \left( \chi(\varpi_v) + \chi(\varpi_v)^{-1}
\right).
\end{equation}
\end{lemma}

\begin{proof}
In the induced model for $\pi_v$, we have
\[ (\phi_0', \phi_0) = (\pi (h_v) \phi_0, \phi_0)
 =  \int_{\GL_2(\mathfrak o_v)} \phi_0 ( k h_v ) \, dk.
\]
Use the fact that the subgroup $K_v$ of $\GL_2(\OF_v)$ is normalized by $h_v$  to get the lemma.
\end{proof}

\begin{lemma} For $v | (\mathfrak N')^{-1} \mathfrak N$, 
so $\pi_v$ is unramified and $m_v=1$, we have
$\~J_{\pi_v}(f_v) =  q_v^{-1}.$
\end{lemma}

\begin{proof}
Write $\phi_1 =\pi_v(h_v) \phi_0$,  and 
\[ \phi_2 =  \frac{ \phi_0 - (\phi_0,\phi_1) \phi_1 } {(1-(\phi_0,
\phi_0')^2)^{1/2} }
= \left( \frac{L(1,\pi_v,Ad)(1+q_v^{-1})}{L(2,1_{F_v})} \right)^{1/2}  \left(
\phi_0 - (\phi_0,\phi_1) \phi_1 \right), \]
so that $\{ \phi_1, \phi_2 \}$ forms an orthonormal
basis for $\pi_v^{K_v}$.  As in Section \ref{klm-sd-sec}, put 
\[ e' = \sum_{t \in T_v/(T_v \cap Z_v K_v)} \Omega_v^{-1}(t) \pi_v(t) \phi_1, \]
so
\[ \~J_{\pi_v}(f_v) = \vol(K_v)^{-1} \int_{K_v} \frac{(\pi_v(k)e',e')}{(e',e')}
\, dk
= \frac{1}{|T_v/(T_v \cap Z_v K_v)|(\phi_1,e')} (\pi_v(f_v)e',e').  \]
Since $\pi_v(f_v)e' = (e',\phi_1)\phi_1 + (e',\phi_2)\phi_2$, we have
\[ \~J_{\pi_v}(f_v) = \frac{1}{|T_v/(T_v \cap Z_v K_v)|(\phi_1,e')} 
\left( (e',\phi_1)(\phi_1,e') + (e',\phi_2)(\phi_2,e') \right). \]
From \cite{MW}*{Section 2.2.4}, where $(\phi_1,e')$
is denoted $\frac{\langle v_0, e''_{T} \rangle}{\langle v_0, v_0 \rangle}$, we know $(\phi_1,e') =
\frac{L(2,1_{F_v})}{L(1,\pi_v,Ad)}$. Thus
\begin{equation}\label{klm-73a}
 \~J_{\pi_v}(f_v) = \frac{1}{|T_v/(T_v \cap Z_v K_v)|} \left(
\frac{L(2,1_{F_v})}{L(1,\pi_v,Ad)}  + 
 \frac{L(1,\pi_v,Ad)}{L(2,1_{F_v})}  |(\phi_2,e')|^2 \right).
\end{equation}
Note
\begin{equation}
 (\phi_2,e') 
= \left( \frac{L(1,\pi_v,Ad)(1+q_v^{-1})}{L(2,1_{F_v})} \right)^{1/2}
((\phi_0,e') - (\phi_0',\phi_0) \frac{L(2,1_{F_v})}{L(1,\pi_v,Ad)} ). 
\label{klm-73b}
\end{equation}
Hence it suffices to compute
\[ (\phi_0,e') = \sum_{t \in T_v/(T_v \cap Z_v K_v)} \Omega_v(t) (\phi_0,
\pi_v(t) \phi_1)
=  \sum_{t \in T_v/(T_v \cap Z_v K_v)} \Omega_v^{-1}(t) (\pi_v(h_v^{-1}t
h_v)\phi_0',  \phi_0). \]
Using the set of representatives for $T_v/(T_v \cap Z_v K_v)$  given in Lemma
\ref{klm-rep-lem}, we see
\begin{multline} \label{klm-avg-phi2e}
 (\phi_0,e') = \sum_{x \in \mathfrak p_v/\mathfrak p_v }
\Omega_v^{-1}(x+\xi_{0,v}) 
 ( \pi_v ( \bmx \varpi_v^{-1} x & \mathbf c \varpi_v^{-1} \\
 -\mathbf a   & x- \mathbf b \emx ) \phi_0, \phi_0 ) \\ 
+  \sum_{y \in \mathfrak o_v/\mathfrak p_v} \Omega_v^{-1}(1+y\xi_{0,v}) 
( \pi_v (  \bmx \varpi_v^{-1}  & \mathbf c \varpi_v^{-1} y \\
 -\mathbf a y & 1- \mathbf by \emx ) \phi_0, \phi_0 ).
\end{multline}

Let $\phi$ (resp.\ $\phi'$) be the realization of the unit newvector $\phi_0$
(resp.\ $\phi_0'$) in the Kirillov model for $\pi_v$ with respect to
an unramified $\psi_v$.  Recall $\phi(z) = 0$ unless $z \in \mathfrak o_v$.
Recall also the action of the standard Borel on the Kirillov model is given by
\[ (\pi_v \bmx a & x \\ & d \emx \phi)(z) = \psi_v(xz/d) \phi(az/d). \]

Since $v$ is odd and unramified, we may assume $\mathbf b = 0$ and $\mathbf a$
is a unit.
For the $x=0$ term in \eqref{klm-avg-phi2e}, note
\[ \pi_v \bmx & \mathbf c \varpi_v^{-1} \\ -\mathbf a  &  \emx \phi(z)
= \pi_v \bmx \mathbf c \varpi_v^{-1} & \\ & \mathbf a \emx \phi(z)
= \phi(\varpi_v^{-1} z) = \phi'(z). \]
For $y \in \mathfrak o_v$, we have
\begin{align*}
  \pi_v (  \bmx \varpi_v^{-1}  & \mathbf c \varpi_v^{-1} y \\
 -\mathbf a  y & 1 \emx ) \phi(z) 
 &= \pi_v ( \bmx \varpi^{-1}(1+ \mathbf a \mathbf c y^2) & \mathbf c
\varpi_v^{-1} y \\ & 1 \emx 
 \bmx 1 & \\ -\mathbf a  y & 1 \emx) \phi(z) \\
 &= \psi_v(\mathbf c \varpi_v^{-1} y z) \phi(\varpi_v^{-1} z) =
\phi(\varpi_v^{-1} z) = \phi'(z).
\end{align*}
In the last line, we used the facts that $1+ \mathbf a \mathbf c y^2 \in
\mathfrak o_v^\times$,
 $\psi$ is unramified, and $\phi$ vanishes outside of $\mathfrak o_v$.  
 Hence \eqref{klm-avg-phi2e} becomes
\begin{equation}
(\phi_0, e') =  \left( \Omega_v^{-1}(\xi_{0,v})  + 
\sum_{y \in \mathfrak o_v/\mathfrak p_v} \Omega_v^{-1}(1+y \xi_{0,v}) \right)
(\phi_0',\phi_0) = 0,
\end{equation}
as this character sum is zero.
Combining \eqref{klm-avg-ip}, \eqref{klm-73a} and \eqref{klm-73b} gives
\begin{align*}
 \~J_{\pi_v}(f_v) &= \frac{1}{|T_v/(T_v \cap Z_v K_v)|} \left(
\frac{L(2,1_{F_v})}{L(1,\pi_v,Ad)} + \frac{q_v^{-1}}{1+q_v^{-1}}
(\chi(\varpi_v)+\chi(\varpi_v)^{-1})^2 \right) \\
  &=  \frac{1+q_v^{-1}}{|T_v/(T_v \cap Z_v K_v)|}.
\end{align*}
This, with Lemma \ref{klm-rep-lem}, gives the result.
\end{proof}

Hence 
\begin{equation}
\frac{\~ J_{\pi'_v}(f_v)}{\~ J_{\pi'_v}(f'_v)} =
\frac{L(1,\pi_v,Ad)}{L(2,1_{F_v})L(1,\eta_v)},
\end{equation}
for $v | (\mathfrak N')^{-1} \mathfrak N$, which yields
\begin{multline*} 
J(f) = \frac 12 \sqrt{\frac{\Delta}{c(\Omega) \Delta_L}}
 L^{S(\mathfrak N_0)}(2,1_F) L_{S(\mathfrak N)}(1,\eta) L_{S(\mathfrak C_0)}(1,\eta)^2
 \prod_{v | \infty}  \frac{2k_v-1}{\pi} \begin{pmatrix} 2k_v-2 \\ k_v-m_v-1 \end{pmatrix} 
 \\
\times \sum_{\mathfrak N'} \sum_{\pi \in \mathcal F(\mathfrak N', 2\mathbf k)}
\left( \prod_{v | (\mathfrak N')^{-1} \mathfrak N}
\frac{L(1,\pi_v,Ad)}{L(2,1_{F_v})} \right)
 \frac{ L(1/2,\pi_L \otimes \Omega) }{L(1,\pi,Ad) }.
\end{multline*}
Here $\mathfrak N'$ runs over all divisors of $\mathfrak N$ which are divisible
by $\mathfrak N_0$.
Writing
\[  \prod_{v | (\mathfrak N')^{-1} \mathfrak N}
\frac{L(1,\pi_v,Ad)}{L(2, 1_{F_v})} \cdot
\frac 1{L(1,\pi,Ad)} = \prod_{v | \mathfrak N'}
\frac{L(2, 1_{F_v})}{L(1,\pi_v,Ad)} \cdot 
\frac 1{L_{S(\mathfrak N)}(2,1_F) L^{S(\mathfrak N)}(1,\pi,Ad)} \]
and observing $L(1,\pi_v,Ad) = L(2,1_{F_v})$ for $v |
\mathfrak N'$ and 
$\pi \in \mathcal F(\mathfrak N',\mathbf k)$ gives
\begin{multline} 
J(f) = \frac 12 \sqrt{\frac{\Delta}{c(\Omega) \Delta_L}}
\frac{L^{S(\mathfrak N_0)}(2, 1_F)}{L_{S(\mathfrak N)}(1,1_F)} L_{S(\mathfrak C_0)}(1,\eta)^2
\prod_{v | \infty}  \frac{2k_v-1}{\pi} \begin{pmatrix} 2k_v-2 \\ k_v-m_v-1
\end{pmatrix}
 \\
\times \sum_{\mathfrak N'} 
 \sum_{\pi \in \mathcal F(\mathfrak N',2 \mathbf k)} 
 \frac{ L(1/2,\pi_L \otimes \Omega) }{L^{S(\mathfrak N)}(1,\pi,Ad) }.
\end{multline}

\subsection{Geometric calculations}

We will now obtain our average value formula from the trace formula
\eqref{avgval-rtf}
and spectral calculation \eqref{avgval-Jf} by computing the geometric side
$I(f)$.  Most of 
the calculations we need are done in \cite{FW}, with the proviso that our choice
of test functions 
$f_v$ ($v \nmid \mathfrak N_1$) are essentially constant multiples of those
therein (the
test functions in \cite{FW} also come ``pre-integrated over the center'').  

\begin{lemma} Let $b \in \epsilon N(L^\times)$.
We have the following vanishing of local orbital integrals.
\begin{enumerate}
\item If $v | \mathfrak N_0$, then $I(\infty, f_v)=0$.

\item If $v | \mathfrak N_0$ and $b \not \in \mathfrak p_v$, then $I(b,f_v)
= 0$.

\item If $v \nmid \mathfrak N$ is finite and 
$v(1-b) > v(\mathfrak d_{L/F} c(\Omega))$, then $I(b,f_v) = 0$.

\item If $v | \mathfrak N_1$ and $v(1-b) > v(c(\Omega))-2$, then $I(b,f_v) =
0$.
\end{enumerate}
\end{lemma}

\begin{proof}
The first three results are directly from \cite{FW}*{Lemmas 4.2, 4.10 and 4.11}.
 So suppose
$v | \mathfrak N_1$ is odd and write $b = \epsilon N(\beta)$ for some $\beta \in
L^\times$.
For $I(b,f_v)$ to be nonzero we need that, for some
$\alpha \in L_v^\times$ and $u \in L_v^1$,
\[ \bmx \alpha & \\ & \bar \alpha \emx  \bmx 1 & \epsilon \beta u \\  \overline{\beta u} &
1 \emx \in
R_v^\times, \]
i.e., 
\[  N(\alpha) (1-b) \in \mathfrak o_v^\times, \quad tr(\alpha) \in \mathfrak
o_v^\times,
 \quad \alpha \in \mathfrak p_v^{1-m_v} \mathfrak o_{L_v}, \quad \alpha(1-\beta u)
\in \mathfrak o_v + \mathfrak p_v^{m_v} \mathfrak o_{L_v}. \]
Note this implies $v(1-b) = - v(N(\alpha)) \le 2m_v-2$.  Hence if $v(1-b)
\ge 2m_v-1$, then $I(b,f_v)=0$.
\end{proof}

\begin{proposition} If $|\mathfrak N_0 |
> d_{L/F} ( |\mathfrak C|/|\mathfrak N_1|)^{h_F}$, then $I(f) = 2L(1,\eta) I(0,
f)$ and
\[ I(0, f) = \frac{\Delta^2 |\mathfrak N|}{\sqrt{c(\Omega)\Delta_L}} 
\frac{L_{S(\mathfrak C_0)}(1,\eta)}{L_{S(\mathfrak N_0)}(1,1_F)} L(2,1_F) 
\prod_{v | \infty} \frac{2k_v-1}{2\pi}. \]
\end{proposition}

\begin{proof}
This argument is adapted from the proof of \cite{FW}*{Lemma 4.21}.
By the first part of the previous lemma, we know the global orbital integral
$I(\infty,f) = 0$. Arguing as in the proof of \cite{FW}*{Lemma 4.21}, we see that, if $|\mathfrak N_0| > d_{L/F} | \mathfrak N_1^{-2}
\mathfrak C |^{h_F}$, then $I(b, f) = 0$ for all $b$.

Next we compute $I(0,f)$. For $v \nmid \mathfrak N_1$,
we recall the following calculations from \cite{FW}*{Section 4.1} (cf.\
\cite{JC}*{Section 2},
\cite{FW}*{Section 2.1} and \cite{FW}*{Proof of Proposition 4.20} for necessary
facts about
local Tamagawa measures).  Due to the difference in our definition of test functions, 
our local orbital integrals $I(0,f_v)$ (for $v \nmid \mathfrak N_1$) 
 will be $\vol(Z_v \cap R_v^\times) / \vol(R_v^\times)$ 
 (resp.\ $(2k_v-1)/\vol(G'(F_v)/Z(F_v))$) times those in \cite{FW} for
 finite (resp.\ infinite) $v$.
  
  For $v | \mathfrak N_0$,
\[ I(0,f_v) = \vol(\mathfrak o_{L_v}^\times/\mathfrak o_v^\times )
\vol(Z_v \cap R_v^\times)/\vol(R_v^\times)
= (q_v-1)L(2,1_{F_v})  \vol(\mathfrak o_{L_v}^\times)/\vol(\mathfrak
o_v^\times)^4, \]
since $\vol(R_v^\times) = L(2,1_{F_v})^{-1}(q_v-1)^{-1} \vol(\mathfrak
o_v^\times)^4$
and $R_v^\times \cap Z_v = \mathfrak o_v^\times$.

For a finite $v \nmid \mathfrak N$,
we have $\vol(R_v^\times) = L(2,1_{F_v})^{-1} \vol(\mathfrak o_v^\times)^4$ and
\[ I(0,f_v) = 
\begin{cases}
\vol(\mathfrak o_{L_v}^\times/\mathfrak o_v^\times)
\vol(Z_v \cap R_v^\times)/\vol(R_v^\times) = L(2,1_{F_v}) \vol(\mathfrak
o_{L_v}^\times)/
\vol(\mathfrak o_v^\times)^4 & m_v = 0, \\
q^{-m_v} L(1,\eta_v) \vol(\mathfrak o_{L_v}^\times)/\vol(R_v^\times)
= q^{-m_v} L(1,\eta_v) L(2,1_{F_v}) \vol(\mathfrak
o_{L_v}^\times)/\vol(\mathfrak o_v^\times)^4 & m_v > 0.
\end{cases} \]

For $v | \infty$,
\[ I(0,f_v) = \vol(F_v^\times \bs L_v^\times) 
\frac{2k_v-1}{\vol(G'(F_v)/Z(F_v))}=
\frac{2k_v-1}{2\pi^2}. \]

Now, for $v | \mathfrak N_1$, our description of $R_v$ readily implies
\[ I(0,f_v) = \vol(\mathfrak o_v^\times(1+\mathfrak p_v^{m_v} \mathfrak
o_{L_v}))
/\vol(R_v^\times)
= q^{-m_v} L(1,\eta_v) \vol(\mathfrak o_{L_v}^\times) / \vol(R_v)^\times. \]
A simple calculation gives $\vol(R_v^\times) = q_v^{-1}\vol(\mathfrak
o^\times)^4 / 
L(1,1_{F_v})$.  Hence when $v | \mathfrak N_1$, we have
\[ I(0,f_v) =  q_v^{1-m_v} L(2,1_{F_v}) \vol(\mathfrak
o_L^\times)/\vol(\mathfrak o^\times)^4. \]
Putting together the nonarchimedean calculations gives
\[ \prod_{v < \infty} I(0,f_v) = \frac{|\mathfrak N|}{\sqrt{c(\Omega)}}
L_{\mathrm{fin}}(2,1_F) 
\prod_{v | \mathfrak N_0} L(1,1_{F_v})^{-1} \prod_{v | \mathfrak C_0}
L(1,\eta_v)
\prod_{v < \infty} \frac{\vol(\mathfrak o_{L_v}^\times)}{\vol(\mathfrak
o^\times)^4}. \]
Noting that $\prod_{v < \infty} \vol(\mathfrak o^\times) = \Delta^{-1/2}$ (and
similarly
over $L$), we have
\[ \prod_{v < \infty} I(0,f_v) = \frac{\Delta^2 |\mathfrak
N|}{\sqrt{c(\Omega)\Delta_L}} 
L_{\mathrm{fin}}(2,1_F) 
\prod_{v | \mathfrak N_0} L(1,1_{F_v})^{-1} \prod_{v | \mathfrak C_0}
L(1,\eta_v). \]
Recalling that $L(2,1_F) = L_{\mathrm{fin}}(2,1_F) /\pi^d$, we see
\[ I(0, f) = \frac{\Delta^2 |\mathfrak N|}{\sqrt{c(\Omega)\Delta_L}} L(2,1_F) 
\prod_{v | \mathfrak N_0} L(1,1_{F_v})^{-1} \prod_{v | \mathfrak C_0}
L(1,\eta_v)
\prod_{v | \infty} \frac{2k_v-1}{2\pi}. \]
\end{proof}

\subsection{Proofs}

{\bf Proof of Theorem \ref{intro-avgval-thm}:} The result immediately follows from our above calculations of
both sides of the equality $J(f) = I(f) = 2L(1,\eta) I(0,f)$. \vskip 0.1in

\noindent {\bf Proof of Theorem \ref{non-van-thm}:} Suppose $\mathfrak N_1$ contains exactly one prime
$\mathfrak p$.
Put
\[ \Sigma_{\mathfrak N}(\mathfrak N') = 
 \prod_{v | \infty}  \begin{pmatrix} 2k_v-2 \\ k_v-m_v-1 \end{pmatrix}
\sum_{\pi \in \mathcal F(\mathfrak N', 2\mathbf k)} 
 \frac{ L(1/2,\pi_L \otimes \Omega) }{L^{S(\mathfrak N)}(1,\pi,Ad) }. \]
Then Theorem \ref{intro-avgval-thm} reads
\begin{equation} 
 \Sigma_{\mathfrak N}(\mathfrak N_0) + \Sigma_{\mathfrak N}(\mathfrak N)
=  2^{2-d} \Delta^{3/2} |\mathfrak N| 
L(1,1_{F_{\mathfrak p}}) L_{S(\mathfrak N_0)}(2, 1_F) L^{S(\mathfrak C_0)}(1,\eta) .
\end{equation}
Applying our average value formula when $\mathfrak N = \mathfrak N_0$, we also see
\begin{equation} 
\Sigma_{\mathfrak N_0}(\mathfrak N_0)
=  2^{2-d} \Delta^{3/2} |\mathfrak N_0|  L_{S(\mathfrak N_0)}(2, 1_F)
L^{S(\mathfrak C_0)}(1,\eta)
\end{equation}
if $|\mathfrak N_0| > d_{L/F} |\mathfrak C|^{h_F}$.  (This is precisely
\cite{FW}*{Theorem 1.1}.)
For $\pi_{\mathfrak p}$ unramified, we have
\[ L(1,1_{F_{\mathfrak p}}) \frac 1{1+2q_{\mathfrak p}^{-1}+q_{\mathfrak
p}^{-2}} \le 
L(1,\pi_{\mathfrak p},Ad) \le 
L(1,1_{F_{\mathfrak p}}) \frac 1{1-2q_{\mathfrak p}^{-1}+q_{\mathfrak p}^{-2}},
\]
which implies
\begin{equation} L(1,1_{F_{\mathfrak p}}) \frac 1{(1+q_{\mathfrak p}^{-1})^2}
\Sigma_{\mathfrak N_0}(\mathfrak N_0) 
\le \Sigma_{\mathfrak N}(\mathfrak N_0) \le 
L(1,1_{F_{\mathfrak p}}) \frac 1{(1-q_{\mathfrak p}^{-1})^2} \Sigma_{\mathfrak
N_0}(\mathfrak N_0).
\end{equation}
Combining the (in)equalities above gives
\[ \Sigma_{\mathfrak N}(\mathfrak N) \le 2^{2-d} \Delta^{3/2} |\mathfrak N_0| 
L(1,1_{F_{\mathfrak p}}) L_{S(\mathfrak N_0)}(2, 1_F) L^{S(\mathfrak C_0)}(1,\eta)
\left( |\mathfrak p| - \frac 1{1+2|\mathfrak p|^{-1}+|\mathfrak p|^{-2}}
\right), \]
and a similar lower bound, which are precisely the bounds asserted in Theorem
\ref{non-van-thm}.
 
To get an asymptotic, we use a special case of \cite{FW}*{Theorem 1.2}, which is an
asymptotic
for
\[ \sum_{\pi \in \mathcal F(\mathfrak N_0, 2 \mathbf k)} \frac{L^{\mathfrak p}(1/2, \pi_L
\otimes \Omega)}{L^{\mathfrak p}(1,\pi,Ad)} = 
 \prod_{v | \infty}  \begin{pmatrix} 2k_v-2 \\ k_v-m_v-1 \end{pmatrix}^{-1}
\frac{\Sigma_{\mathfrak N}(\mathfrak N_0)}{L_{S(\mathfrak N_0)}(2, 1_F)}\]
as $|\mathfrak N_0| \to \infty$.
(Note $L^{\mathfrak p}(1/2, \pi_L \otimes \Omega) = L(1/2, \pi_L \otimes \Omega)$ since
$\Omega$ is ramified at $\mathfrak p$.)
Specifically, \cite{FW}*{Theorem 1.2} tells us
\begin{equation} 
\Sigma_{\mathfrak N}(\mathfrak N_0) \sim
2^{2-d} \Delta^{3/2} |\mathfrak N_0| L(1,1_{F_{\mathfrak p}}) 
L_{S(\mathfrak N_0)}(2, 1_F) L^{S(\mathfrak C_0)}(1,\eta) 
\end{equation}
as $|\mathfrak N_0| \to \infty$ along a sequence of squarefree ideals $\mathfrak
N_0$ coprime to 
$\mathfrak C$ satisfying our parity and ramification assumptions.
Consequently, we have
\[ \Sigma_{\mathfrak N}(\mathfrak N) \sim 2^{2-d} \Delta^{3/2} |\mathfrak N_0| 
L(1,1_{F_{\mathfrak p}}) L_{S(\mathfrak N_0)}(1,\eta) 
L^{S(\mathfrak C_0)}(1,\eta) \left( |\mathfrak p| - 1 \right). \]
This gives the asymptotic asserted in the theorem.

\subsection{Nonvanishing mod $p$}
Let
$\pi \in \mathcal F(\mathfrak N, 2 \mathbf k)$ and $\mathbf f$ be the
corresponding normalized
Hilbert modular newform of weight $2\mathbf k$ and level $\mathfrak N$ over $F$.
As before, $\Omega$ is a unitary character of
$\A_L^\times/L^\times \A_F^\times$ such that, for all $v | \infty$, $\Omega_v(z)
= (z/\bar z)^{\pm m_v}$  with
$0 \le m_v < k_v$.  Put $\mathbf m = (m_1, \ldots, m_d)$.  Then $\Omega$ gives
rise to a 
Hilbert modular form 
$\bf g$ over $F$ of weight $\mathbf m+1 = (m_1+1, \ldots, m_d+1)$ (cf.\
\cite{shimura}*{Sec. 5}).
Assume $m_1 \equiv m_2 \equiv \cdots \equiv m_d \mod 2$.  This implies
$\Omega$ is algebraic, so that the field of rationality $\mathbb Q(\mathbf g) \subset \bar {\mathbb Q}$
\cite{shimura}*{Prop 2.8}.  

Put $k_0 = \max_{v | \infty} k_v$ and $m_0 = \max_{v | \infty} m_v$. 
  Then Shimura \cite{shimura}*{Thm
4.1}
proved
\[ \frac{D(s_0, \mathbf f, \mathbf g)}{\sqrt{\Delta}\pi^{2|\mathbf k|} ( \mathbf
f, \mathbf f )}
 \in \bar{ \mathbb Q}(\mathbf g) = \bar {\mathbb Q} \]
for any $s_0 \in \mathbb Z$ such that $\frac{2k_0+m_0-1}2  < s_0 <
\frac{2k_0+m_0 + 2k_v - m_v}2$
for all $v | \infty$.  Here $D(s, \mathbf f, \mathbf g)$ is the Dirichlet series
defined in \cite{shimura},
$( \mathbf f, \mathbf f)$ is the Petersson norm defined as in \cite{Hida},
and
$| \mathbf k| = \sum_{v | \infty} k_v$.
Assume that $m_0 \equiv 0 \mod 2$.  Then, for $s_0 = \frac{2k_0 + m_0}2$, this
means
\begin{equation} \label{lvalalg-def}
 L^{\alg}(1/2, \pi_L \otimes \Omega) := \frac 1{L(1,\eta)} \frac{L_{\fin}(1/2,
\pi_L \otimes \Omega)}{\sqrt{\Delta}\pi^{2|\mathbf k|} ( \mathbf f, \mathbf f )}
\in \bar{\mathbb Q}.
 \end{equation}
(Note we normalize the algebraic part of the $L$-value in a different way than
other authors.)
Recall the archimedean $L$-factors are given by
\[ L_v(1/2, \pi_L \otimes \Omega) =
(2\pi)^{-2k_v}{4\Gamma(k_v+m_v)\Gamma(k_v-m_v)},
\quad v| \infty. \]
 From \cite{HidT}*{Theorem 7.1} and \cite{Hida}*{(7.2 c)} (cf.\ \cite{GG}*{Theorem 5.16}),
we have
 \begin{equation} \label{ad-pet-eq}
  L(1, \pi, Ad) = \frac{2^{2|\mathbf k|-1}}{ \Delta^2 h_F |\mathfrak N|}
(\mathbf f, \mathbf f). 
 \end{equation} Thus
 \begin{equation*}
  \frac{L(1/2, \pi_L \otimes \Omega)}{L(1,\pi,Ad)} = 
2^{2d+1-4|\mathbf k|} \Delta^{5/2} h_F |\mathfrak N|  L(1,\eta)
 L^{\alg}(1/2, \pi_L \otimes \Omega) 
  \prod_{v | \infty} \Gamma(k_v+m_v)\Gamma(k_v-m_v). 
 \end{equation*}
Hence we can rewrite the average value formula from Theorem \ref{intro-avgval-thm} as
 \begin{equation}
 2^{3d-4|\mathbf k|-1} \Delta h_F \prod_v(2k_v-2)! \sum_{\pi \in \mathcal
F(\mathfrak N, 2\mathbf k)}
 L^{\alg}(1/2, \pi_L \otimes \Omega) =
\frac 1{L_{S(\Omega)}(1,\eta)}.
\end{equation}
This immediately implies Theorem \ref{avgval-modp}.

\end{document}